\documentclass[onefignum,onetabnum,onethmnum]{siamonline220329}

\usepackage{lipsum}
\usepackage{amsfonts}
\usepackage{graphicx}
\usepackage{epstopdf}
\usepackage[noend]{algpseudocode}
\usepackage{amssymb}
\usepackage{amstext}
\usepackage{tabularx}
\usepackage{multirow}
\usepackage{soul}
\usepackage{caption}
\usepackage{amsopn}

\graphicspath{{./Figures/}}

\ifpdf
  \DeclareGraphicsExtensions{.eps,.pdf,.png,.jpg}
\else
  \DeclareGraphicsExtensions{.eps}
\fi

\DeclareMathAlphabet\mathbfcal{OMS}{cmsy}{b}{n}

\usepackage{enumitem}
\setlist[enumerate]{leftmargin=.5in}
\setlist[itemize]{leftmargin=.5in}

\newcommand{\reals}{\mathbb{R}}

\newcommand{\cmgen}[1]{\operatorname{GenCM}_{#1}}
\newcommand{\cm}[1]{\operatorname{CM}_{#1}}

\newcommand{\Gen}{\operatorname{Gen}}

\newcommand{\q}{\mathbb{Q}}

\newcommand{\ideal}[1]{\left\langle #1 \right\rangle}
\newcommand{\grobner}{Gr\"obner}
\newcommand{\ffield}[1]{\operatorname{Frac}{#1}}

\newcommand{\supp}[1]{\operatorname{supp}{#1}}
\newcommand{\res}[3]{\operatorname{Res}(#1, #2, #3)}

\newcommand{\lmat}[1]{\operatorname{\mathbfcal L}_{#1}} 
\newcommand{\amat}[1]{\operatorname{\mathbfcal A}(#1)} 
\newcommand{\smat}[1]{\operatorname{\mathbfcal S}_{#1}} 
\newcommand{\cres}[3]{\operatorname{CRes}(#1, #2, #3)}

\newcommand{\rank}{\operatorname{rank}}

\newsiamremark{mproblem}{Main Problem}
\newsiamremark{example}{Example}
\newsiamremark{problem}{Open Problem}

\headers{Computing Circuit Polynomials in the Algebraic Rigidity Matroid}{G. Mali\'c, and I. Streinu}

\title{
Computing Circuit Polynomials in the Algebraic Rigidity Matroid
\thanks{Submitted to the editors on Aug 1, 2021. This paper extends the conference abstract~\cite{malic:streinu:socg}, where the main result was announced, and includes results from the pre-print \cite{malic:streinu:arxiv:k33:2021}.
\funding{Both authors acknowledge funding from the NSF CCF:1703765 and CCF:2212309 grants to Ileana Streinu.}}}

\author{Goran Mali\'c\thanks{Computer Science Department, Smith College, Northampton, MA, USA 
  (\email{gmalic@smith.edu}, \url{http://www.goranmalic.com}).}
\and Ileana Streinu\thanks{Computer Science Department, Smith College, Northampton, MA, USA  
  (\email{istreinu@smith.edu, streinu@cs.umass.edu}, \url{http://cs.smith.edu/\~istreinu}).}
}

\ifpdf
\hypersetup{
  pdftitle={Computing Circuit Polynomials in the Algebraic Rigidity Matroid},
  pdfauthor={G. Mali\'c, and I. Streinu}
}
\fi

\begin{document}

\maketitle

\begin{abstract}
    We present an algorithm for computing {\em circuit polynomials} in the algebraic rigidity matroid $\amat{\cm{n}}$ associated to the Cayley-Menger ideal $\cm n$ for $n$ points in 2D.  It relies on {\em combinatorial resultants}, a new operation on graphs that captures properties of the Sylvester resultant of two polynomials in this ideal. We show that every rigidity circuit has a {\em construction tree} from $K_4$ graphs based on this operation. Our algorithm performs an {\em algebraic elimination} guided by such a construction tree, and uses classical resultants, factorization and ideal membership. 
To highlight its effectiveness, we implemented the algorithm in Mathematica: it took less than 15 seconds on an example where a Gr\"obner Basis calculation took 5 days and 6 hrs. Additional speed-ups are obtained using non-$K_4$ generators of the Cayley-Menger ideal and simple variations on our main algorithm.
\end{abstract}



\nopunct 

\section{Introduction.}
\label{app:sec:introduction}

The focus of this paper is the following problem straddling combinatorial rigidity and algebraic matroids:

\medskip

\subparagraph{Main Problem.}
{\em Given a rigidity circuit, compute its corresponding circuit polynomial.}

\medskip

Its motivation comes from the following ubiquitous problem in {\em distance geometry}: 

\medskip

\subparagraph{Localization.} A graph together with {\em weights} associated to its edges is given. The goal is to find {\em placements} for its vertices in some Euclidean space (2D, in our case), so that the resulting edge lengths match the given weights. To this purpose we set up a system of quadratic equations with unknowns corresponding to the Cartesian coordinates of the vertices. The possible {\em placements} (or {\em realizations}) are among its (real) solutions and can be found with numerical methods (see e.g.\ \cite{li:homotopy, numericalMethods:sommese:wampler, numericalMethods:bates:sommese:hauenstein:wampler}). A related problem is to look for the possible values of a {\em single unknown distance} corresponding to a {\em non-edge} (a pair of vertices that are not connected by an edge). If we could solve this second problem for a collection of non-edge pairs that, together with the original edges, contain a trilateration,  then one placement for the graph could be obtained afterwards in linearly many steps of quadratic equation solving.

\subparagraph{Rigidity circuits.}
The {\em generic} version of the {\em single unknown distance} problem, where the weights are symbols rather than concrete numbers, is amenable to techniques from Rigidity Theory. In 2D, one can predict whether, generically, the set of solutions for the unique unknown distance will be discrete (if the given  graph is {\em rigid}) or continuous (if the graph is {\em flexible}). We formulate the problem algebraically by using Cayley coordinates $X_n = \{x_{ij}: 1\leq i<j\leq n\}$, with $x_{ij}$ denoting the squared distance between vertices $i$ and $j$ and $n$ being the number of vertices. There are certain dependencies between these variables, captured by the polynomials $f\in \q[X_n]$ generating the Cayley-Menger ideal. When $G$ is a minimally rigid graph, the addition of a new edge $e$ induces a unique subgraph $C\subseteq G\cup\{ e\}$ which is a {\em circuit} in the 2D rigidity matroid whose bases are the minimally rigid graphs. There also exists a unique (up to multiplication by a scalar) polynomial dependency $p_C$ between the distances corresponding to the edges of $C$. This is a {\em circuit polynomial} in the Cayley-Menger ideal, and is the main object of study in this paper. The unique unknown distance problem is solved by substituting in this circuit polynomial concrete values for the edge weights of $G$  and then computing the roots of the resulting uni-variate polynomial.

\subparagraph{How tractable is the problem?} 
Circuit polynomial computations can be done, in principle, by using the Gr\"obner basis algorithm with an elimination order\footnote{See Exercises 5 and 6 in \S1 of Chapter 3 in \cite{CoxLittleOshea}}. In the worst case, this is a doubly-exponential method but in practice, the complexity and performance of Gr\"obner basis algorithms depends heavily on the choice of a {\em monomial order}. There exist known cases, e.g. zero-dimensional polynomial ideals \cite{dickenstein:etal:membProblem,Lakshman1991}, which have single-exponential complexity with respect to any monomial order. However, {\em elimination orders} have been reported to behave badly. In general, the main problems of Elimination Theory, such as the Ideal Triviality Problem, the Ideal Membership Problem for Complete Intersections, the Radical Membership Problem, the General Elimination Problem, and the Noether Normalization are in the PSPACE complexity class \cite{MateraMariaTurull}. 

In our experimentation, the {\bf GroebnerBasis} function of Mathematica 12 (running on a 2019 iMac computer with 6 cores at 3.6Ghz) took 5 days and 6 hours to compute the Desargues-plus-one circuit (a graph on $6$ vertices) reported in \cref{tbl:circPoly} of \cref{sec:experiments}, but in most cases it timed out or crashed.

\subparagraph{Overview of Results.} Our goal is to make such calculations {\em more tractable} by taking advantage of {\em structural information} inherent in the problem. We describe a new {\em algorithm to compute a circuit polynomial with known support.} It relies on resultant-based elimination steps guided by a novel {\em inductive construction for rigidity circuits}. Inductive constructions have been often used in Rigidity Theory, most notably the Henneberg sequences for Laman graphs \cite{henneberg:graphischeStatik:1911-68} and Henneberg II sequences for $3$-connected rigidity circuits \cite{BergJordan}. We argue that our combinatorial construction is more {\em natural} due to its direct algebraic interpretation, a property not shared with any of the other previously known constructions. We have implemented our method in Mathematica and applied it successfully to compute all but one of the circuit polynomials on up to $6$ vertices, as well as a few on $7$ and $8$ vertices, the largest of which having over nine million terms. The previously mentioned example of the Desargues-plus-one circuit that took over 5 days to complete with GroebnerBasis, was solved by our algorithm in less than 15 seconds. 

The only example on 6 vertices that remained elusive was the circuit polynomial for the $K_{3,3}$-plus-one circuit (see \cref{tbl:circPoly} of \cref{sec:experiments}): the computational resources for its computation far exceeded the capabilities of both our machines and of a HPC system we experimented with. We succeeded by extending the basic algorithm to work with additional generators of the Cayley-Menger ideal, besides those corresponding to $K_4$'s. These are irreducible polynomials supported on dependent rigid graphs that are not necessarily circuits.

\subparagraph{Related work.}  Our approach builds upon ideas from \emph{distance geometry} and \emph{rigidity theory} and combines them with the theory of algebraic matroids. The former enjoy a long and distinguished history - too long to survey here but see \cite{blumenthal:distGeo,crippen:havel:distanceGeoMolConf}. Combinatorial and linear (but not algebraic) matroids occupy a central place in Rigidity Theory \cite{graverServatiusServatius,whiteley:Matroids:1996}. 
To the best of our knowledge, the study of circuit polynomials in \emph{arbitrary} polynomial ideals was initiated in the PhD thesis of Rosen \cite{rosen:thesis}. His Macaulay2 code \cite{rosen:GitHubRepo} is useful for exploring small cases, but the Cayley-Menger ideal is beyond its reach. A recent article \cite{rosen:sidman:theran:algebraicMatroidsAction:2020} popularizes algebraic matroids and uses for illustration the smallest circuit polynomial $K_4$ in the Cayley-Menger ideal. \emph{We could not find non-trivial examples anywhere}. Indirectly related to our problem are results such as \cite{WalterHusty}, where an explicit univariate polynomial of degree 8 is computed (for an unknown angle in a $K_{3,3}$ configuration given by edge lengths, from which the placement of the vertices is determined) and \cite{sitharam:convexConfigSpaces:2010}, for its usage of Cayley coordinates in the study of configuration spaces of some families of distance graphs. A closely related problem is that of computing the \emph{number of embeddings of a minimally rigid graph} \cite{streinu:borcea:numberEmbeddings:2004}, which has received a lot of attention in recent years (e.g. \cite{capco:schicho:realizations:2017,bartzos:emiris:etAl:realizations:2021,emiris:tsigaridas:varvitsiotis:mixedVolume,emiris:mourrain}, to name a few).
References to specific results in the literature that are relevant to the theory developed here and to our proofs are given throughout the paper.

\subparagraph{Overview of the paper.}  Our main theoretical result is split into a combinatorial \cref{thm:combResConstruction} and an algebraic \cref{thm:circPolyConstruction}, each with an algorithmic counterpart and each preceeded by a section introducing the concepts necessary for a self-contained presentation. Section \ref{sec:prelimRigidity} reviews 2D combinatorial rigidity matroids. Then in \cref{sec:combRes} we  define the {\em combinatorial resultant} of two graphs as an abstraction of the classical resultant, prove \cref{thm:combResConstruction} and describe the algorithm for computing a \emph{combinatorial circuit-resultant (CCR) tree}.

\begin{theorem}
	\label{thm:combResConstruction}
	Each rigidity circuit can be obtained, inductively, by applying combinatorial resultant operations starting from $K_4$ circuits. The construction is captured by a binary {\em resultant tree} whose nodes are intermediate rigidity circuits and whose leaves are $K_4$ graphs.
\end{theorem}

This leads to a  {\em graph algorithm} for finding a {\em CCR tree} of a circuit. Each step of the construction can be carried out in polynomial time using variations on the {\em Pebble Game} matroidal sparsity algorithms \cite{streinu:lee:pebbleGames:2008} combined with Hopcroft and Tarjan's linear time $3$-connectivity algorithm \cite{hopcroft:tarjan:73}. However, it is conceivable that the tree could be exponentially large and thus the entire construction could take an exponential number of steps: understanding in detail the algorithmic complexity of our method {\em remains a problem for further investigation.}

In sections \ref{sec:prelimAlgMatroids},\ref{sec:prelimCMideal}, \ref{sec:prelimResultants} and \ref{sec:circuitsCMideal} we include a brief, self-contained  overview of the algebraic concepts relevant to this paper: ideals and their algebraic matroids, the Cayley-Menger ideal, resultants, and the circuit polynomials in the Cayley-Menger ideal. In \cref{sec:algResCircuits} we prove: 

\begin{theorem}
	\label{thm:circPolyConstruction}
	Each circuit polynomial can be obtained, inductively, by applying resultant operations. The procedure is guided by the combinatorial circuit-resultant (CCR) tree from \cref{thm:combResConstruction} and builds up from $K_4$ circuit polynomials. At each step, the resultant produces a polynomial that may not be irreducible. A polynomial factorization and a test of membership in the ideal are then applied to identify the factor which is the actual circuit polynomial.
\end{theorem}

The algorithmic counterpart of \cref{thm:circPolyConstruction} appears in \cref{sec:resTree}. Overall, the resulting {\em algebraic elimination algorithm} runs in exponential time, in part because of the growth in size of the polynomials that are being produced. Several theoretical {\em open questions} remain, whose answers may affect the precise time complexity analysis. 

In \cref{sec:extendedCombRes} we define and characterize a more general \emph{combinatorial resultant tree} which generalizes the CCR tree by allowing more freedom in the choice of graphs used at the leaves of the tree: besides $K_4$ circuits, we now can use dependent rigid graphs. This extension allows the use of polynomials supported on dependent sets in the Cayley-Menger ideal that are not necessarily circuits. The dependent, non-circuit generators of the Cayley-Menger ideal are discussed in \cref{sec:generatorsCM} and the full generalization of our main algorithm is given in \cref{sec:algebraicNew}.

The preliminary experimental results we carried with the implementation of our method in Mathematica are discussed in \cref{sec:experiments}. We used Mathematica v13 on an 2019 iMac with the following specifications: Intel i5-9600K 3.7GHz, 16 GB RAM, macOS Monterey 12.3.1. We also explored Macaulay2, but it was much slower than Mathematica (hours vs.\ seconds) in computing one of our examples. The resulting polynomials are made available on a github repository \cite{malic:streinu:GitHubRepo}. 

Open questions are introduced throughout the paper and in the final \cref{sec:concludingRemarks}.
\medskip
\subparagraph{Further connections: circuit polynomials in matroid theory.} The Matroid Theory literature is rich in realizability questions of various sorts \cite{Oxley:2011} and has seen in recent years a surge of interest in algebraic matroids. Ingleton \cite{Ingleton} proved that algebraic matroids over fields of characteristic 0 are linearly realizible, but this is not the case in positive characteristic \cite{Oxley:2011}. Recently, \cite{bollen:draisma:pendavingh:algMatroidsFrobenius:2018} have identified an infinite class of algebraic matroids over fields of positive characteristic that have a linear representation in the same characteristic, namely those for which the so-called Lindstr\"om valuation is trivial. The problem of computing the Lindstr\"om valuation was addresed in \cite{cartwright:lindstromValuation:2018}, where the fundamental step is to compute all circuit polynomials of a given algebraic matroid in positive characteristic. We remark that for the algebraic matroids whose combinatorial structure allows descriptions of their circuits in terms of an operation similar to our combinatorial resultants, the methods presented in this paper are applicable and likely to be more efficient than \grobner{} basis methods.

\subparagraph{Remark.} The main results of this paper have been announced in the conference abstract \cite{malic:streinu:socg} and in \cite{malic:streinu:arxiv:k33:2021}.

\section{Preliminaries: rigidity circuits.}
\label{sec:prelimRigidity}

We start with the combinatorial aspects of our problem and review the relevant notions and results from combinatorial rigidity theory of bar-and-joint frameworks in dimension $2$.

\subparagraph{Notation.} We work with (sub)graphs given by subsets $E$ of edges of the complete graph $K_n$ on vertices $[n]:=\{1,\dots,n\}$. If $G$ is a (sub)graph, then $V(G)$, resp.\ $E(G)$ denote its vertex, resp.\ edge set. The support of $G$ is $E(G)$. The {\em vertex span} $V(E)$ of edges $E$ is the set of all edge-endpoint vertices. A subgraph $G$ is {\em spanning} if its edge set $E(G)$ spans $[n]$. The {\em neighbours} $N(v)$ of vertex $v$ are the vertices adjacent to $v$ in $G$.

\subparagraph{Frameworks.} A {\em 2D bar-and-joint framework} is a pair $(G,p)$ of a graph $G=(V,E)$ and a \emph{placement map} $p\colon V\to\mathbb R^2$. We view the edges as {\em rigid bars} and the vertices as {\em rotational joints} which allow the framework to deform continuously as long as the bars retain their original lengths. The {\em realization space} of the framework is the set of all of its possible placements in the plane with the same bar lengths. Two realizations are congruent if they are related by a planar isometry. The {\em configuration space} of the framework is made of congruence classes of realizations. The {\em deformation space of a given framework $(G,p)$} is the  connected component of the configuration space that contains this particular placement (given by $p$). A framework is {\em rigid} if its deformation space consists of exactly one configuration, and {\em flexible} otherwise. 
We say that a framework is {\em minimally rigid} if it is rigid and, when any of its edges is removed, it becomes {\em flexible}.

\subparagraph{Laman Graphs.} 

The concept of a {\em generic framework} is introduced rigorously in \cref{sec:prelimCMideal}.  All but a measure-zero set of possible placements of a graph are generic. 
The following theorem allows us to refer to the rigidity and flexibility of a generic framework solely in terms of its underlying graph. The proof goes through the intermediate concept of {\em infinitesimal rigidity},  which implies rigidity; this is also introduced in \cref{sec:prelimCMideal}.

\begin{theorem} 
	\cite{geiringer,laman:rigidity:1970}
	\label{thm:laman}
	A generic bar-and-joint framework is {\em minimally rigid} in 2D iff its underlying graph $G=(V,E)$ satisfies two conditions: (a) it has exactly $|E|=2|V|-3$ edges, and (b) any proper subset $V'\subset V$ with $|V'|\geq 2$ of vertices spans at most $2|V'|-3$ edges.
\end{theorem}

A graph satisfying the conditions of \cref{thm:laman} is said to be a {\em Laman graph}, or just {\em Laman}. The hereditary property (b) is also referred to as the {\em $(2,3)$-sparsity condition}. Together, properties (a) and (b) define a graph said to be $(2,3)$-\emph{tight} (in addition to being $(2,3)$-sparse). 

\cref{thm:laman} allows us to talk now about {\em (minimal) rigidity} of graphs rather than frameworks. 
A Laman graph is {\em minimally rigid} and it becomes {\em flexible} when any of its edges is removed. Adding extra edges to a Laman graph keeps it rigid, but the minimality is lost: these graphs are said to be rigid and {\em overconstrained} or {\em dependent}. In short, for a graph to be rigid, its vertex set must span a Laman graph; otherwise the graph is flexible. Other graphs may be simultaneously flexible and overconstrained. In this paper, we work primarily with graphs which are rigid and dependent. The minimally dependent ones, called {\em rigidity circuits}, are introduced next.

\subparagraph{Matroids. } A matroid is an abstraction capturing (in)dependence relations among collections of elements from a {\em ground set}, and is inspired by both {\em linear} dependencies (among, say, rows of a matrix) and by {\em algebraic} constraints imposed by algebraic equations on a collection of otherwise free variables. The standard way to specify a matroid is via its {\em independent sets}, which have to satisfy certain axioms (which we omit, and refer the interested reader to  \cite{Oxley:2011}). A {\em base} is a maximal independent set and a set which is not independent is said to be {\em dependent}. A minimal dependent set is called a {\em circuit}. Relevant for our purposes are the following general aspects: (a) (hereditary property) a subset of an independent set is also independent; (b) all bases have the same cardinality, called the {\em rank} of the matroid. Further properties will be introduced in context, as needed.

In this paper we encounter three {\em types of rigidity-related matroids}: a {\em graphic\footnote{Not to be confused with the matroid of spanning trees of the complete graph.} matroid}, defined on a ground set given by all the edges $E_n:=\{ij: 1\leq i < j \leq n\}$ of the complete graph $K_n$; this is the {\em $(2,3)$-sparsity matroid} or the {\em generic 2D rigidity matroid} described below; a {\em linear matroid}, defined on an isomorphic set of {\em row vectors} of the {\em rigidity matrix} associated to a bar-and-joint framework; 
and an {\em algebraic matroid},  defined on an isomorphic ground set of variables $X_n:=\{x_{ij}: 1\leq i < j \leq n\}$; this is the {\em algebraic matroid associated to the Cayley-Menger ideal}. The linear and algebraic matroids will be defined in \cref{sec:prelimCMideal}. 

\subparagraph{The $(2,3)$-sparsity matroid: independent sets, bases, circuits.}
The $(2,3)$-sparse graphs on $n$ vertices form the collection of independent sets for a matroid $\smat{n}$ on the ground set $E$ of edges of the complete graph $K_n$ \cite{whiteley:Matroids:1996}, called the (generic) {\em 2D rigidity matroid}, or the {\em $(2,3)$-sparsity matroid}. The bases of the matroid $\smat n$ are the maximal independent sets, hence are Laman graphs. A set of edges which is not sparse is a {\em dependent} set. For instance, adding one edge to a Laman graph creates a dependent set of $2n-2$ edges, called a Laman-plus-one graph: examples are given in \cref{fig:lamanAndCircuits}. 

\begin{figure}[ht]
	\centering
		\includegraphics[width=.24\textwidth]{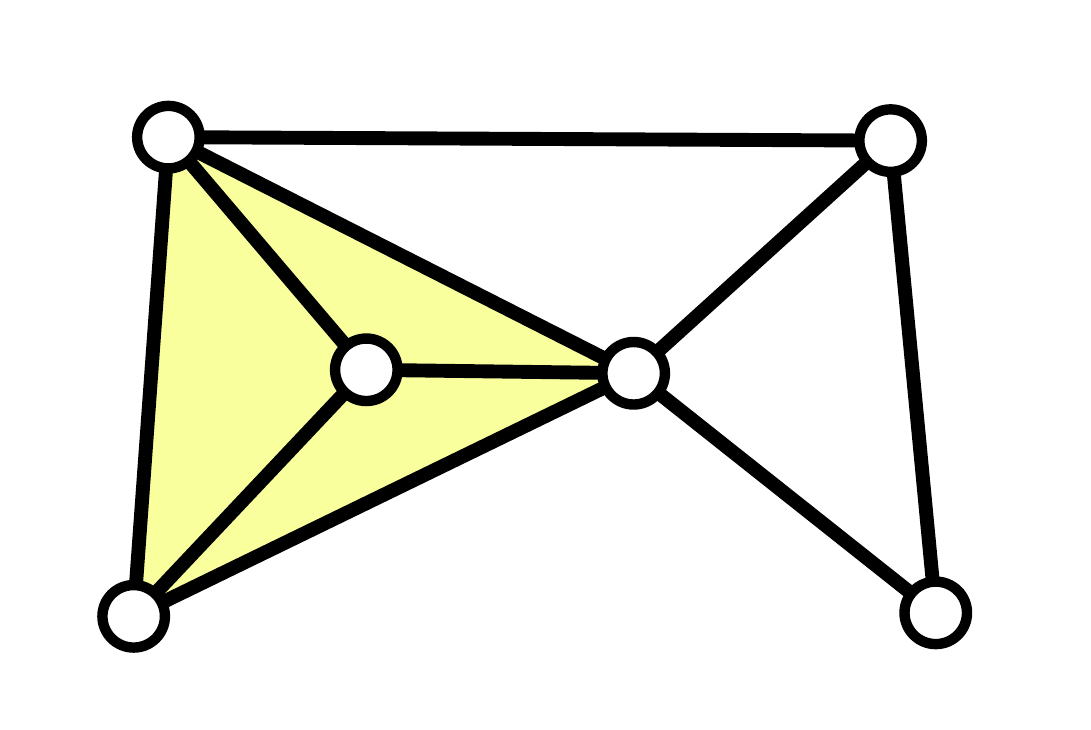}
		\includegraphics[width=.24\textwidth]{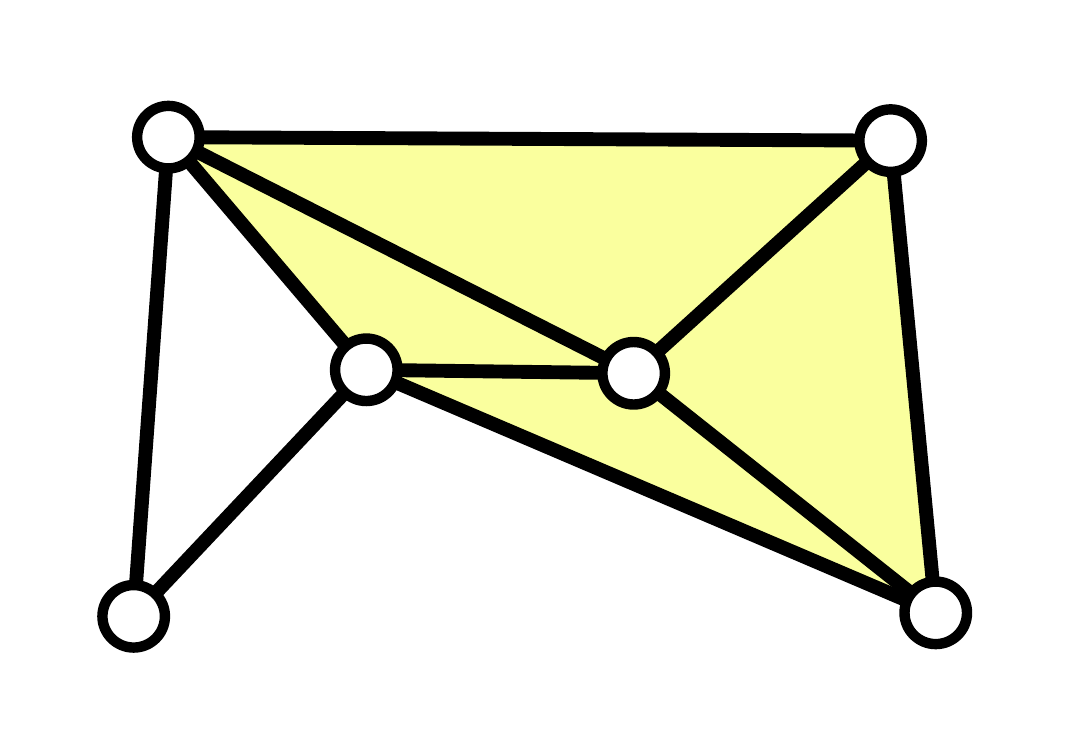}
		\includegraphics[width=.24\textwidth]{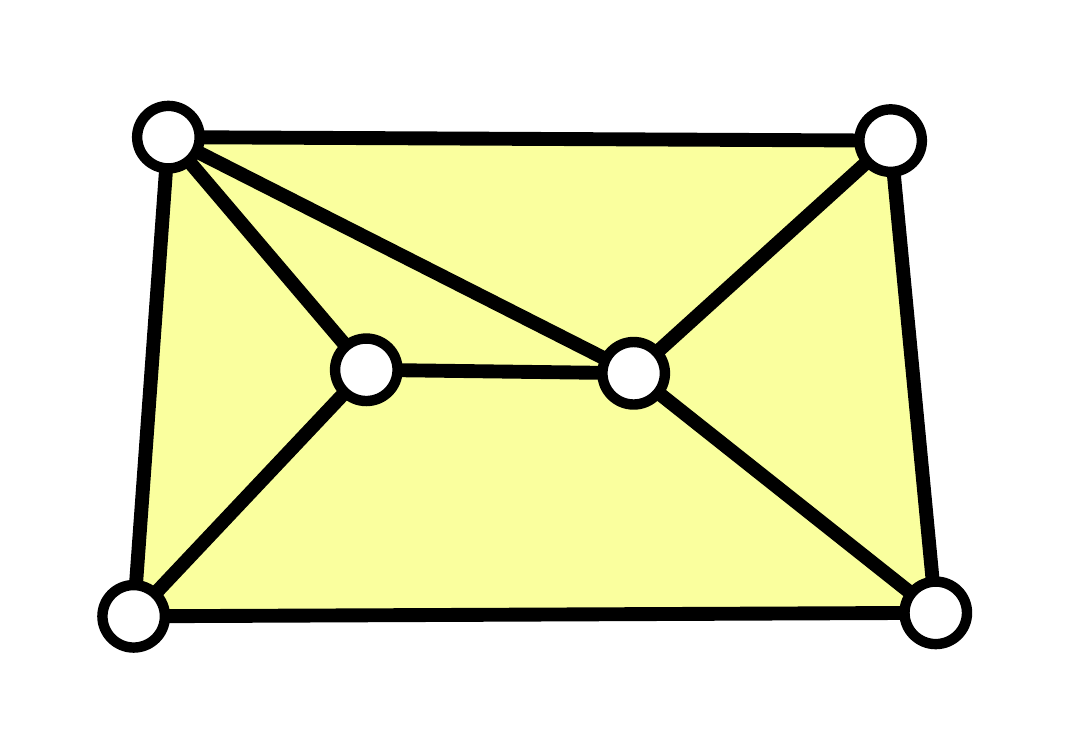}
	\caption{A Laman-plus-one graph contains a unique circuit (highlighted): (Left and Center) The circuit is not spanning the entire vertex set. (Right) A spanning circuit. 
	}
	\label{fig:lamanAndCircuits}
\end{figure}

A {\em minimal} dependent set is a (sparsity) {\em circuit}. The edges of a circuit span a subset of the vertices of $V$. A circuit spanning $V$ is said to be a {\em spanning} or {\em maximal} circuit in the sparsity matroid $\smat n$. See \cref{fig:lamanAndCircuits}(right) and \cref{fig:6circuits} for examples. 

\begin{figure}[ht]
	\centering
	\includegraphics[width=0.26\textwidth]{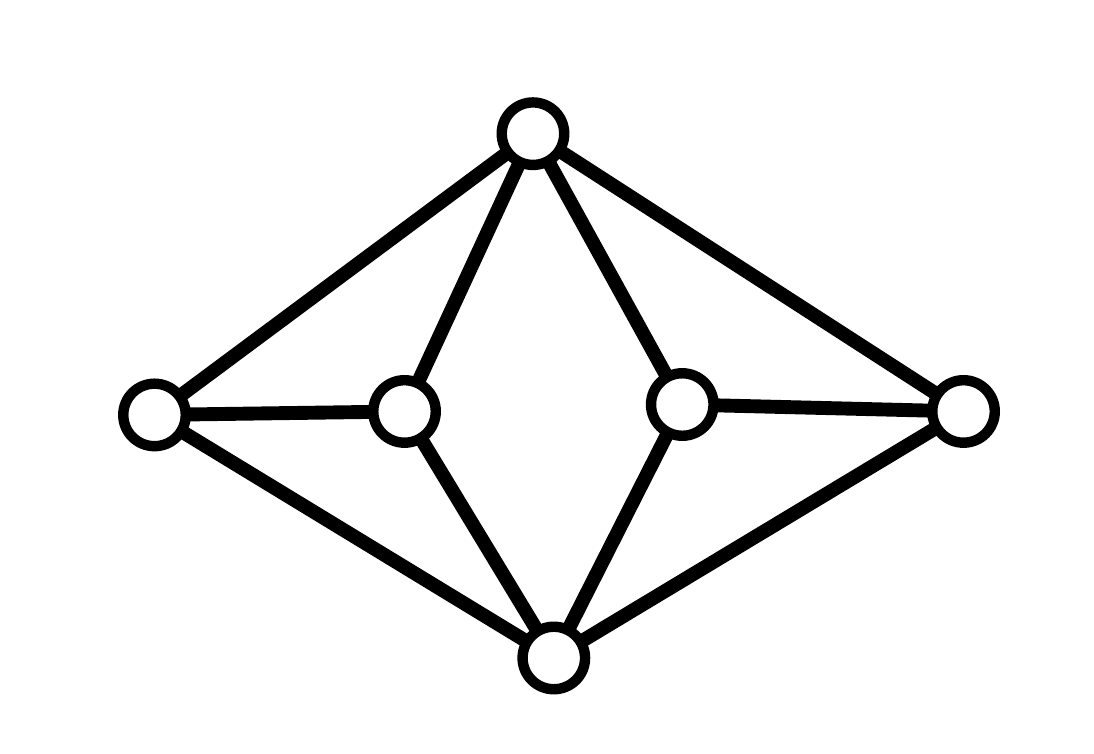}
	\includegraphics[width=0.22\textwidth]{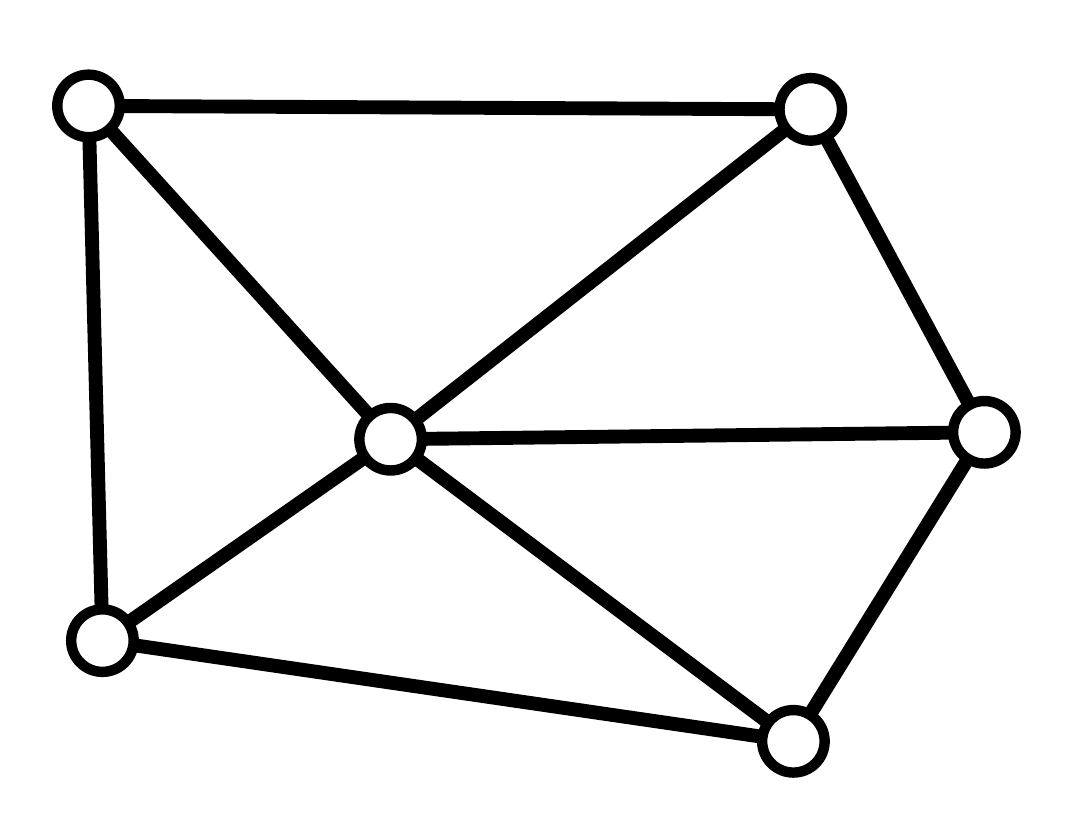}
	\includegraphics[width=0.24\textwidth]{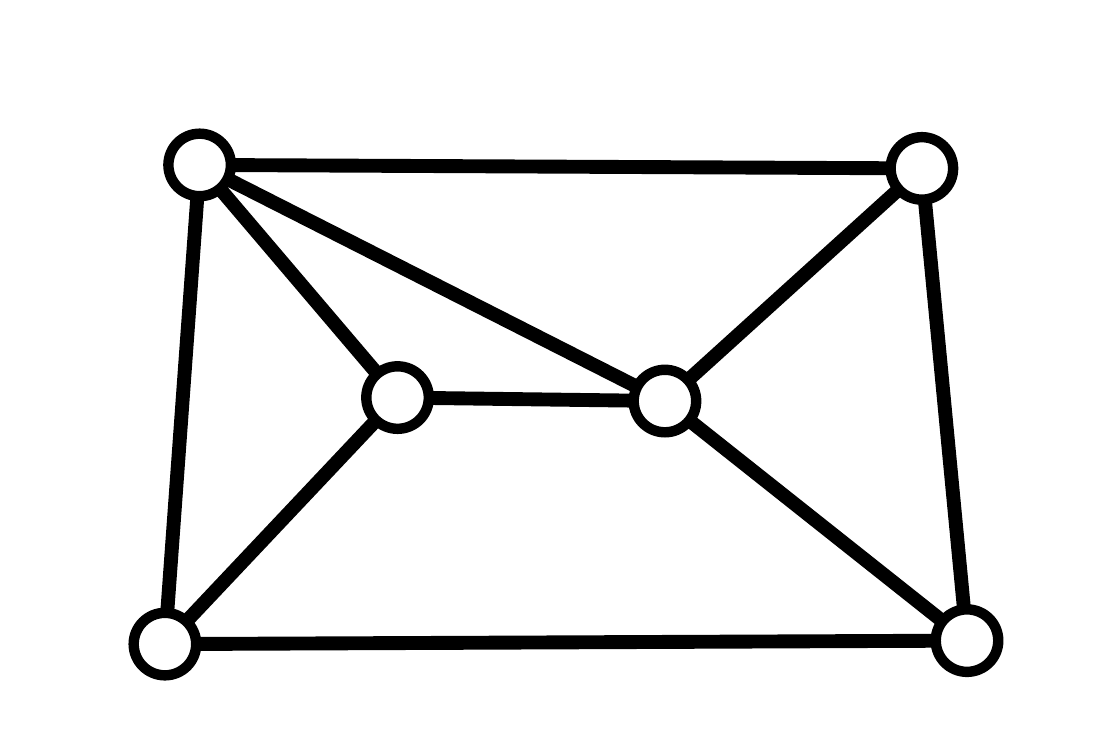}
	\includegraphics[width=0.24\textwidth]{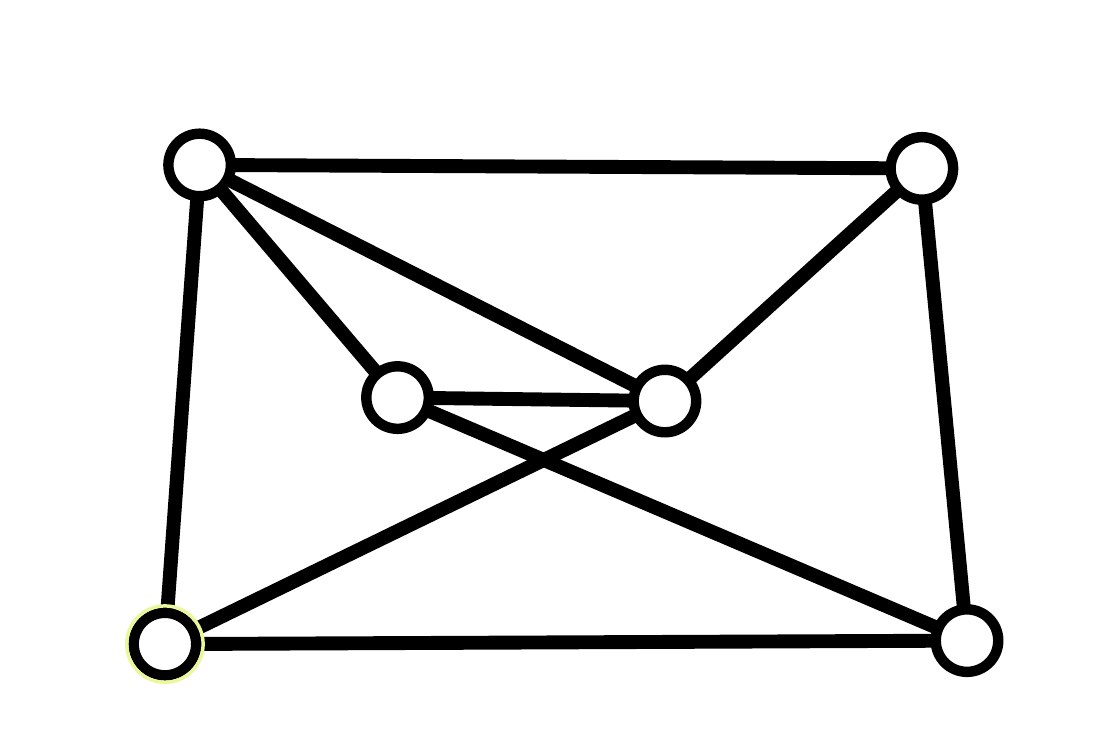}
	\caption{The four types of spanning circuits on $n=6$ vertices: 2D {\em double-banana}, {\em $5$-wheel} $W_5$, {\em Desargues-plus-one} and $K_{3,3}$-plus-one.}
	\label{fig:6circuits}
\end{figure}

A {\em Laman-plus-one} graph contains a unique subgraph which is {\em minimally dependent}, in other words, a unique circuit. A spanning rigidity circuit $C=(V,E)$ is a special case of a Laman-plus-one graph: it has a total of $2n-2$ edges but it satisfies the $(2,3)$-sparsity condition on all proper subsets of at most $n'\leq n-1$ vertices. Simple sparsity considerations can be used to show that the removal of {\em any} edge from a spanning circuit results in a Laman graph.

\subparagraph{Combining graphs and circuits.} We define now operations that combine two graphs (with some common vertices and edges) into one. 

If $G_1$ and $G_2$ are two graphs, we use a consistent {\bf notation} for their number of vertices and edges $n_i=|V(G_i)|$, $m_i=|E(G_i)|$, $i=1,2$, and for their union and intersection of vertices and edges, as in $V_{\cup}=V(G_1)\cup V(G_2)$, $V_{\cap}=V(G_1)\cap V(G_2)$, $n_{\cup}=|V_{\cup}|$, $n_{\cap}=|V_{\cap}|$ and similarly for edges, with $m_{\cup}=|E_{\cup}|$ and $m_{\cap}=|E_{\cap}|$. The {\em common subgraph} of two graphs $G_1$ and $G_2$ is $G_{\cap} = (V_{\cap}, E_{\cap})$.

Let $G_1$ and $G_2$ be two graphs with exactly two vertices $u, v\in V_{\cap} $ and one edge $uv \in E_{\cap}$ in common. Their {\bf $2$-sum} is the graph $G=(V,E)$ with $V=V_{\cup}$ and $E=E_{\cup} \setminus \{uv\}$. The inverse operation of splitting $G$ into $G_1$ and $G_2$ is called a $2$-split or $2$-separation (\cref{fig:2sum}).

\begin{figure}[ht]
	\centering
	\includegraphics[width=0.3\textwidth]{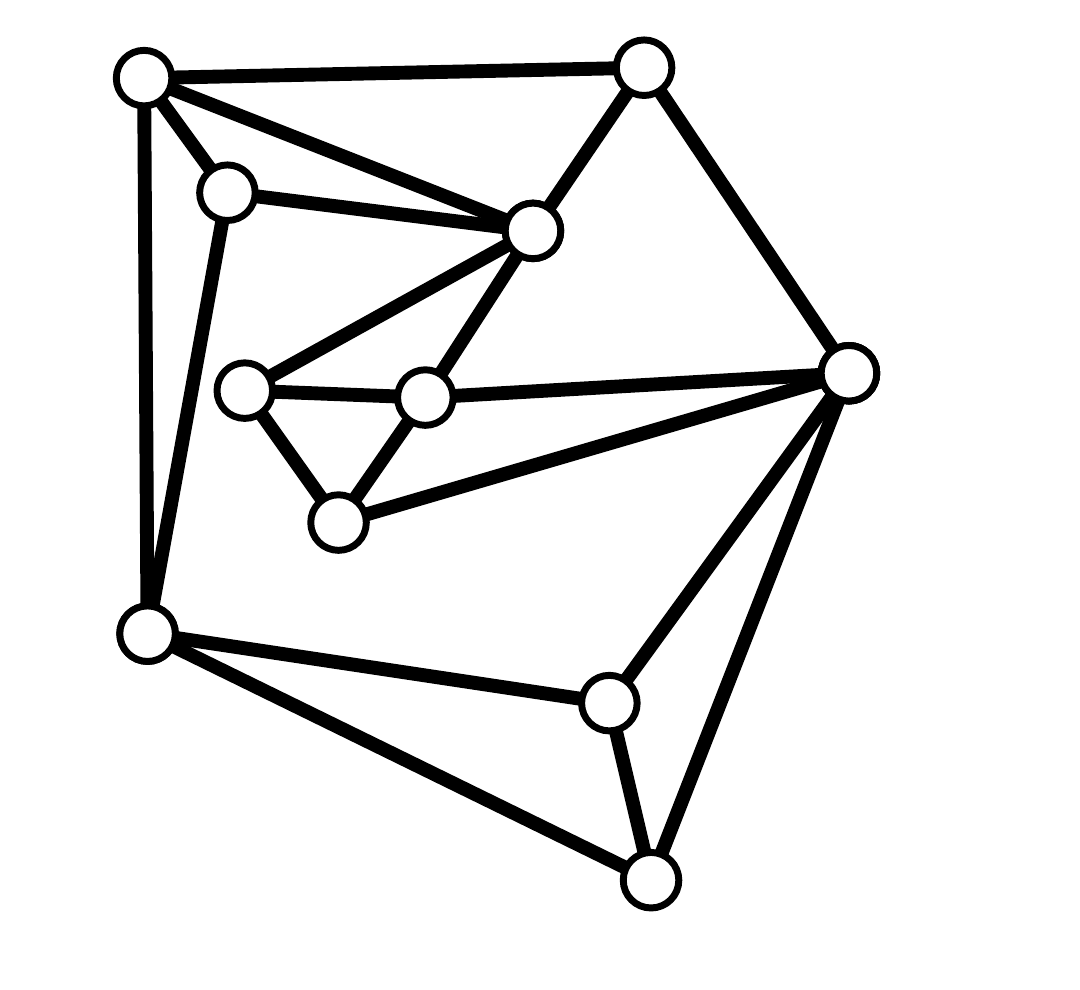}
	\includegraphics[width=0.3\textwidth]{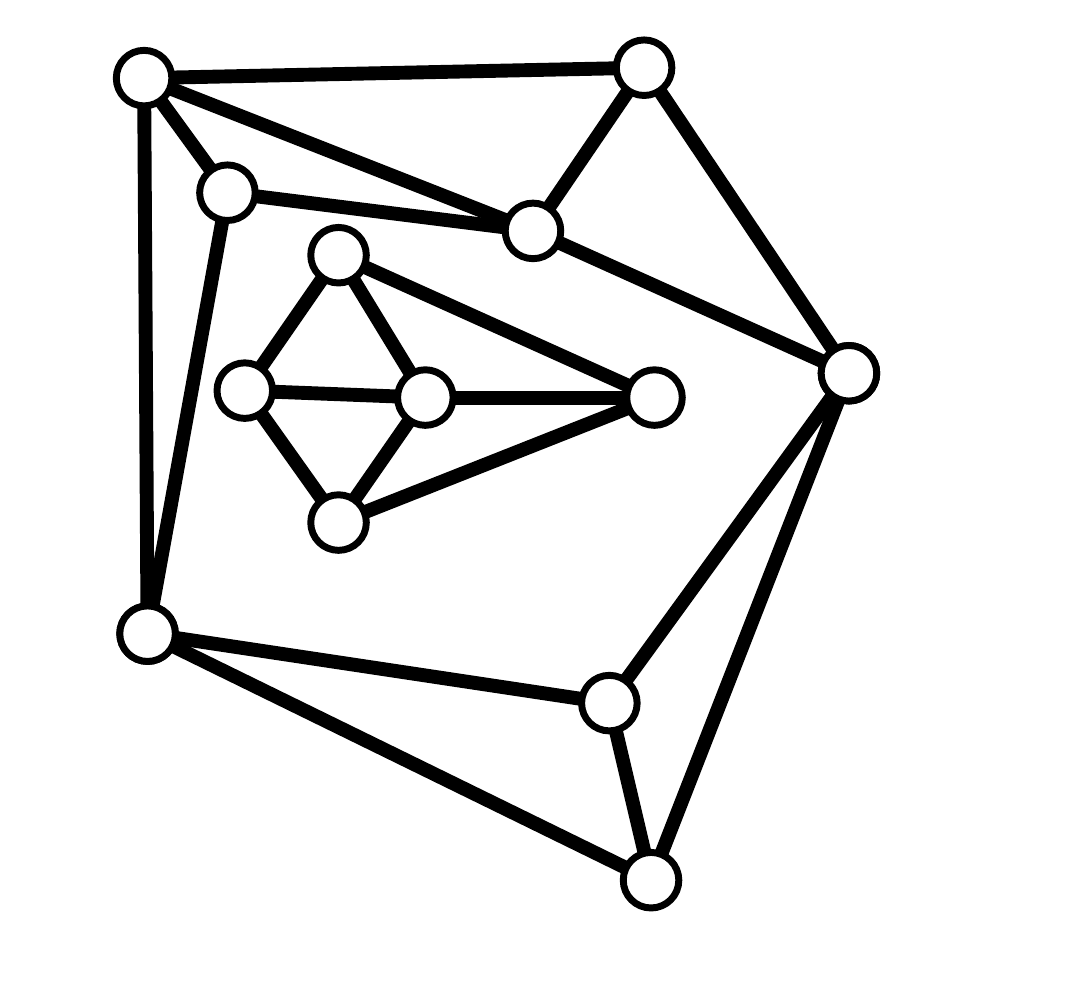}
	\includegraphics[width=0.3\textwidth]{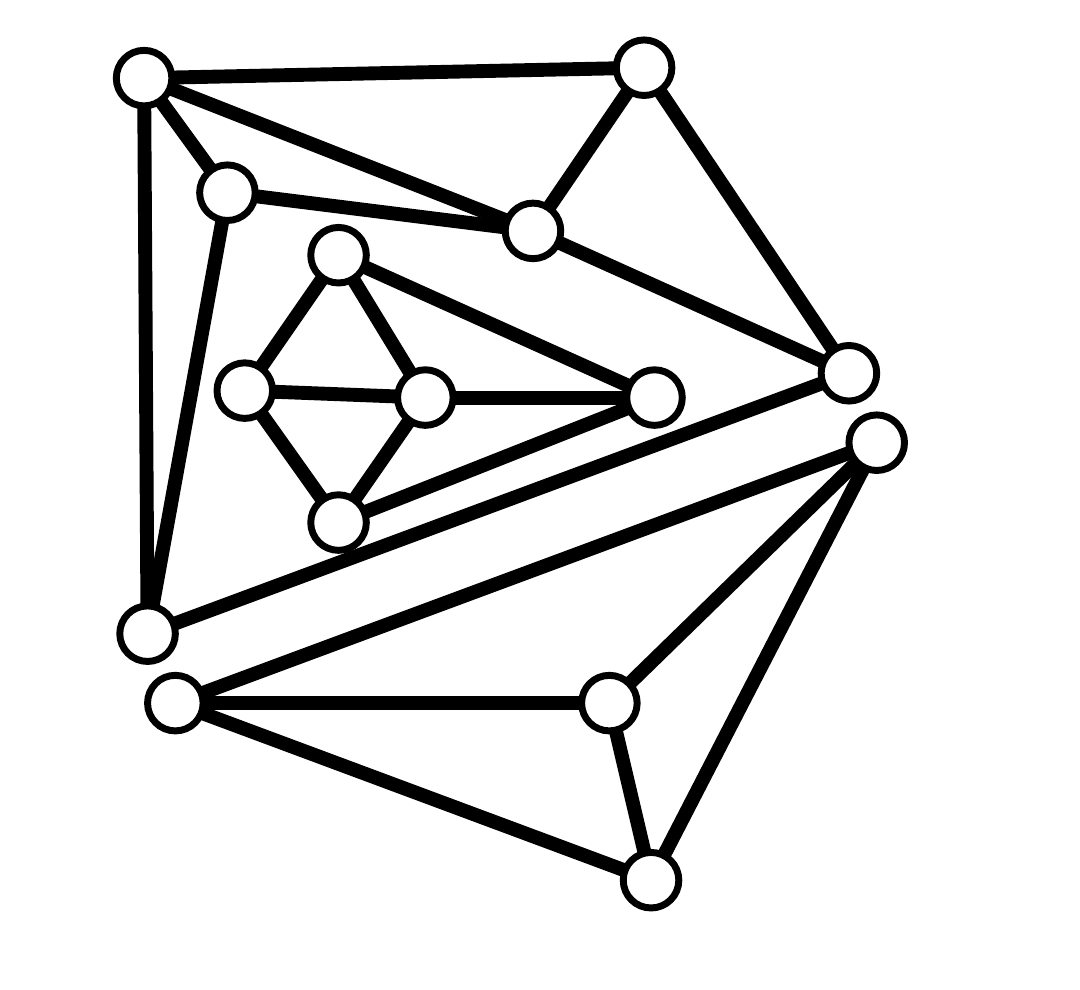}
	\caption{(Left-to-right) Separating a $2$-connected circuit into three $3$-connected circuits via $2$-split operations. (Right-to-left) Combining three $3$-connected circuits into a larger (not-$3$-connected) one, via $2$-sum operations.}
	\label{fig:2sum}
\end{figure}

\begin{lemma}[\cite{BergJordan}, Lemmas 4.1 and 4.2] \label{lem:twoSum}
	The $2$-sum of two circuits is a circuit. The 2-split of a circuit is a pair of circuits.
\end{lemma}

\subparagraph{Connectivity.} It is well known and easy to show that a circuit is always a $2$-connected graph. {\em If a circuit is not $3$-connected,}  we refer to it simply as a {\em $2$-connected circuit}.  The  Tutte decomposition \cite{tutte:connectivity:1966} of a $2$-connected graph into $3$-connected components amounts to identifying separating pairs of vertices. For a circuit, the separating pairs induce $2$-splits (inverse of $2$-sum) operations and produce smaller circuits (see also Lemma 2.4(c) in \cite{BergJordan}). Thus a $2$-connected circuit can be constructed from $3$-connected circuits via $2$-sums, as illustrated in the right-to-left sequence from \cref{fig:2sum}.

\subparagraph{Inductive constructions for $3$-connected  circuits.} A {\em Henneberg II} extension (also called an {\em edge splitting} operation) is defined for an edge $uv$ and a non-incident vertex $w$, as follows: the edge $uv$ is removed, a new vertex $a$ and three new edges $au, av, aw$ are added. Berg and Jordan \cite{BergJordan} have shown that, if $G$ is a $3$-connected circuit, then a Henneberg II extension on $G$ is also a $3$-connected circuit. The {\em inverse Henneberg II} operation on a circuit removes one vertex of degree $3$ and adds a new edge among its three neighbors in such a way that the result is also a circuit, \cref{fig:hennebergII}. Berg and Jordan have shown that every $3$-connected circuit admits an inverse Henneberg II operation which also maintains $3$-connectivity. As a consequence, a $3$-connected circuit has an {\em inductive construction}, i.e.\ it can be obtained from $K_4$ by Henneberg II extensions that maintain $3$-connectivity. Their proof is based on the existence of two non-adjacent vertices with $3$-connected inverse Henneberg II circuits. We will make use in \cref{sec:combRes} of the following weaker result, which does not require maintaining of $3$-connectivity in the inverse Henneberg II operation.

\begin{figure}[th]
	\centering
	\includegraphics[width=0.3\textwidth]{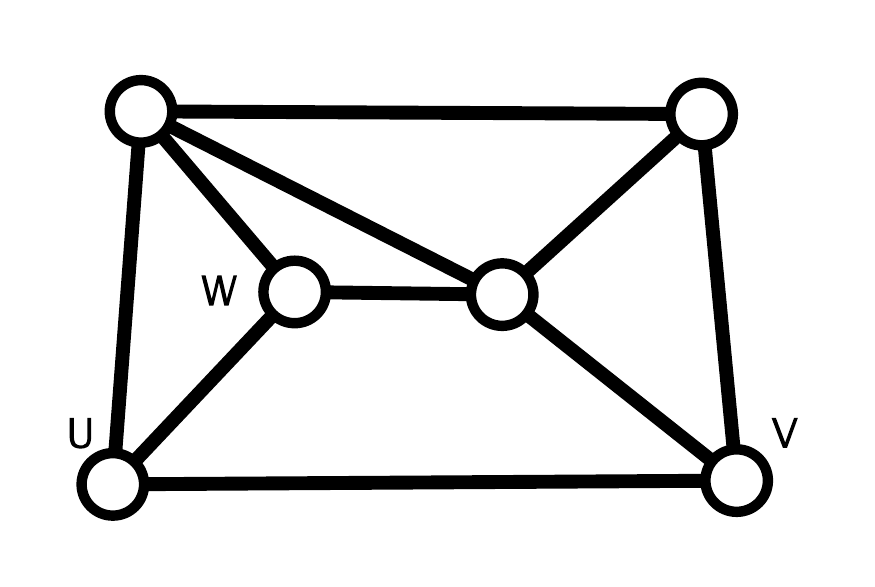}
	\includegraphics[width=0.3\textwidth]{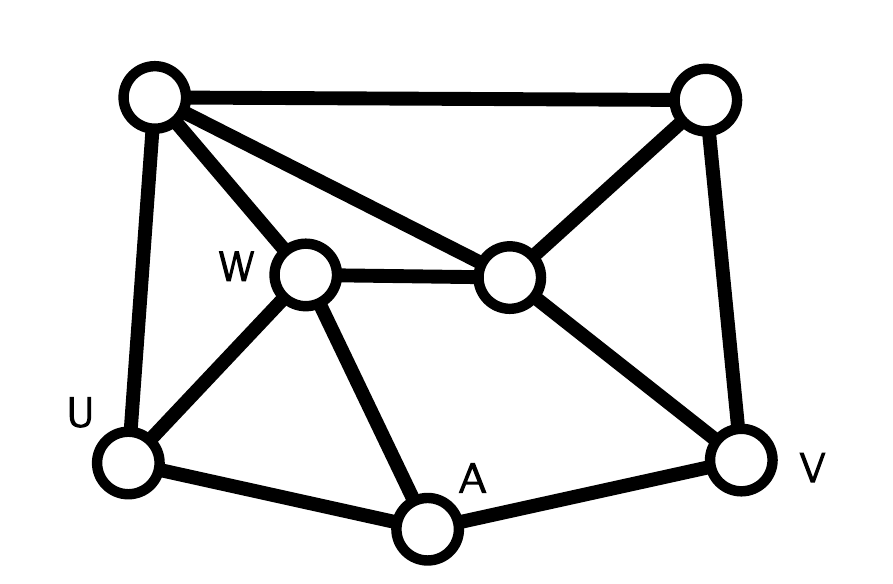}
	\caption{A Henneberg II extension of the Desargues-plus-one circuit.}
	\label{fig:hennebergII}
\end{figure}

\begin{lemma}[Theorem 3.8 in \cite{BergJordan}]
	\label{thm:BergJordan}Let $G=(V,E)$ be a $3$-connected circuit with $|V|\geq 5$. Then either $G$ has four vertices that admit an inverse Henneberg II that is a circuit, or $G$ has three pairwise non-adjacent vertices that admit an inverse Henneberg II that is a circuit (not necessarily $3$-connected).
\end{lemma}

\section{Combinatorial Resultant Constructions.}
\label{sec:combRes}

We define now a new operation, the \emph{combinatorial resultant} of two graphs, prove \cref{thm:combResConstruction} and describe its algorithmic implications. 

\subsection{Definition: Combinatorial resultant.}

\begin{figure}[ht]
	\centering
	\includegraphics[width=0.18\textwidth,trim={0 0.275mm 0 0},clip]{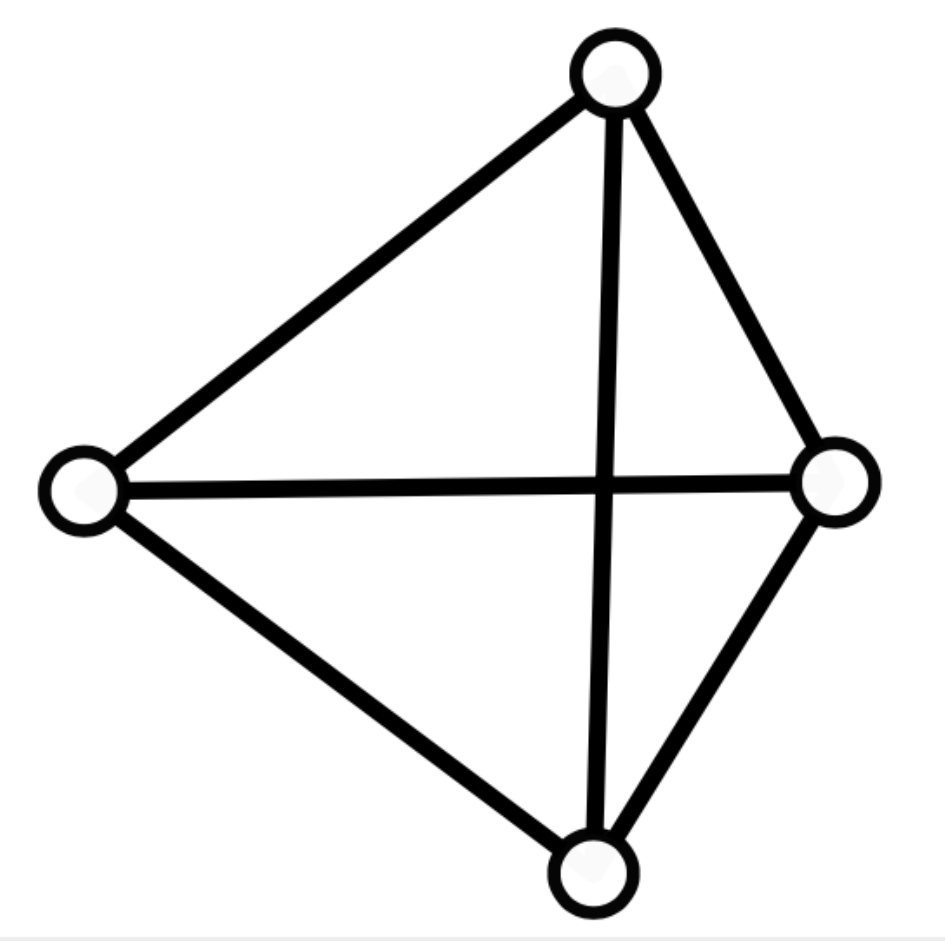}
	\includegraphics[width=0.22\textwidth]{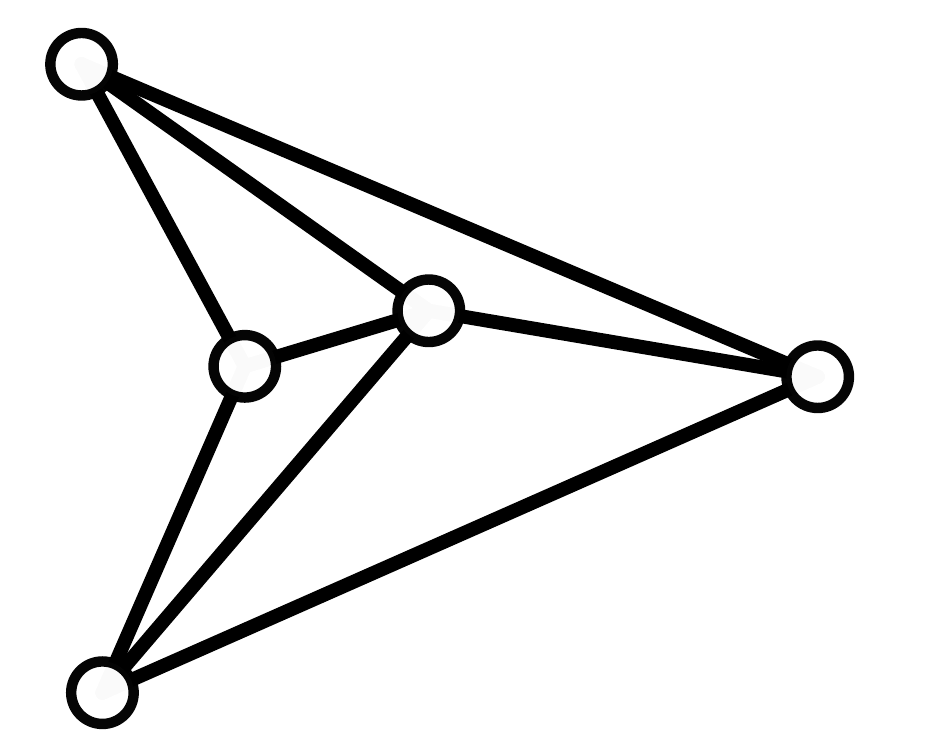}
	\includegraphics[width=0.27\textwidth]{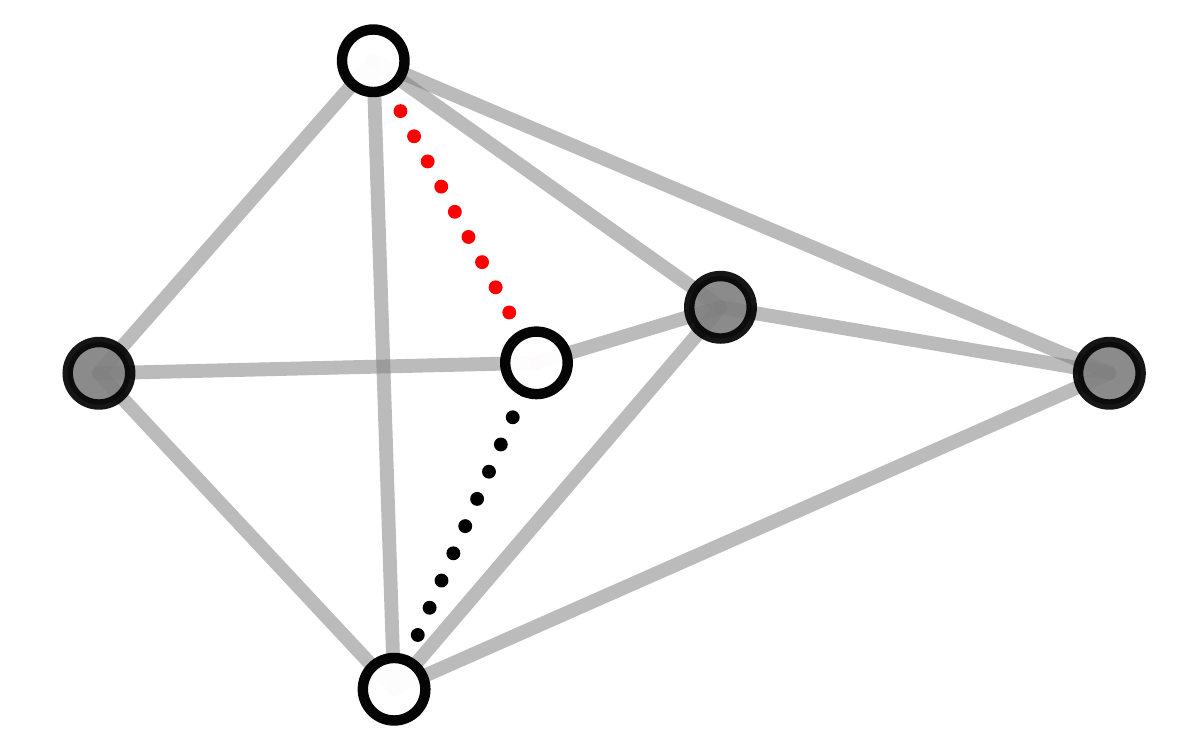}
	\includegraphics[width=0.27\textwidth]{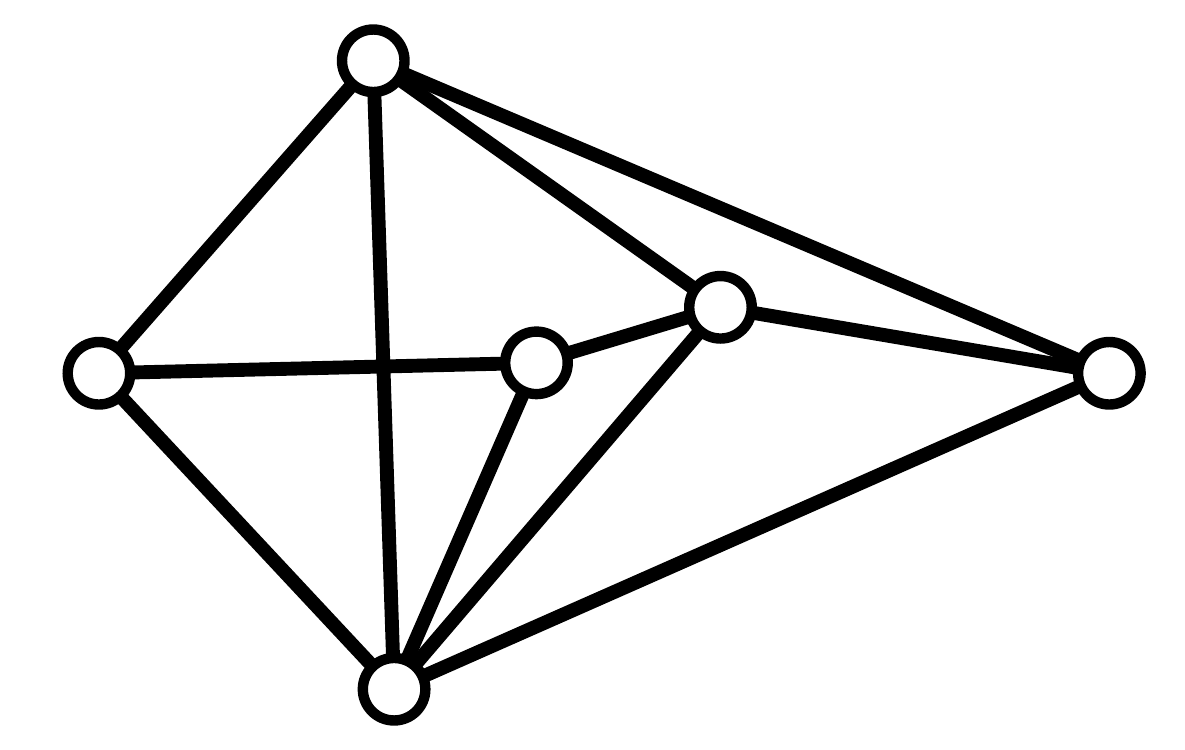}	
	\caption{A complete $K_4$ graph, a $4$-wheel $W_4$, their common edges (dotted, with elimination edge in red) and their combinatorial resultant, which has more than $2n-2$ edges and thus is not a circuit.}
	\label{fig:graphK4W4}
\end{figure}

	Let $G_1$ and $G_2$ be two distinct graphs with non-empty intersection $E_{\cap} \neq\emptyset$ and let $e\in E_{\cap}$ be a common edge. The {\em combinatorial resultant} of $G_1$ and $G_2$ on the {\em elimination edge} $e$ is the graph $\cres{G_1}{G_2}{e}$ with vertex set $V_{\cup}$ and edge set $E_{\cup}\setminus\{e\}$.

The 2-sum appears as a special case of a combinatorial resultant when the two graphs have exactly one edge in common, which is eliminated by the operation. Circuits are closed under the $2$-sum operation, but they are not closed under this general combinatorial resultant operation: two examples are shown in \cref{fig:graphK4W4} and \cref{fig:graphW5K4}. 

\begin{figure}[ht]
	\centering
	\includegraphics[width=0.24\textwidth]{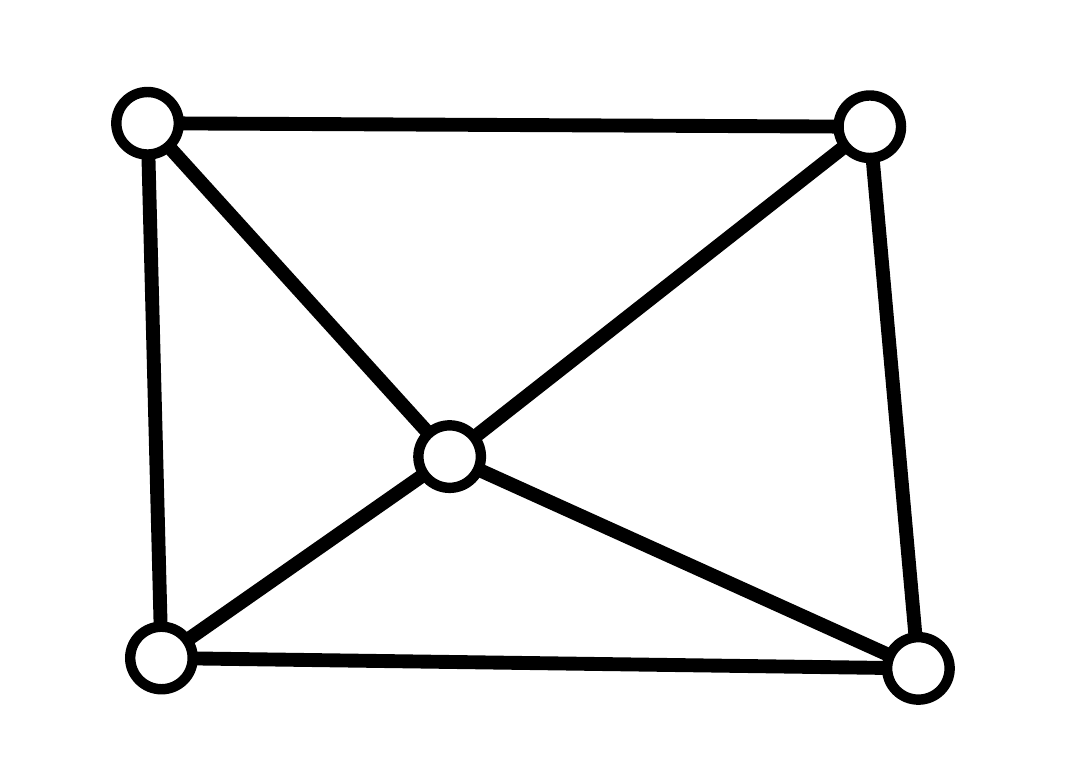}
	\includegraphics[width=0.24\textwidth]{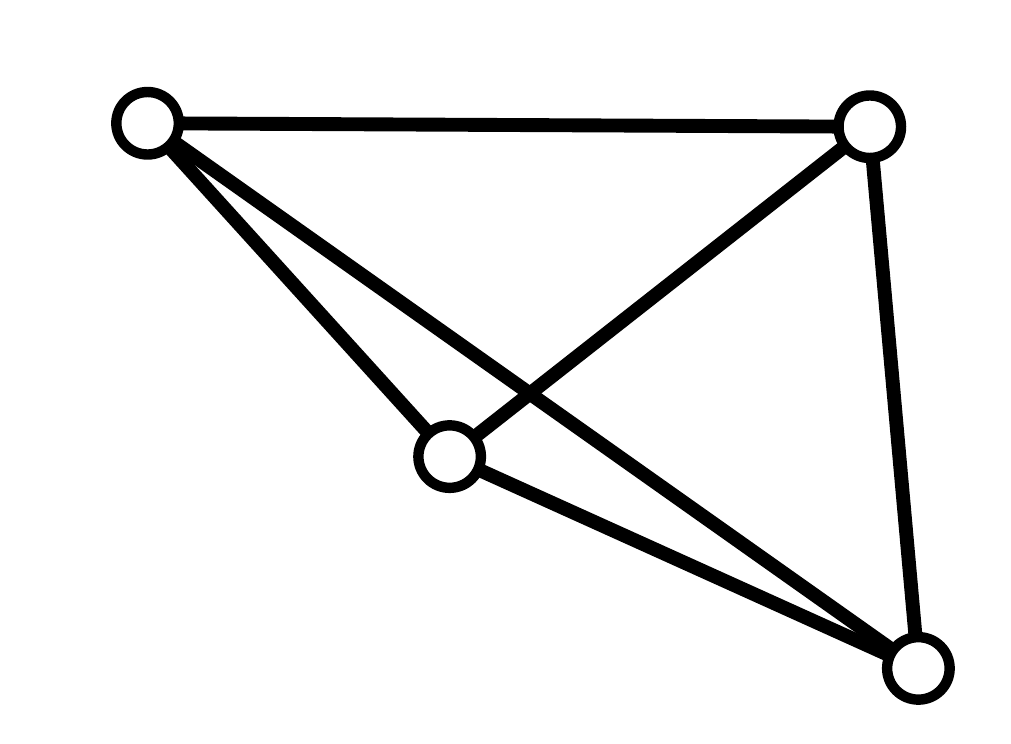}
	\includegraphics[width=0.24\textwidth]{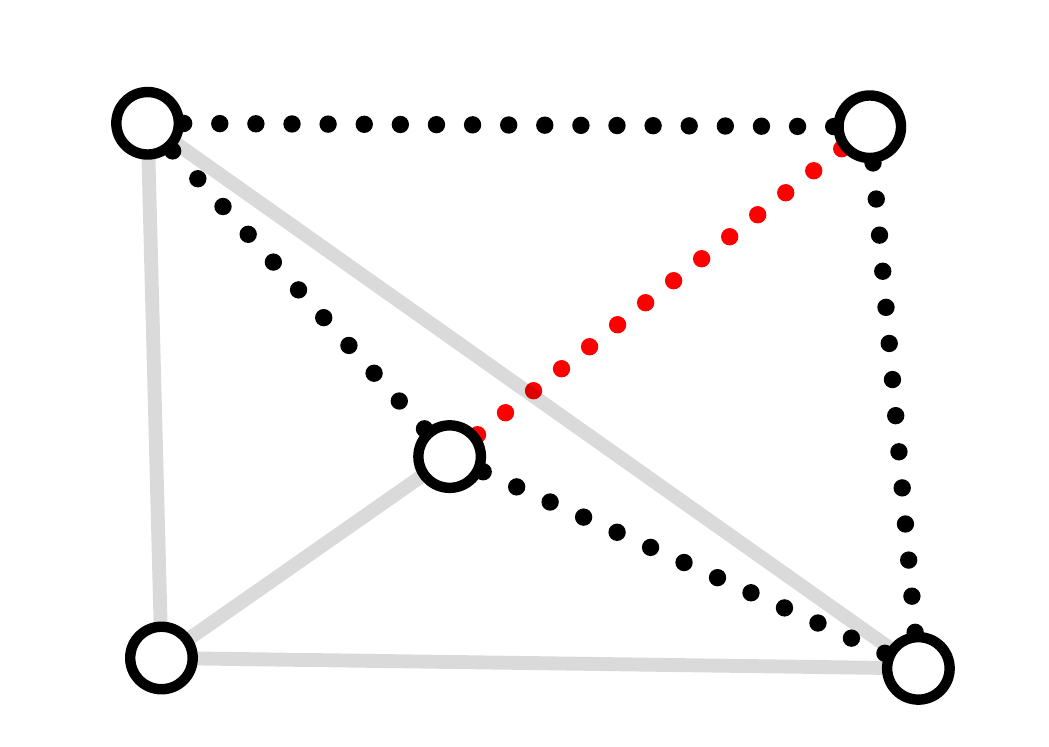}
	\includegraphics[width=0.24\textwidth]{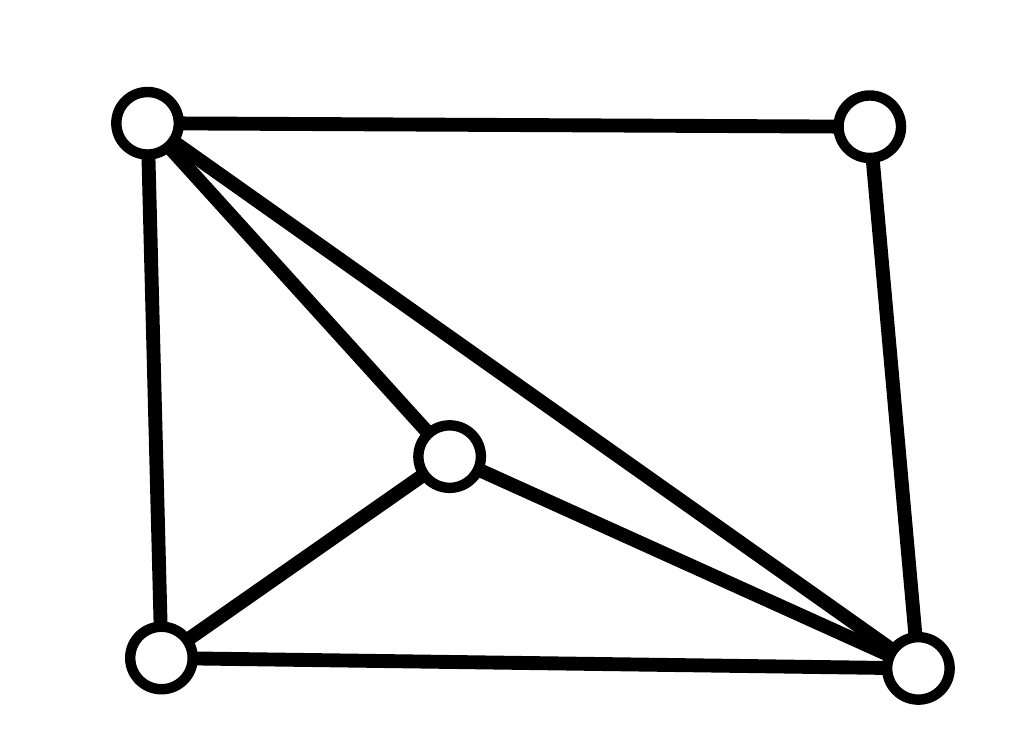}	
	\caption{A $4$-wheel $W_4$, a complete $K_4$ graph,  their common  Laman graph (dotted, with red elimination edge) and their combinatorial resultant, which is a Laman-plus-one graph but not a circuit.}
	\label{fig:graphW5K4}
\end{figure}

\subparagraph{Circuit-valid combinatorial resultants.} We are interested in combinatorial resultants that produce circuits from circuits. Towards this goal, we say that two circuits are {\em properly intersecting} if their common subgraph (of common vertices and common edges) is Laman. The example in \cref{fig:graphK4W4} is not properly intersecting, but those in \cref{fig:graphW5K4} and \cref{fig:combResultant} are.

\begin{lemma}
	\label{lem:combRes2n2}
	The combinatorial resultant of two circuits has $m=2n-2$ edges iff the common subgraph $G_{\cap}$ of the two circuits is Laman.
\end{lemma}

\begin{figure}[ht]
	\centering
	\includegraphics[width=0.24\textwidth]{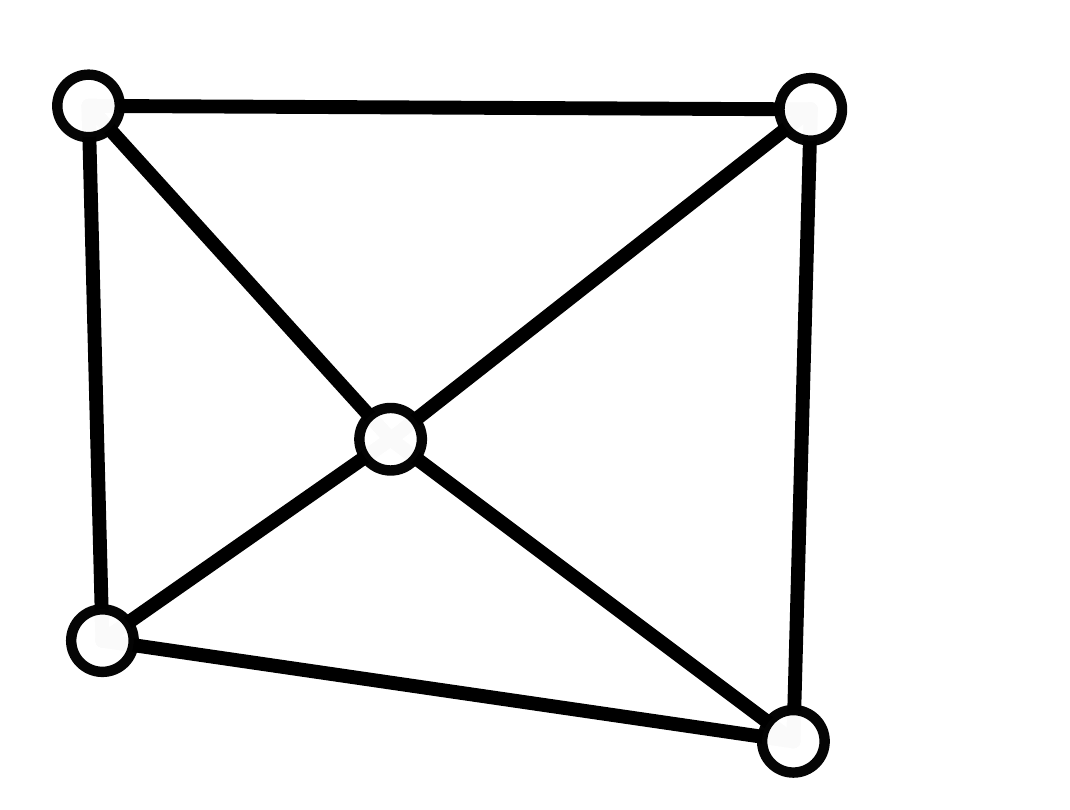}
	\hspace{-40pt}
	\includegraphics[width=0.24\textwidth]{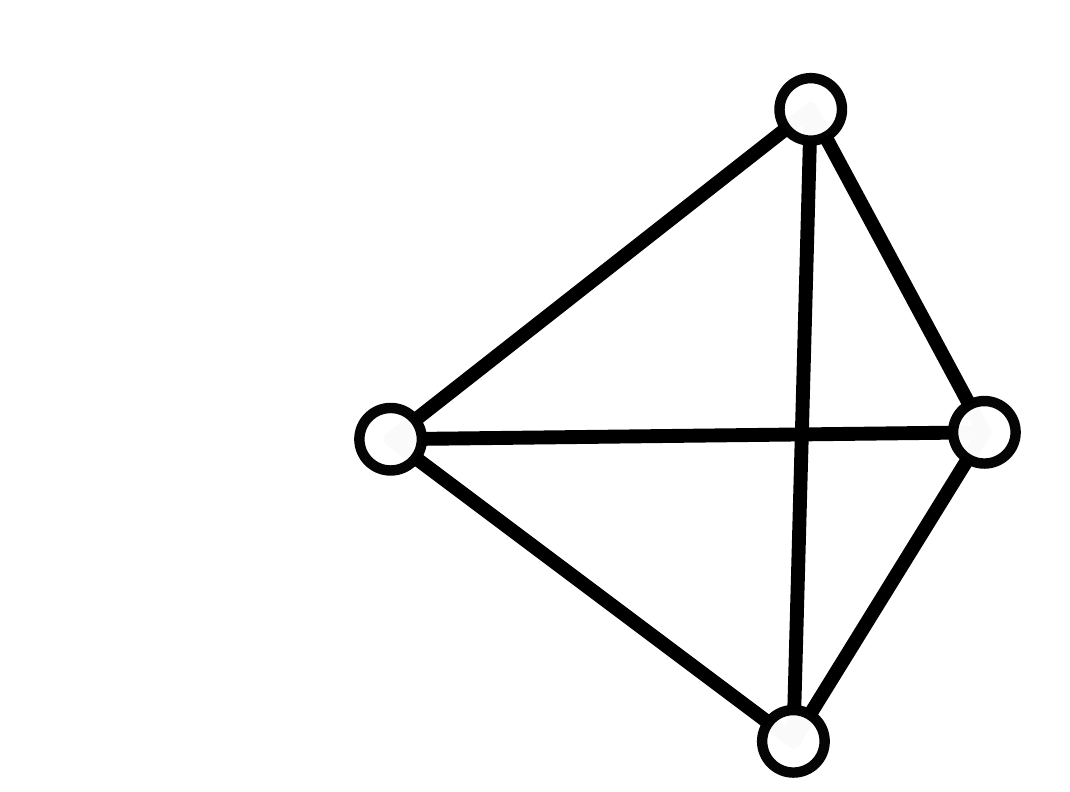}
	\includegraphics[width=0.24\textwidth]{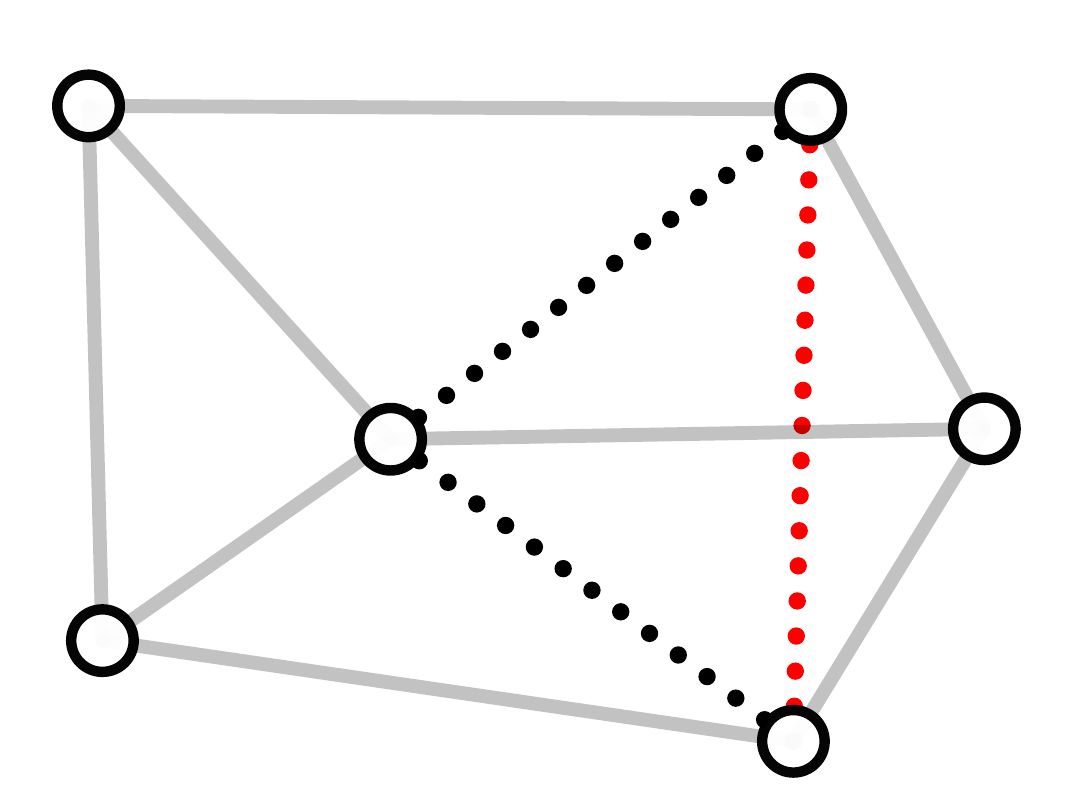}
	\includegraphics[width=0.24\textwidth]{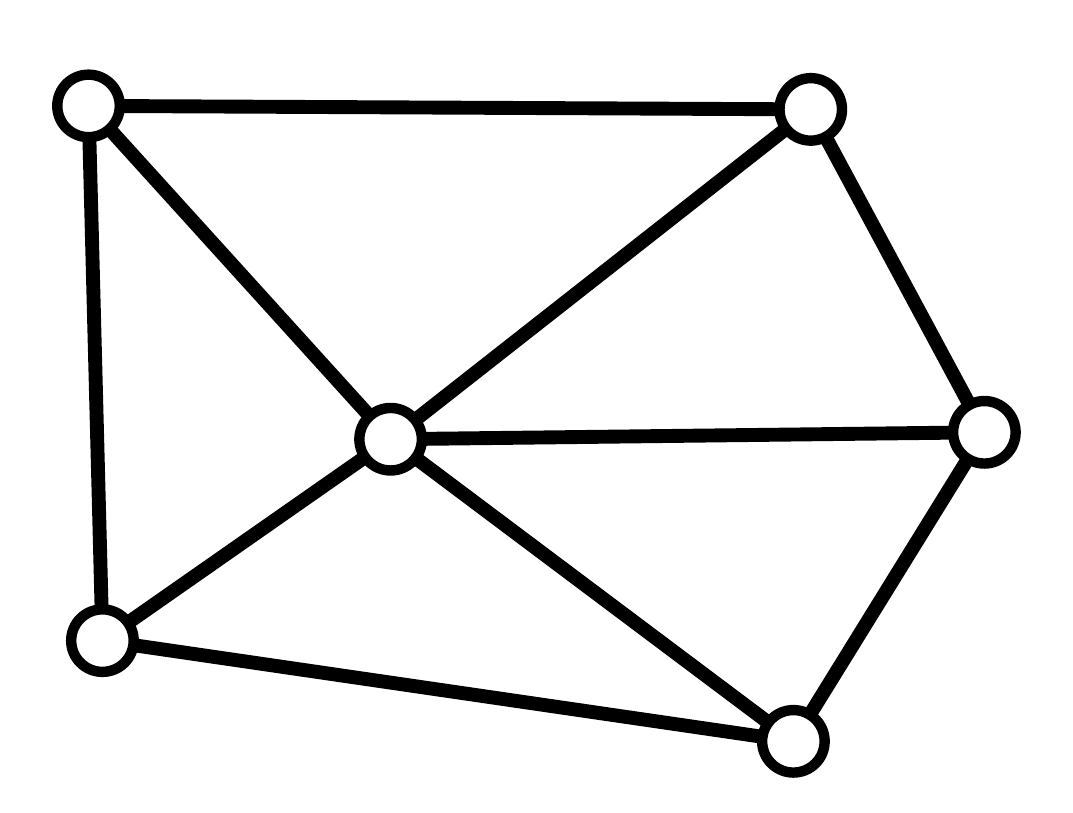}
	\caption{A $4$-wheel $W_4$ and a complete $K_4$ graph, their common  Laman graph (dotted, with red elimination edge) and their combinatorial resultant, the $5$-wheel $W_5$ circuit.}
	\label{fig:combResultant}
\end{figure}

\begin{proof} 
Let $C_1$ and $C_2$ be two circuits with $n_i$ vertices and $m_i$ edges, $i=1,2$, and let $C$ be their combinatorial resultant with $n$ vertices and $m$ edges.
By inclusion-exclusion $n=n_1+n_2-n_{\cap}$ and $ m = m_1 + m_2 - m_{\cap} - 1$. Substituting here the values for $m_1=2n_1-2$ and $m_2=2n_2-2$, we get $ m = 2 n_1 - 2 + 2 n_2 - 2 - m_{\cap} - 1 = 2(n_1 + n_2 - n_{\cap}) -2 + 2 n_{\cap} - 3 - m_{\cap} = (2n-2) + (2 n_{\cap} - 3) - m_{\cap}$. We have $m = 2 n -2$ iff $m_{\cap} = 2 n_{\cap} - 3$. Since both $C_1$ and $C_2$ are circuits, it is not possible that one edge set is included in the other: circuits are minimally dependent sets of edges and thus cannot contain other circuits. As a proper subset of both $E_1=E(C_1)$ and $E_2=E(C_2)$, $E_{\cap}$  satisfies the hereditary $(2,3)$-sparsity property. If furthermore $G_{\cap}$ has exactly $2 n_{\cap} - 3$ edges, then it is Laman.
\end{proof}

It is important to retain that the {\em common subgraph} is defined on both the common vertex and the common edge set. The following lemma allows us to sometimes consider just the graph {\em induced on the common vertex set} in the union of $G_1$ and $G_2$, when checking if two circuits are properly intersecting. This observation is applicable to the type of combinatorial resultants used from now on in this paper.
 
\begin{lemma}
	\label{lem:combResCommon}
	Let $C_1 = (V_1, E_1)$ and $C_2 =(V_2, E_2)$ be two circuits whose common vertex set $V_{\cap}$ is a strict subset of both $V_1$ and $V_2$. If the common subgraph $G_{\cap}=(V_{\cap},E_{\cap})$ is Laman, then neither $C_1$ nor $C_2$ contain additional edges (besides $E_{\cap}$) spanned by their common vertices.
\end{lemma}

\begin{proof} 
	Assume that $C_1$ contains an additional edge spanned by $V_{\cap}$. Since $(V_{\cap},E_{\cap})$ is Laman, this edge induces a circuit, entirely contained in $C_1$ and spanned by a proper subset of the vertices of $V_1$: this contradicts the fact that $C_1$ is a circuit: by the definition of a circuit, as a minimal dependent set of edges, a circuit cannot contain a subgraph that is smaller, yet dependent.
\end{proof}

\noindent
A combinatorial resultant operation applied to two properly intersecting circuits is said to be \emph{circuit-valid} if it results in a spanning circuit. An example is shown in \cref{fig:combResultant}. Being properly intersecting is a necessary condition for the combinatorial resultant of two circuits to produce a circuit, but the example in \cref{fig:graphW5K4} shows that this is not sufficient. 

\begin{problem}
Find necessary and sufficient conditions for the combinatorial resultant of two circuits to be a circuit.
\end{problem}

Our first goal is to show that each circuit can be obtained from $K_4$ circuits via a sequence of circuit-valid combinatorial resultant operations, in a manner that adds at least one new vertex at each step.

\subsection{Proof of \cref{thm:combResConstruction}.}
\label{ssec:proofThm1}

We prove now that each rigidity circuit can be obtained, inductively, by applying combinatorial resultant operations starting from $K_4$ circuits. The proof handles separately the $2$- and $3$-connected cases. In \cref{sec:prelimRigidity} we have seen that a $2$-connected circuit can be obtained from $3$-connected circuits via $2$-sums. The bulk of the proof is in the following proposition, which handles the $3$-connected circuits.

\begin{proposition}
	\label{prop:circuits3conn}
	Let $C=(V,E)$ be a $3$-connected circuit spanning $n+1\geq 5$ vertices. Then we can find two circuits $A$ and $B$ such that $A$ has $n$ vertices, $B$ has at most $n$ vertices and $C$ can be represented as the combinatorial resultant of $A$ and $B$.
\end{proposition}

\begin{proof}
We apply \cref{thm:BergJordan} to find two non-adjacent vertices $a$ and $b$ of degree 3 such that a circuit $A$ can be produced via an inverse Henneberg II operation on vertex $a$ in $C$ (see \cref{fig:invCombResA}). Let the neighbors of vertex $a$ be $N(a)=\{u,v,w\}$ such that $e=uv$ was not an edge of $C$ and is the one added to obtain the new circuit $A=(V\setminus{\{a\}},(E\setminus\{au,av,aw\})\cup\{uv\})$. 
\begin{figure}[ht]
	\centering
		\includegraphics[width=.34\textwidth]{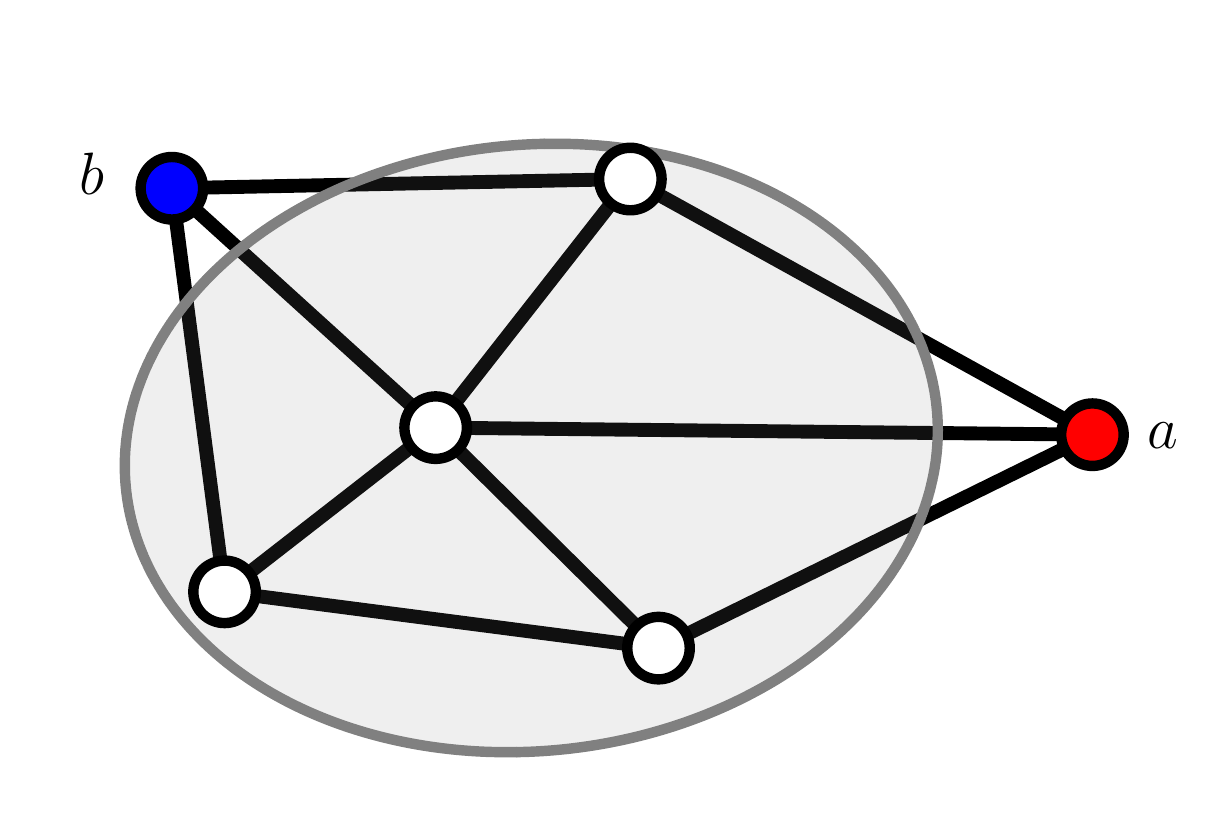}
		\includegraphics[width=.31\textwidth]{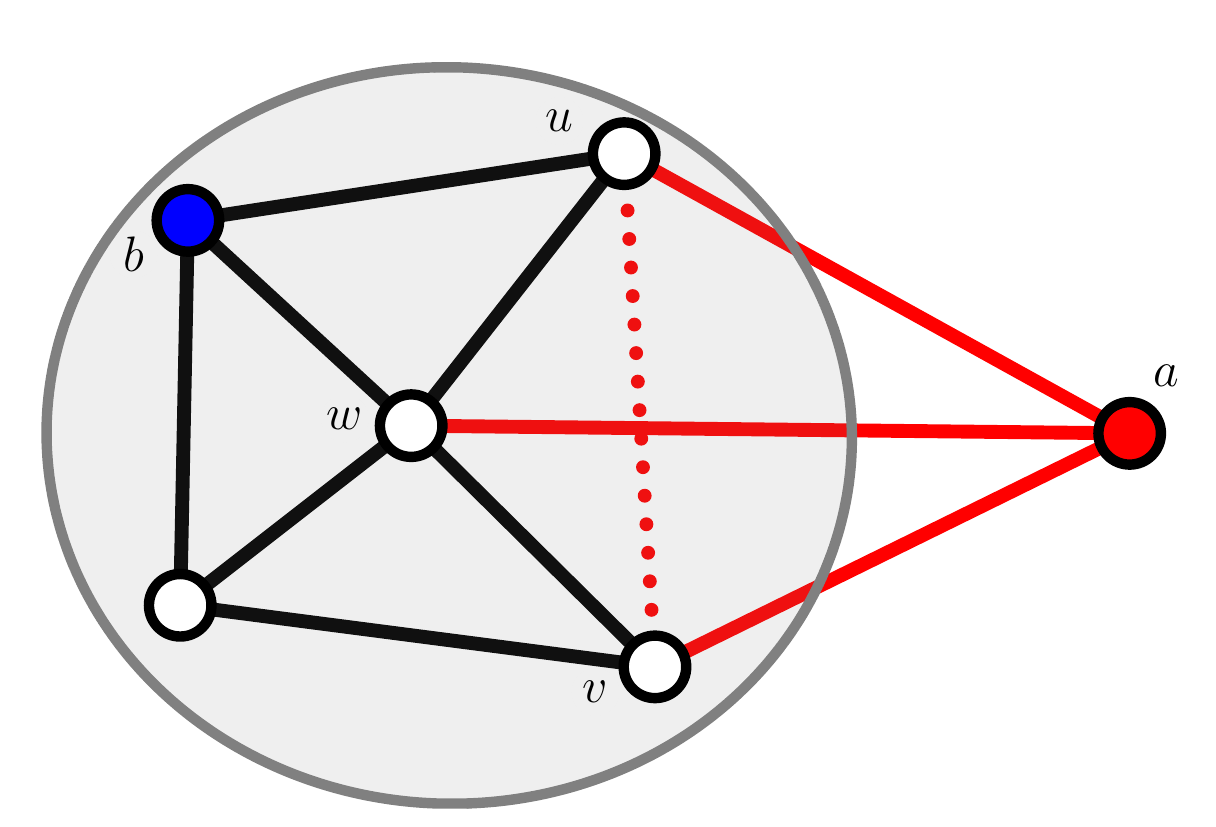}
		\includegraphics[width=.3\textwidth]{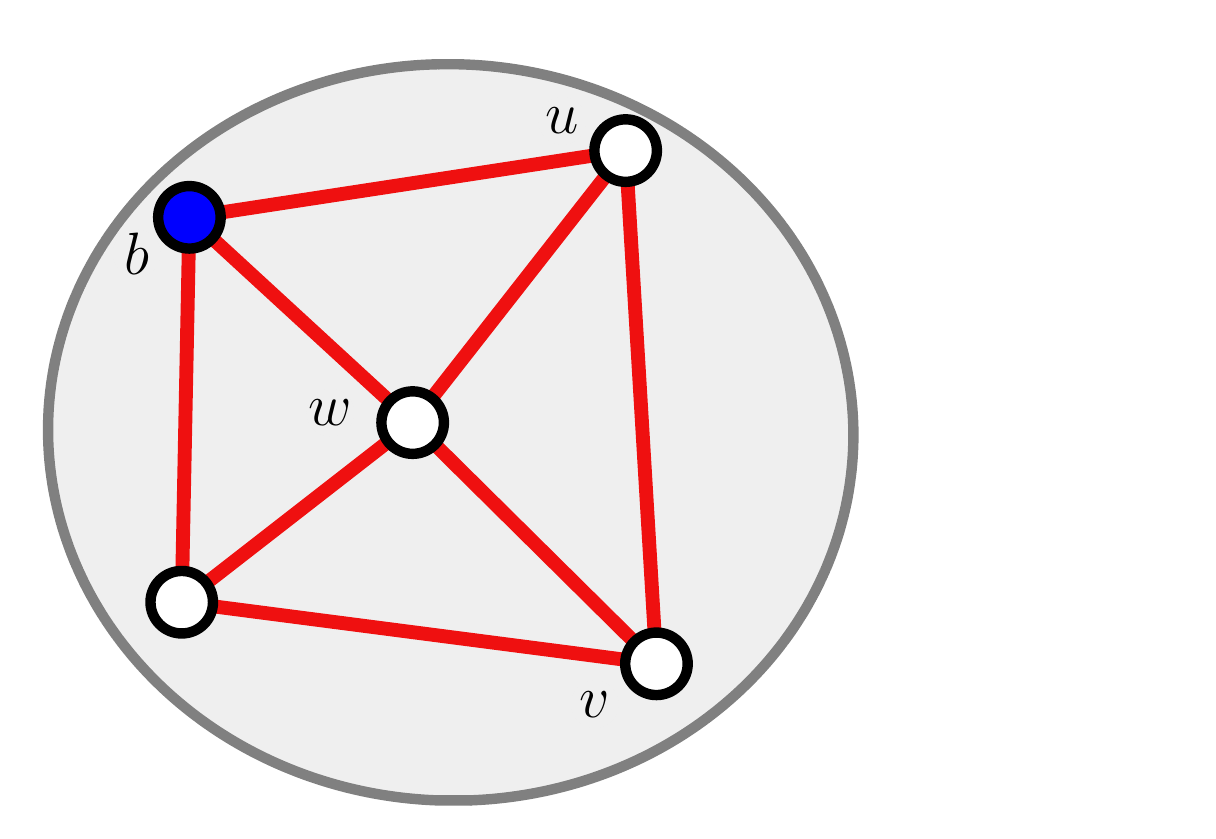}
	\caption{
	The 3-connected circuit $C$ spanning $n+1$ vertices with two non-adjacent vertices $a$ (red) and $b$ (blue) of degree 3. Note that $N(a)$ and $N(b)$ may not be disjoint. An inverse Henneberg II at $a$ removes the red edges at $a$ and adds dotted red edge $e = uv$. Circuit $A$ (red). 
	}
	\label{fig:invCombResA}
\end{figure}

To define circuit $B$, we first let $L$ be the subgraph of $C$ induced by $V\setminus\{b\}$. Simple sparsity consideration show that $L$ is a Laman graph. The graph $D$ obtained from $L$ by adding the edge $e=uv$, as in \cref{fig:invCombResB} (left), is a Laman-plus-one graph containing the three edges incident to $a$ (which are not in $A$) and the edge $e$ (which is in $A$). $D$ contains a unique circuit $B$ (\cref{fig:invCombResB} left) with edge $e\in B$ (see e.g.~\cite[Proposition 1.1.6]{Oxley:2011}). It remains to prove that $B$ contains $a$ and its three incident edges. If $B$ does not contain $a$, then it is a proper subgraph of $A$. But this contradicts the minimality of $A$ as a circuit. Therefore $a$ is a vertex in $B$, and because a vertex in a circuit can not have degree less than $3$, $B$ contains all its three incident edges.

\begin{figure}[ht]
	\centering
		\includegraphics[width=.33\textwidth]{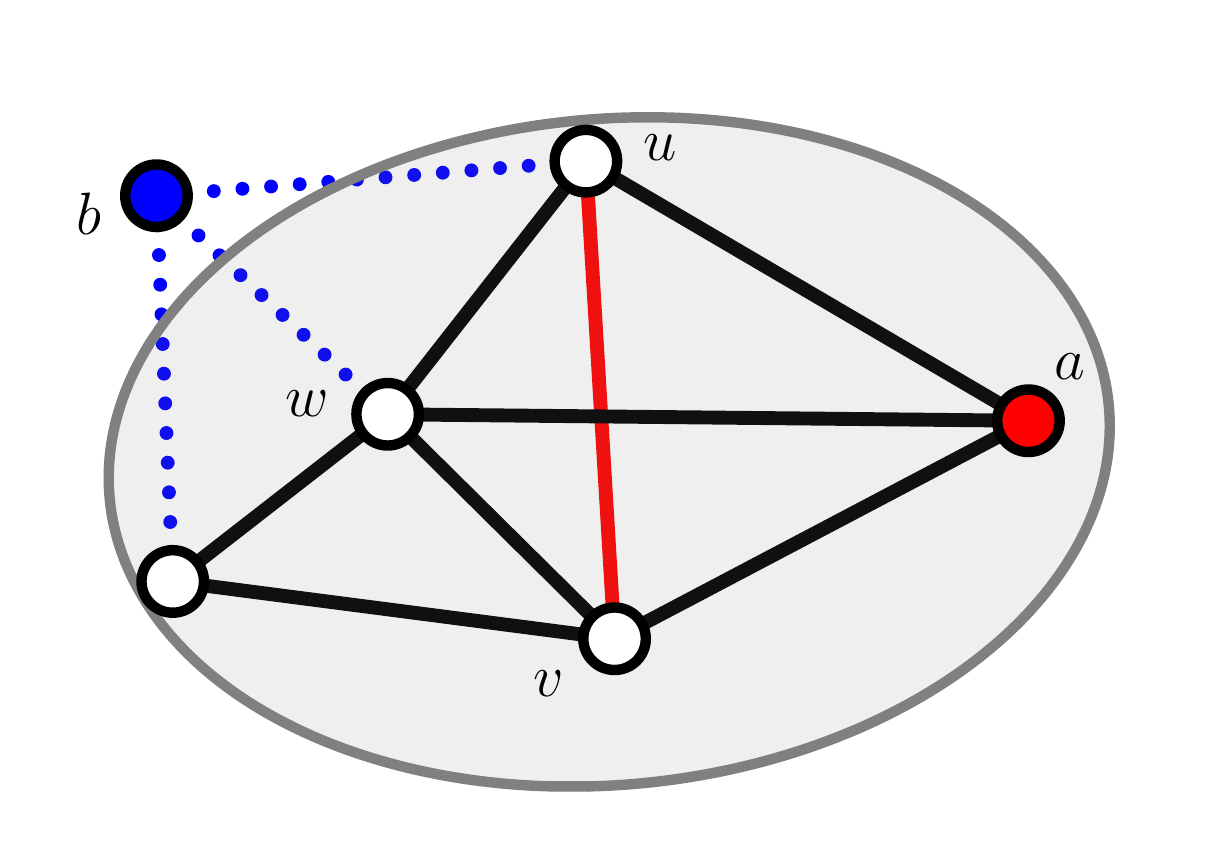}
		\includegraphics[width=.3\textwidth]{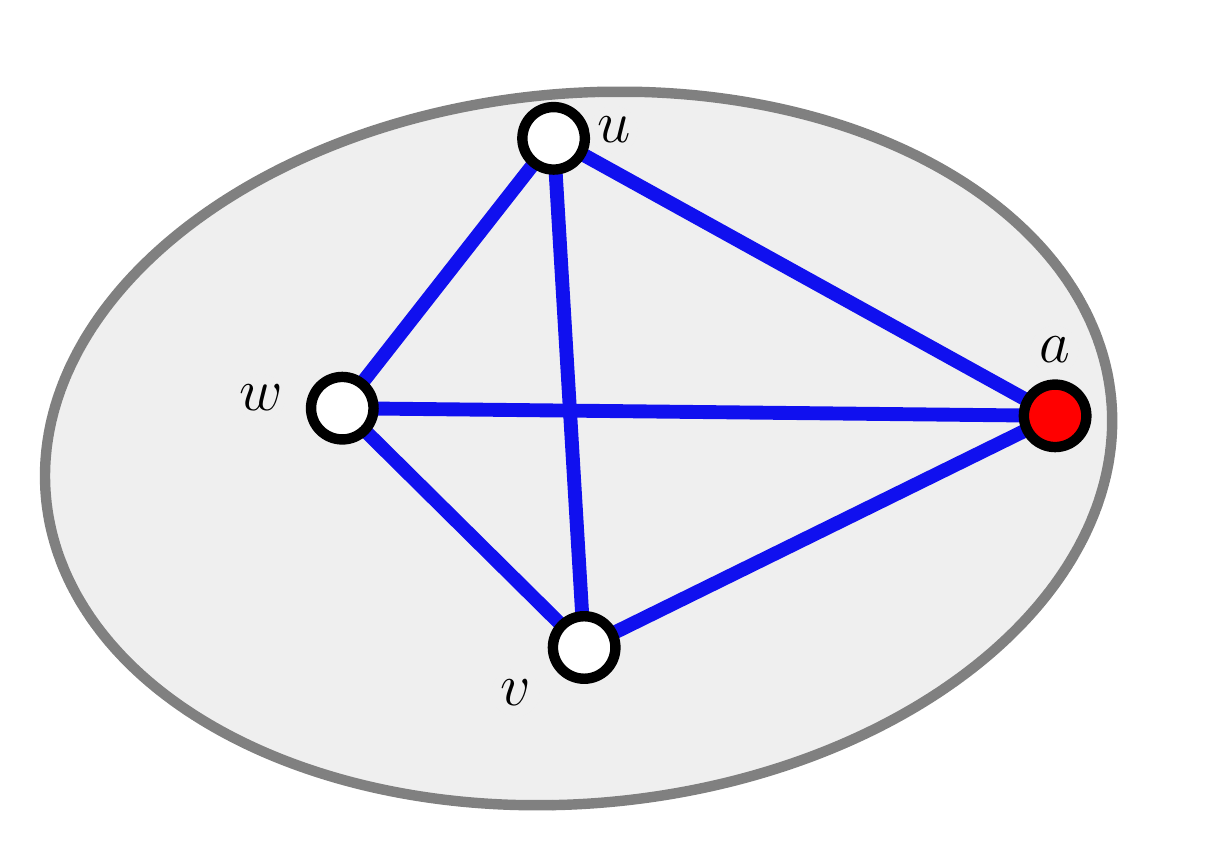}
	\caption{
	Remove from $C$ the edges from $b$ (blue dotted) and add red edge $e$. Circuit $B$ (blue).}
	\label{fig:invCombResB}
\end{figure}

The combinatorial resultant $\cres{A}{B}{e}$ of the circuits $A$ and $B$ with $e$ the eliminated edge satisfies the desired property that $C=\cres{A}{B}{e}$.
\end{proof}

\subsection{Algorithmic aspects.} \cref{alg:comb} captures the procedure described in \cref{prop:circuits3conn}. It can be applied recursively until the base case $K_4$ is attained.  Its main steps, the Inverse Henneberg II step on a circuit at line \ref{line:invH} and finding the unique circuit in a Laman-plus-one graph at line \ref{line:circuit} can be carried out in polynomial time using slight variations of the $(2,3)$ and $(2,2)$-sparsity pebble games from \cite{streinu:lee:pebbleGames:2008}.

\begin{algorithm}[ht]
	\caption{Inverse Combinatorial Resultant}
	\label{alg:comb}
	\textbf{Input}: $3$-connected circuit $C$\\
	\textbf{Output}: circuits $A$, $B$ and edge $e$ such that $C=\cres{A}{B}{e}$
	\begin{algorithmic}[1]
		\For{each vertex $a$ of degree $3$}
		\If{inverse Henneberg II is possible on $a$\\ \ \ \ \ \textbf{and} there is a non-adjacent degree $3$  vertex $b$}
		\State Get the circuit $A$ and the edge $e$ by inverse Henneberg II in $C$ on $a$\label{line:invH}
		\State Let $D = C$ without $b$ (and its edges) and with new edge $e$
		\State Compute the unique circuit $B$ in $D$ \label{line:circuit}
		\State \Return circuits $A, B$ and edge $e$
		\EndIf
		\EndFor
	\end{algorithmic}
\end{algorithm}

The algorithm faces many choices for the two degree-$3$ vertices $a$ and $b$. These choices may lead to different representations of a circuit as the combinatorial resultant of two other circuits. 

\begin{corollary}\label{cor:nonunique}
The representation of $C$ as the combinatorial resultant of two smaller circuits is in general not unique. An example is the ``double-banana'' 2-connected circuit shown in \cref{fig:exampleCircuit}.
\end{corollary}

\begin{figure}[ht]
\centering
\includegraphics[width=.3\textwidth]{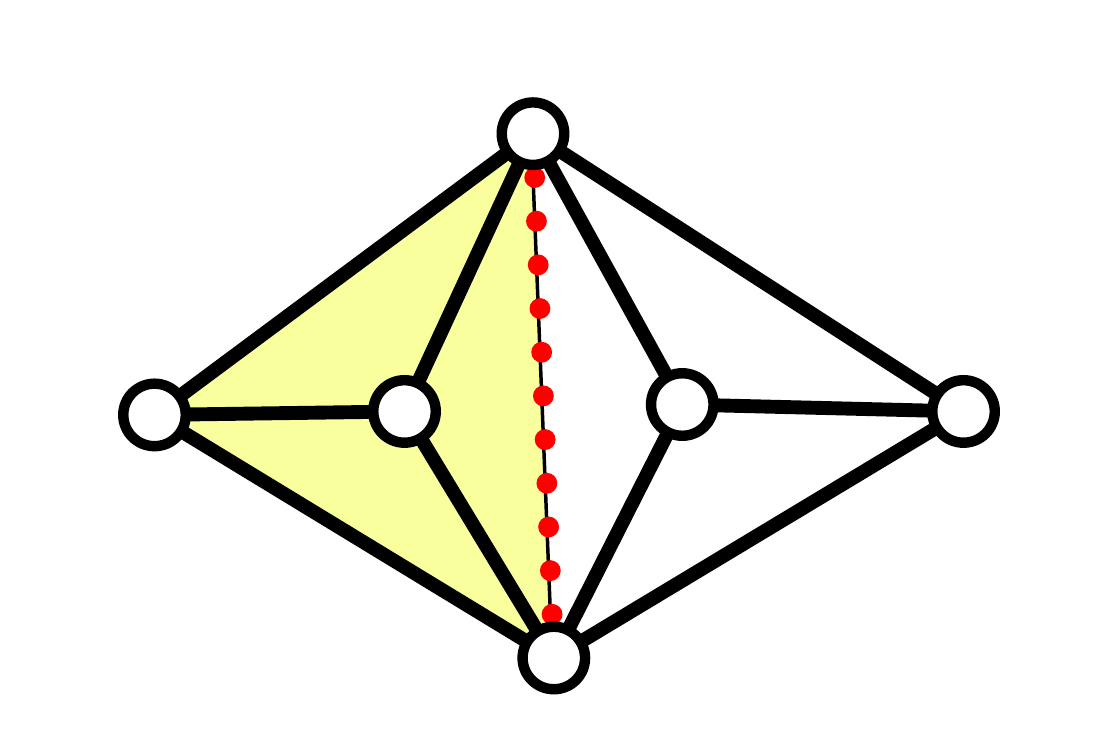}
\includegraphics[width=.3\textwidth]{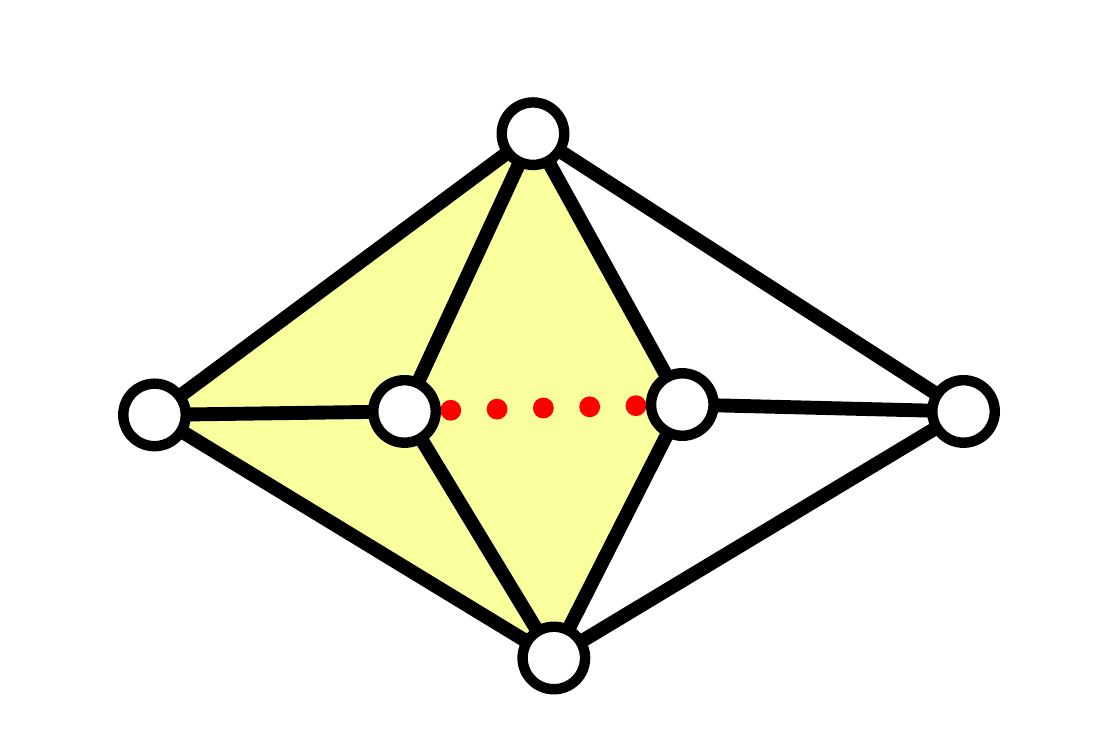}
\caption{The 2-connected {\em double-banana} circuit can be obtained as a combinatorial resultant from two $K_4$ graphs (left, $2$-sum), and from two wheels on 4 vertices sharing two triangles (right). Dashed lines indicate the eliminated edges, and in each case one of the two circuits is highlighted to distinguish $K_4$ from $W_4$.}
\label{fig:exampleCircuit}
\end{figure}

\subsection{Combinatorial Circuit Resultant (CCR) Tree.}\label{sec:resultantTree} 
Each one of the possible constructions of a circuit using combinatorial resultant operations can be represented in a {\em tree} structure. Let $C$ be a rigidity circuit with $n$ vertices. A \emph{combinatorial circuit-resultant (CCR) tree} $T_C$ for the circuit $C$ is a rooted binary tree with $C$ as its root and such that: (a) the nodes of $T_C$ are circuits; (b) circuits on level $l$ have at most $n-l$ vertices; (c) the two children $\{C_j,C_k\}$ of a parent circuit $C_i$ are such that $C_i=\cres{C_j}{C_k}{e}$, for some common edge $e$, and (d) the leaves are complete graphs on 4 vertices. An example is illustrated in \cref{fig:resTree}.

\begin{figure}[ht]
	\centering
		\includegraphics[width=.7\textwidth]{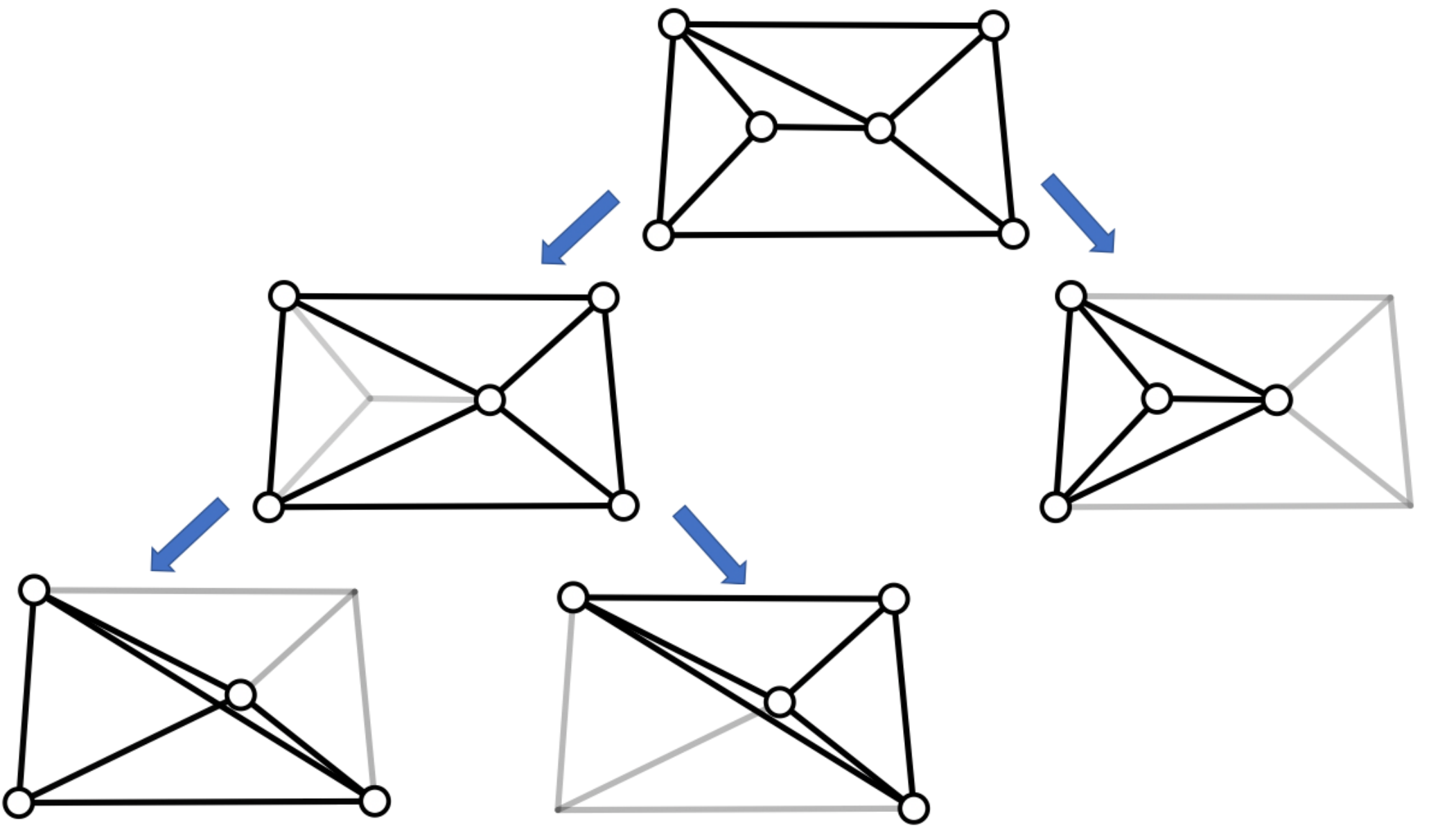}
	\caption{
	A CCR tree for the Desargues-plus-one circuit. To help the reader visualize the common Laman subgraphs and the eliminated edge at each node of the tree, the lower circuits are shown, in black and with large vertices, in the context of the combinatorial resultant circuit above them (light-greyed).}
	\label{fig:resTree}
\end{figure}

\subparagraph{Complexity of CCR trees.} If the intermediate circuits are all $3$-connected, the depth of a tree obtained by our method is $n-4$, and this is the worst possible case. The best case for depth is $\log n$ and occurs when all the intermediate circuits are $2$-connected and are split into two circuits of the same size.

In terms of size (number of nodes), the CCR tree may be, in principle, anywhere between linear to exponential in size. Best cases occur when the resultant tree is path-like, with each internal node having a $K_4$ leaf, or when the tree is balanced of depth $\log n$ and each resultant operation is a $2$-sum. Conceivably, the worst case (exponential size) could be a complete (balanced) binary tree of linear depth: each internal node at level $k$ would combine two circuits with the same number of vertices $n-k-1$ into a circuit with $n-k$ vertices. Sporadic examples of small, full height and balanced CCR trees exist (e.g.\ for $K_{33}$-plus-one), but we do not know how far they generalize.

\begin{problem}
	Are there infinite families of circuits with linear-depth, balanced CCR trees? 
\end{problem}

It would be interesting to understand the worst case size of these trees, even if families as above do not exist:

\begin{problem}
	Characterize the circuits produced by the worst-case size of the CCR tree. 
\end{problem}

Understanding the worst cases may help our Algorithm \ref{alg:comb} avoid the corresponding choices of vertices $a$ and $b$ in Steps 1--3. The goal would then be to produce the {\em best CCR tree,} or at least a good one, according to some well-defined measure of {\em CCR tree complexity}. We will return to this question in \cref{sec:resTree}.

In order to answer problems 11 and 12 one may have to do experimentation with CCR trees. However, the number of trees can be very large, which leads to the following:

\begin{problem}\label{prob:enumCCR}
Develop an efficient algorithm for enumerating CCR trees of a circuit.
\end{problem}

\begin{problem}
Compute or estimate the number of distinct CCR trees of a circuit.
\end{problem}
\section{Preliminaries: Ideals and Algebraic Matroids.}
\label{sec:prelimAlgMatroids}

We turn now to the algebraic aspects of our problem in order to introduce algebraic matroids and circuit polynomials.  We work over the field of rational numbers $\q$. In this section, the set of variables $X_n$ denotes $X_n = \{x_i: 1\leq i\leq n\}$; when we turn to the Cayley-Menger ideal, it will be $X_n = \{x_{ij}: 1\leq i < j \leq n\}$. Polynomial rings $R$ are always of the form $R=\q[X]$, over sets of variables $X\subset X_n$.  The \emph{support} $\supp f$ of a polynomial $f\in \q[X_n]$ is the set of indeterminates appearing in it. The degree of a variable $x$ in a polynomial $f$ is denoted by $\deg_xf$.

\subsection{Polynomial ideals.}
\label{sec:polyIdeals}
A set of polynomials $I \subset\q[X]$ is an \emph{ideal of} $\q[X]$ if it is closed under addition and multiplication by elements of $\q[X]$. Every ideal contains the zero ideal $\{0\}$. A {\em generating set} for an ideal is a set $S\subset \q[X]$ of polynomials such that every polynomial in the ideal is a finite algebraic combination of elements in $S$ with coefficients in $\q[X]$. {\em Hilbert's Basis Theorem} (see e.g.\ \cite{CoxLittleOshea}) guarantees that every ideal in a polynomial ring has a finite generating set. Ideals generated by a single polynomial are called \emph{principal}. An ideal $I$ is a \emph{prime} ideal if, whenever $fg\in I$, then either $f\in I$ or $g\in I$. A polynomial is {\em irreducible} (over $\q$) if it cannot be decomposed into a product of non-constant polynomials in $\q[X]$. 
A principal ideal is prime iff it is generated by an irreducible polynomial.
An ideal generated by two or more irreducible polynomials is not necessarily prime.
The \emph{dimension} $\dim I$ of an ideal $I$ of $\q[X]$ is the cardinality of the maximal subset $S\subseteq X$ with the property $I\cap \q[S]=\{0\}$.

Let $I$ be an ideal of $\mathbb Q[X_n]$ and $X'\subset X_n$ non-empty. The \emph{elimination ideal} of $I$ with respect to $X'$ is the ideal $I\cap \mathbb Q[X']$ of the ring $\mathbb Q[X']$.
Elimination ideals frequently appear in the context of \grobner{} bases \cite{Buchberger, CoxLittleOshea} which give a general approach for computing elimination ideals: if $\mathcal G$ is a \grobner{} basis for $I$ with respect to an \emph{elimination order} (see Exercises 5 and 6 in \S1 of Chapter 3 in \cite{CoxLittleOshea}), e.g.\ the lexicographic order with $x_{i_1}>x_{i_2}>\dots>x_{i_n}$, then the elimination ideal $I\cap \q[x_{i_{k+1}},\dots,x_{i_n}]$ which eliminates the first $k$ indeterminates from $I$ in the specified order has $\mathcal G\cap \q[x_{i_{k+1}},\dots,x_{i_n}]$ as its \grobner{} basis.

\subsection{Algebraic independence and algebraic matroids.}
\label{subsection:AlgDepMat}

Recall that a set of vectors in a vector space is linearly dependent if there is a non-trivial linear relationship between them. Similarly, given a finite collection $A$ of complex numbers, we say that $A$ is \emph{algebraically dependent} if there is a non-trivial polynomial relationship between the numbers in $A$.

\begin{definition} 
	Let $k$ be a field (e.g.\ $k=\q$) and $k\subset F$ a field extension of $k$. A finite subset $A=\{\alpha_1,\dots,\alpha_n\}$ of $F$ is said to be {\em algebraically dependent over $k$} if there is a non-zero (multivariate) polynomial with coefficients in $k$ vanishing on $A$. Otherwise, we say that $A$ is algebraically independent over $k$.
\end{definition}

It was noticed by van der Waerden that the algebraically independent subsets $A$ of a finite subset $E$ of $F$ satisfy matroid axioms \cite{VDWmoderne,VDW2} and therefore define a matroid.

\begin{definition}
	\label{def:algmat}
	Let $k$ be a field and $k\subset F$ a field extension of $k$. Let $E=\{\alpha_1,\dots,\alpha_n\}$ be a finite subset of $F$. The {\em algebraic matroid on $E$ over $k$} is the matroid whose independent sets are the algebraically independent (over $k$) subsets of $E$.
\end{definition}

\subsection{Algebraic matroid of a prime ideal.}
\label{subsection:AlgMatPrime} 
An equivalent definition of algebraic matroids, in terms of polynomial ideals, is more useful for the purposes of this paper. Intuitively, a collection of variables is \emph{independent} with respect to an ideal $I$ if it is not constrained by any polynomial in $I$, and {\em dependent} otherwise. The \emph{algebraic matroid}  induced by the ideal is, informally, a matroid on the ground set of variables $X_n$ whose independent sets are subsets of variables that are {\em not} supported by any polynomial in the ideal.  Its \emph{dependent sets} are supports of polynomials in the ideal.

\begin{definition}
	Let $I$ be a prime ideal in the polynomial ring $\q[X_n]$. The \emph{algebraic matroid of} $I$, denoted $\amat I$, is the matroid $(X_n,\mathcal I)$ whose independent sets are:
	\[\mathcal I=\{X\subseteq X_n\mid I\cap \q[X]=\{0\}\}.\]
\end{definition}

\subsection{Equivalence of the definitions.}
\label{subsection:equivalence} 
It is well known that every algebraic matroid of a prime ideal $I$ arises as an algebraic matroid of a field extension in the sense of \cref{def:algmat}, and vice-versa. For completeness, we include a proof.

\subparagraph{From a field extension to a prime ideal.} 
Let $E=\{\alpha_1,\dots,\alpha_n\}$ be a set of elements in a field extension of $\q$ and let $\mathcal M$ be the algebraic matroid on $E$ over $\q$ whose dependent sets are algebraically dependent subsets $A\subset E$. 
To realize $\mathcal M$ as an algebraic matroid of a prime ideal $I$ of $\q[X_n]$, we define $I:=\ker\varphi$ as the kernel of the homomorphism $\varphi\colon \q[X_n]\to \q(\alpha_1,\dots,\alpha_n)$ mapping 
$x_i\mapsto\alpha_i$ for $i\in\{1,\dots,n\}$ and $a\mapsto a$ for $a\in \q$. Kernels of homomorphisms are known to be prime ideals \cite{Lang}. The kernel $\ker\varphi$ is non-zero, since any polynomial in $\ker\varphi$ defines a dependency in $\mathcal M$, and any dependent set $A\subset \{\alpha_1,\dots,\alpha_n\}$ in $\mathcal M$ vanishes on a polynomial in $\q[X_n]$. Let $\q[X_A]$ be the ring of polynomials supported on subsets of $X_A :=\varphi^{-1}(A)$. We have
$\ker\varphi\cap \q[X_A]\neq\{0\}$
if and only if $A$ is a dependent set of $\mathcal M$. Hence $\varphi$ induces an isomorphism between dependent sets in the matroid induced by $\ker\varphi$ and $\mathcal M$. 

\subparagraph{From a prime ideal to a field extension.} Let $I$ be a prime ideal in $\q[X_n]$. We construct a finite field extension $F$ and a subset $\{\overline{x_1},\dots,\overline{x_n}\}\in F$ via an isomorphism that takes sets $X\subset X_n$ that are in/dependent in the ideal $I$ to algebraically in/dependent sets $\overline{X} \subset \{\overline{x_1},\dots,\overline{x_n}\}$.
The quotient ring $\q[X_n]/I$ is an integral domain with a well defined fraction field $F=\ffield{(\q[X_n]/I)}$ which contains $\q$ as a subfield. The image of $X_n$ under the canonical injections $\q[X_n]\hookrightarrow \q[X_n]/I\hookrightarrow\ffield{(\q[X_n]/I)}=F$ is the subset $\{\overline{x_1},\dots,\overline{x_n}\}$ of $F$, where $\overline{x_j}$ denotes the equivalence class of $x_j$ in both $\q[X_n]/I$ and $F$.

Let $X$ be a non-empty subset of $X_n$ (taken wlog to be $X=\{x_1,\dots,x_i\}$) and let $\overline X=\{\overline{x_1},\dots,\overline{x_i}\}$ in $F$ be its image under the canonical injections. The set $\overline X$ is by definition algebraically dependent over $\q$ if and only if there exists a non-zero polynomial $f\in \q[x_1,\dots,x_i]$ vanishing on $\overline X$, i.e.\ $f(\overline{x_1},\dots,\overline{x_i})=\overline 0$. This happens if and only if $f(x_1,\dots,x_i)\in I$, that is if and only if $I\cap \q[X]\neq\{0\}$. Similarly, $\overline X$ is algebraically independent over $\q$ if and only if $I\cap \q[X]=\{0\}$.

\medskip
We are now ready to define the core algebraic concept underlying this paper.

\subsection{Circuits and circuit polynomials.}
\label{sec:circuit polys}
A {\em circuit} in a matroid is a minimal dependent set. In an algebraic matroid, a circuit $C\subset X_n$ is a minimal set of variables supported by a polynomial in the prime ideal $I$ defining the matroid. An irreducible polynomial whose support is a circuit $C$ is called a {\em circuit polynomial} and is denoted by $p_C$. A theorem of Dress and Lovasz \cite{DressLovasz} states that,
up to multiplication by a constant, {\em a circuit polynomial $p_C$ is the unique irreducible polynomial in the ideal with the given support $C\subset X_n$}. We'll just say, shortly, that it is {\em unique}. 

We retain the following property, stating that {\em circuit polynomials generate elimination ideals supported on circuits.}

\begin{theorem}[\cite{rosen:sidman:theran:algebraicMatroidsAction:2020}, Theorem 11]
	\label{thm:circuitPolyPrincipalIdeal}
	Let $I$ be a prime ideal in $\q[X_n]$ and $C\subset X_n$ a circuit of the algebraic matroid $\amat I$. The ideal $I\cap\q[C]$ is principal, prime and generated by the circuit polynomial $p_C$.
\end{theorem}
\section{The Cayley-Menger ideal.}
\label{sec:prelimCMideal}

In this section we introduce the 2D Cayley-Menger ideal $\cm n$. We will show\footnote{This equivalence is well-known, however we were not able to track down an original reference, and include a proof for completeness. }  that its algebraic matroid is isomorphic to the $(2,3)$-sparsity matroid $\smat n$. As a consequence, we get a full combinatorial characterization of the supports of circuit polynomials in the Cayley-Menger ideal: they are in one-to-one correspondence with the rigidity circuits introduced in \cref{sec:prelimRigidity}.

Throughout this section and later, when working with the Cayley-Menger ideal, we use variables $X_n = \{x_{i,j}: 1\leq i<j\leq n\}$ for unknown squared distances between pairs of points. 

\subsection{The Cayley-Menger  ideal and its algebraic matroid.}
The {\em distance matrix} of $n$ labeled points is the matrix of squared distances between pairs of points. The {\em Cayley matrix} is the distance matrix bordered by a new row and column of 1's, with zeros on the diagonal: 

\begin{center}
$$
\begin{pmatrix}
	0 & 1 & 1 & 1 & \cdots & 1\\
	1 & 0 & x_{1,2} & x_{1,3} & \cdots & x_{1,n}\\
	1 & x_{1,2} & 0 & x_{2,3} & \cdots & x_{2,n}\\
	1 & x_{1,3} & x_{2,3} & 0 & \cdots & x_{3,n}\\
	\vdots & \vdots & \vdots &\vdots &\ddots &\vdots\\
	1 & x_{1,n} & x_{2,n} & x_{3,n} & \cdots & 0
\end{pmatrix}
$$
\end{center}

\noindent
Cayley's Theorem says that, if the distances come from a point set in the Euclidean space $\reals^d$, then the rank of this matrix must be at most $d+2$. Thus all the $(d+3)\times(d+3)$ minors of the Cayley matrix should be zero. An additional condition, due to Menger \cite{menger:newfoundation} (see also \cite{blumenthal:distGeo,crippen:havel:distanceGeoMolConf}), guarantees that the entries in a Cayley matrix correspond to actual squared distances between $n$ points in $\mathbb R^d$. Menger's condition states that all $m\times m$ minors containing $m-1$ points have the sign $(-1)^{m-1}$ or be zero, for $m\leq d+2$. For our purposes, we will make use only of Cayley's but not Menger's condition.

The set of all $(d+3)\times(d+3)$ minors of the Cayley matrix, each minor inducing a polynomial in $\q[X_n]$, constitutes a generating set for the $(n,d)$-Cayley-Menger ideal $\cm{n}^d$. These generators are {\em homogeneous polynomials} with integer coefficients \emph{irreducible} over $\q$, and will be discussed in more detail in \cref{sec:generatorsCM}. The $(n,d)$-Cayley-Menger ideal is a {\em prime ideal} of dimension $dn-{\binom{d+1}{2}}$ \cite{borcea:cayleyMengerVariety:2002, Giambelli, HarrisTu, JozefiakLascouxPragacz} and codimension $\binom{n}{2}-dn+{\binom{d+1}{2}}$.

As defined in Section \ref{sec:prelimAlgMatroids}, the algebraic matroid $\amat{\cm{n}^d}$ of the Cayley-Menger ideal is the matroid on the ground set $X_n=\{x_{i,j}\mid 1\leq i<j\leq n\}$ where a subset of distance variables $X\subseteq X_n$ is independent if $\cm{n}^d\cap~\q[X]=\{0\}$, i.e. $X$ supports no polynomial in the ideal.

As an immediate consequence of the definition of dimension of an ideal in a ring of polynomials (\cref{sec:polyIdeals}), we obtain:

\begin{proposition}The rank of $\amat{\cm{n}^d}$ is equal to $\dim{\cm{n}^d}=dn-{\binom{d+1}{2}}$.
\end{proposition}
	
\subsection{Equivalence of the $(2,3)$-sparsity matroid and the algebraic matroid of $\cm{n}$.}
\label{ssec:equivSparsityAlgebraic}
From now on\footnote{This section is included for completeness and can be skipped. }, we work only with the {\em 2D Cayley-Menger ideal} $\cm n:= \cm{n}^2$, generated by the $5\times 5$ minors of the Cayley matrix, and its algebraic matroid, denoted by $\amat {\cm n}$. In this case, the rank of the algebraic matroid is precisely the rank of the $(2,3)$-sparsity matroid $\smat n$ on $n$ vertices, introduced in \cref{sec:prelimRigidity}. We establish the equivalence of the two matroids by proving that both are isomorphic to the $2$-dimensional generic linear rigidity matroid that we now introduce.

\subparagraph{2D linear rigidity matroids.} Let $G=(V,E)$ be a graph and $(G,p)$ a $2D$ bar-and-joint framework on points $\{p_1,\dots,p_n\}\subset\mathbb R^2$.

The {\em rigidity matrix} $R_{(G,p)}$ (or just $R_G$ when there is no possibility of confusion) of the bar-and-joint framework $(G,p)$ is the $|E|\times 2 n$ matrix with {\em pairs of columns} indexed by the vertices $\{1,2,\dots,n\}$ and rows indexed by the edges $ij\in E$ with $i<j$. The $i$-th entry in the row $ij$ is $p_i-p_j$ ($2$ coordinates), the $j$-th entry is $p_j-p_i$, and all other entries are $0$.

The rigidity matrix is defined up to an order of the vertices and the edges; to eliminate this ambiguity we fix the order on the vertices as $1<2<\cdots<n$ and we order the edges $ij$ with $i<j$ lexicographically. For example, let $G=K_4$. Then the rows are ordered as $12$, $13$, $14$, $23$, $24$ and $34$ and the corresponding rigidity matrix $R_{K_4}$ is given by
\[R_{K_4}=\begin{pmatrix}
	p_1-p_2 & p_2 - p_1 & 0 & 0\\
	p_1-p_3 & 0 & p_3-p_1 & 0\\
	p_1-p_4 & 0 & 0 & p_4 - p_1\\
	0 & p_2-p_3 & p_3-p_2 & 0\\
	0 & p_2-p_4 & 0 & p_4 - p_2\\
	0 & 0 & p_3-p_4 & p_4-p_3\\
\end{pmatrix}.\]

The {\em linear matroid} associated to a matrix is defined on the ground set given by its rows. An {\em independent set} is a linearly independent collection of rows.

The 2D  \emph{linear rigidity matroid $\lmat{(G,p)}$ induced by a framework $(G,p)$} is the linear matroid associated to the rigidity matrix of the framework.
Note that it depends not just on $G$ but also on the plane configuration $p$. For example, if $G=K_4$, $p$ is a configuration in which at most two vertices of $K_4$ are on a line, and $q$ is a configuration in which the vertices $\{2,3,4\}$ are on the same line, then $\rank\lmat{(K_4,p)}>\rank\lmat{(K_4,q)}$.

The \emph{2D linear rigidity matroid} $\lmat{p}$ is the linear matroid associated to the rigidity matrix of a complete graph framework $(K_n,p)$. 

\subparagraph{Genericity.} Let $G$ be a graph and consider the set of all possible plane configurations $p$ for $G$. We say that a 2D bar-and-joint framework $(G,p)$ is \emph{generic} if the rank of the row space of $R_{(G,p)}$ is maximal among all these configurations. If $p$ and $p'$ are distinct generic plane configurations for a graph $G$, the 2D linear matroids $\lmat{(G,p)}$ and $\lmat{(G,p')}$ are isomorphic \cite[Theorem 2.2.1]{graverServatiusServatius}. Hence we can define the \emph{2D generic linear matroid} $\lmat{G}$ as the 2D linear matroid $\lmat{(G,p)}$ for a generic plane configuration $p$.

An alternative viewpoint \cite{sliderPinning:streinu:theran:2010} is to work with coordinate indeterminates $p_i = \{ x_i, y_i\}, i\in [n]$ over the set of variables $X_n\cup Y_n$. We define the {\em generic rigidity matrix} as having entries in these variables. The generic rigidity matrix has rank at least $r$ if there exists an 
$r\times r$ minor which, as a polynomial in $\q[X_n\cup Y_n]$, is not {\em identically zero.} An alternative proof of \cref{thm:laman} given in \cite{sliderPinning:streinu:theran:2010} shows that maximal independent sets of rows in the generic rigidity matrix of $K_n$ correspond to Laman graphs on $n$ vertices. The maximal minors of the generic rigidity matrix of a Laman graph vanish on a measure-zero set of points, and all points in the complement of the vanishing locus are said to be {\em generic} for the given Laman graph.  

\subparagraph{The equivalence between the algebraic Cayley-Menger and the sparsity matroids.} We are now ready to prove:

\begin{theorem}
	\label{theorem:algSparseIsomorphic}
	The algebraic matroid $\amat{\cm{n}}$ of the 2D Cayley-Menger ideal and the $(2,3)$-sparsity matroid $\smat n$ are isomorphic.
\end{theorem}
\begin{proof}
It follows from \cref{thm:laman} that, for a given graph $G$ on $n$ vertices, the generic linear matroid $\lmat{(G,p)}$ and the $(2,3)$-sparsity matroid $\smat n$ are isomorphic. It remains to show that the algebraic matroid $\amat{\cm{n}}$ is equivalent to the generic linear rigidity matroid $\lmat{K_{n}}$. 

This equivalence is a consequence of a classical result of Ingleton \cite[Section 6]{Ingleton} (see also \cite[Section 2]{EhrenborgRota}) stating that algebraic matroids over a field of characteristic zero are linearly representable over an extension of the field, with the linear representation given by the Jacobian. We now note that the Cayley-Menger variety is realized as the Zariski closure of the image of the map $f=(f_{ij})_{\{i,j\}\in{\binom{n}{2}}}\colon(\mathbb C^2)^n \to \mathbb C^{\binom{n}{2}}$ given by the edge function:
\[(p_1,\dots,p_n)\mapsto(||p_i-p_j||^2)_{\{i,j\}\in{\binom{n}{2}}}.\]
The Jacobian of the edge function at a generic point in $(\mathbb C^2)^n$ is precisely the matrix $2R_{(K_n,p)}$ for a generic configuration $p$ of the complete graph.
\end{proof}

From now on, we will use the isomorphism to move freely between the formulation of algebraic circuits as subsets of variables $X\subset X_n$ and their graph-theoretic interpretation as graphs that are rigidity circuits.

\subparagraph{Comment: beyond dimension 2?}
Note that the $d$-dimensional linear rigidity matroid $\lmat n$ and the algebraic matroid $\amat{\cm{n}^d}$ of the $(n,d)$-Cayley-Menger matroid are isomorphic by the same Jacobian argument as above. However, the equivalence between the 2D sparsity matroid $\smat n$ and $\amat{\cm{n}}$ does not extend, in higher dimensions, to some known graphical matroid. The generalization $dn - \binom{d+1}{2}$  of the $(2n-3)$-sparsity condition from dimension $2$ to dimension $d$, called Maxwell's sparsity \cite{maxwell:equilibrium:1864}, does not satisfy matroid axioms, and is known to be only a necessary but not sufficient condition for minimal rigidity in dimensions $d\geq 3$.
\section{Preliminaries: Resultants.}
\label{sec:prelimResultants}

In this section we review known concepts and facts about resultants; in the next section we specialize this setup to the Cayley-Menger ideal. In \cref{sec:algResCircuits}, in order to prove \cref{thm:circPolyConstruction}, we will use the resultant of two circuit polynomials in the Cayley-Menger ideal as the algebraic counterpart of the combinatorial resultant operation which deletes a common edge $e$ of two circuits.

\subparagraph{Resultants.} The resultant can be introduced in several equivalent ways \cite{GelfandKapranovZelevinsky}. Here we use its definition as the determinant of the Sylvester matrix.

Let $f,g\in R[x]$ be two polynomials in $x$ with coefficients in some ring of polynomials $R$, with $\deg_xf=r$ and $\deg_xg=s$,  such that at least one of $r$ or $s$ is non-zero and let
\begin{align*}
 f(x)&=a_{r}x^{r}+\cdots+a_1 x + a_0,\\
 g(x)&=b_{s}x^{s}+\cdots+b_1 x + b_0. 
\end{align*}

The \emph{resultant} of $f$ and $g$ with respect to the indeterminate $x$, denoted $\res{f}{g}{x}$, is the determinant of the $(r+s)\times(r+s)$ Sylvester matrix made from the coefficients of $f$ and $g$ arranged in staggered rows according to the following pattern:
\[
\operatorname{Syl}(f,g,x)=\begin{pmatrix}
a_{r} & a_{r-1} & a_{r-2} & \cdots & a_0 & 0 & 0 & \cdots & 0\\
0 & a_r & a_{r-1} & \cdots & a_1 & a_0 & 0 & \cdots & 0\\
0 & 0 & a_r & \cdots & a_2 & a_1 & a_0 & \cdots & 0\\
\vdots & \vdots & \vdots & \ddots & \vdots & \vdots & \vdots & \ddots & 0\\
0 & 0 & 0 & \cdots & a_r & a_{r-1} & a_{r-2} & \ddots & a_0\\

b_{s} & b_{s-1} & b_{s-2} & \cdots & b_0 & 0 & 0 & \cdots & 0\\
0 & b_s & b_{s-1} & \cdots & b_1 & b_0 & 0 & \cdots & 0\\
0 & 0 & b_s & \cdots & b_2 & b_1 & b_0 & \cdots & 0\\
\vdots & \vdots & \vdots & \ddots & \vdots & \vdots & \vdots & \ddots & 0\\
0 & 0 & 0 & \cdots & b_s & b_{s-1} & b_{s-2} & \ddots & b_0
\end{pmatrix}
\]
where the submatrix $S_f$ containing only the coefficients of $f$ is of dimension $s\times (r+s)$, and the submatrix $S_g$ containing only the coefficients of $g$ is of dimension $r\times (r+s)$. Unless $r=s$, the columns $(a_{0}~a_{1}~\cdots~a_r)$ and $(b_{0}~b_{1}~\cdots~b_s)$ of $S_f$ and $S_g$, respectively, are not aligned in the same column of $\operatorname{Syl}(f,g,x)$, as displayed above, but rather the first is shifted to the left or right of the second, depending on the relationship between $r$ and $s$. We will make implicit use of the following well-known symmetric and multiplicative properties of the resultant:

\begin{proposition}
	\label{prop:basicPropResultants} Let $f, g, h\in R[x]$. The resultant of $f$ and $g$ satisfies
\begin{itemize}
	\item $\res{f}{g}{x}=(-1)^{rs}\res{g}{f}{x}$,
	\item $\res{fg}{h}{x}=\res{f}{h}{x}\res{g}{h}{x}$.
	\item $f$ and $g$ have a common factor in $R[x]$ if and only if $\res{f}{g}{x}=0$.
\end{itemize}
\end{proposition}

The first two properties can be found in \cite[pp.~398]{GelfandKapranovZelevinsky}.
The third one is stated, without proof, in \cite[pp.~9]{GriffithsHarris} for unique factorization domains. When $R$ is a field, a proof of this property can be found in \cite[Chapter 3, Proposition 3 of \S6]{CoxLittleOshea}, and it directly generalizes to polynomial rings via Hilbert's Nullstellensatz.

\subparagraph{Resultants and elimination ideals.} We will work with multivariate homogeneous polynomials $f$ and $g$ in $\q[X_n]$, where a particular variable $x\in X_n$ is singled out. Since the resultant is a polynomial in the coefficients of $f$ and $g$, its net effect is that the specific variable $x$ is being {\em eliminated}. Formally, let $X'\subset X$ be non-empty and $R=\mathbb Q[X']$. Let $f,g\in R[x]$, where $x\in X\setminus X'$. It is clear from the definition of the resultant that $\res{f}{g}{x}\in R$. We will make frequent use of the following proposition, summarizing this observation; its proof can be found in \cite[pp.~167]{CoxLittleOshea}.

\begin{proposition}\label{prop:resultantElimination}
Let $I$ be an ideal of $R[x]$ and $f,g\in I$. Then $\res{f}{g}{x}$ is in the elimination ideal $I\cap R$.
\end{proposition}

\subparagraph{Homogeneous properties.} From next section on we will be working in the Cayley-Menger ideal, where the generators and the circuit polynomials are homogeneous. In \cref{sec:algResCircuits} and in \cref{sec:experiments} we will make use of the following proposition. 

\begin{proposition}
	\label{prop:resultantHomogeneous}
	Let $f=a_{m-r}x^{r}+\cdots+a_{m-1} x + a_m$ and $g=b_{n-s}x^{s}+\cdots+b_{n-1} x + b_n$ be homogeneous polynomials in $\q[y_1,\dots, y_t,x]$ of homogeneous degree $m$, resp.\ $n$, so that the coefficients $a_i,b_j\in \q[y_1,\dots,y_t]$ are polynomials of homogeneous degree $i$, resp.\ $j$, for all $i\in\{m-r,\dots,m\}$ and all $j\in\{n-s,\dots,n\}$. If $\res{f}{g}{x}\neq 0$, then it is a homogeneous polynomial in $\q[y_1,\dots,y_t]$ of homogeneous degree
\[m\deg_x{g}+n\deg_{x}{f}-\deg_x{f}\cdot\deg_x{g}=ms + nr-rs.\]
\end{proposition}

We were not able to find a reference for this proposition in the literature. In \ \cite[pp.~454]{CoxLittleOshea} (Lemma 5 of $\S7$ of Chapter 8) we found the following special case: let $f$ and\ $g$ be homogeneous polynomials of degree $r$, resp.\ $s$ with $\deg_xf=r$ and $\deg_xg=s$, so that $f=a_{0}x^{r}+\cdots+a_{1} x + a_r$ and $g =b_{0}x^{s}+\cdots+b_{1} x + b_s$. In this case $\res{f}{g}{x}$ is of homogeneous degree $rs$. The proof below is a direct adaptation of the proof of this special case, which itself follows directly from \cref{prop:resultantHomogeneous} by substituting $m\to r$ and $n\to s$ so to obtain $rs+sr-rs=rs$.

\begin{proof} Let $\operatorname{Syl}(f,g,x)=(S_{i,j})$ be the Sylvester matrix of $f$ and $g$ with respect to $x$, and let, up to sign, $\prod_{i=1}^{r+s}S_{i,\sigma(i)}$ be a non-zero term in the Leibniz expansion of its determinant for some permutation $\sigma$ of $[r+s]$.

A non-zero entry $S_{i,\sigma(i)}$ has degree $m-(r+i-\sigma(i))$ if $1\leq i\leq s$ and degree $n-(i-\sigma(i))$ if $s+1\leq i\leq r+s$. Therefore, the total degree of $\prod_{i=1}^{r+s}S_{i,\sigma(i)}$ is
\begin{align*}
&\sum_{i=1}^s[m-(r+i-\sigma(i))]+\sum_{i=s+1}^{r+s}[n-(i-\sigma(i))]
=\sum_{i=1}^s(m-r)+\sum_{i=s+1}^{s+r}n-\sum_{i=1}^{r+s}(i-\sigma(i))\\
=&s(m-r)+rn-0=m\deg_xg+n\deg_xf-\deg_xf\cdot\deg_xg.\end{align*}\end{proof}

\section{Circuit polynomials in the Cayley-Menger ideal.}
\label{sec:circuitsCMideal}

In this section we define {\em circuit polynomials} in the CM ideal and make the connection with combinatorial rigidity circuits via their supports. 

\subparagraph{Circuits of $\amat{\cm n}$ and circuit polynomials in $\cm n$.} The isomorphism between the algebraic matroid $\amat {\cm n}$ and the sparsity matroid $\smat n$ (\cref{theorem:algSparseIsomorphic}) immediately implies that the sets of circuits of these two matroids are in a one-to-one correspondence. We will identify a sparsity circuit $C=(V_C,E_C)\in \smat n$, with the algebraic circuit $\{x_{i,j}\mid ij\in E_C\} \in \amat{\cm n}$; similarly for dependent sets. Conversely, we will identify the support of a polynomial $f\in \q[\{x_{i,j}\mid 1\leq i<j\leq n\}]$ with the graph $G_f=(V_f,E_f)$ where
\[V_f=\{i\mid x_{i,j}\text{ or }x_{j,i}\in \supp f\}\text{ and }E_f=\{ij\mid x_{i,j}\in\supp f\}.\]

Given a (rigidity) circuit $C$, we denote by $p_C$ the corresponding {\em circuit polynomial} in the Cayley-Menger ideal $\cm n$. Recall that by \cref{thm:circuitPolyPrincipalIdeal} the circuit polynomial of a circuit $C$ in $\cm n$ is the unique (up to multiplication with a unit) polynomial $p_C$ irreducible over $\mathbb Q$ such that $\supp{p_C}=C$. Hence {\em we will identify from now on a circuit $C$ with the support $\supp p_C$ of its circuit polynomial $p_C$.} Furthermore, $p_C$ generates the elimination ideal $\cm n\cap\  \q[C]$.

\begin{proposition}\label{prop:circHomogeneous}Circuit polynomials in $\cm n$ are homogeneous polynomials.
\end{proposition} 
\begin{proof}Since $\cm n$ is generated by homogeneous polynomials, any reduced \grobner{} basis of $\cm n$ consists only of homogeneous polynomials (see e.g.\ Theorem 2 in \S3 of Chapter 8 of \cite{CoxLittleOshea}). If $C$ is a circuit in $\cm n$, we can choose an elimination order in which all the indeterminates in the complement of $C$ are greater than those in $C$. The \grobner{} basis $\mathcal G_C$ with respect to that elimination order will necessarily contain $p_C$ because $\mathcal G_C\cap \q[C]$ must generate the elimination ideal $\cm n\cap \q[C]$.
\end{proof}

\subparagraph{Example: the $K_4$ circuit.} The smallest circuit polynomials are found among the generators of $\cm n$. Their supports are in correspondence with the edges of complete graphs $K_4$ on all subsets of $4$ vertices in $[n]$. The circuit polynomial $p_{K_4^{1234}}$ given below corresponds to a $K_4$ on vertices $1234$. It is homogeneous of degree 3, has 22 terms and has degree 2 in each of its variables.
	\begin{align*}
	p_{K_4^{1234}}&=x_{3,4} x_{1,2}^2+x_{3,4}^2 x_{1,2}+x_{1,3} x_{2,3} x_{1,2}-x_{1,4} x_{2,3} x_{1,2}-x_{1,3} x_{2,4} x_{1,2}\\
	&+x_{1,4}^2 x_{2,3}+x_{1,3} x_{2,4}^2+x_{1,4} x_{2,4} x_{1,2}-
	x_{1,3} x_{3,4} x_{1,2}-x_{1,4} x_{3,4} x_{1,2}\\
	&+x_{1,3}^2 x_{2,4}+x_{1,4} x_{2,3}^2-x_{2,3} x_{3,4} x_{1,2}-x_{2,4} x_{3,4} x_{1,2}
	+x_{2,3} x_{2,4} x_{3,4}\\
	&-x_{1,3} x_{2,4} x_{3,4}-x_{1,3} x_{1,4} x_{2,3}
	-x_{1,3} x_{1,4} x_{2,4}-x_{1,3} x_{2,3} x_{2,4}\\
	&-x_{1,4} x_{2,3} x_{2,4}+x_{1,3} x_{1,4} x_{3,4}-x_{1,4} x_{2,3} x_{3,4}.
	\end{align*}

\subparagraph{Resultants of circuit polynomials.} Let $f,g$ be two polynomials in the Cayley-Menger ideal with $x_{ij}$ one of their common variables. We treat them as polynomials in $x_{ij}$, therefore the coefficients are themselves polynomials in the remaining variables. Our {\em main observation}, which motivated the  definition of the combinatorial resultant, is that the entries in the Sylvester matrix are polynomials supported exactly on the variables corresponding to the {\em combinatorial resultant} of the supports of $f$ and $g$ on elimination variable (edge) $ij$. 

The following lemma, whose proof follows immediately from Proposition \ref{prop:resultantElimination}, will be used frequently in the rest of the paper. 

\begin{lemma}
	\label{lem:resultantSupport}
	Let $I$ in $\q[X_n]$ be an ideal, let $f,g\in I$ be polynomials with support graphs $G_f= \supp f$ and $G_g= \supp g$ and with $x_{ij}$ a common variable, i.e. with edge $ij\in G_f\cap G_g$. Let  the combinatorial resultant of the support graphs be $S=\cres{G_f}{G_g}{ij}$, viewed as a set of variables $S \subset X_n$. Then $\res{f}{g}{x_{ij}}\in I\cap \q[S]$.
\end{lemma}
\section{Computing a circuit polynomial as a resultant of two smaller ones.}
\label{sec:algResCircuits}

We are now ready to complete the proof of our second result, \cref{thm:circPolyConstruction}. We show that combinatorial resultants are the combinatorial analogue of classical polynomial resultants in the following sense: if a (rigidity) circuit $C$ is obtained as the combinatorial resultant $\cres{A}{B}{e}$ of two circuits $A$ and $B$ with the edge $e$ eliminated, then the resultant $\res{p_{A}}{p_{B}}{x_e}$ of circuit polynomials $p_{A}$ and $p_{B}$ with respect to the indeterminate $x_e$ is supported on $C$ and contained in the elimination ideal $\ideal{p_C}$ generated by the circuit polynomial $p_C$. When $\res{p_{A}}{p_{B}}{x_e}$ is irreducible then it will be equal to $p_C$. However in general $p_C$ will only be one of its irreducible factors over $\mathbb Q$. In fact {\em exactly one factor} (counted with multiplicity) of $\res{p_{A}}{p_{B}}{x_e}$ may correspond to $p_C$ and that factor can be deduced by examining the supports of the factors and performing an ideal membership test on those factors that have the support of $p_C$.

These facts are summarized by \cref{alg:resultant}, where the work to {\em clean up} the resultant in order to extract the circuit polynomial is presented as the separate \cref{alg:cleanUpResultant}. 
The rest of this section is devoted to the proof of correctness of \cref{alg:resultant} and \cref{alg:cleanUpResultant}, along with several remaining open problems.

\begin{algorithm}[h]
 	\caption{\textbf{CircuitPolynomialResultant($\{A,B,e\}$, $\{p_A,p_B,x_e\}$)}\newline 
 	Compute a circuit polynomial based on a given combinatorial resultant decomposition}
 	\label{alg:resultant}
	\textbf{Input}: \\
 		Circuits $A$, $B$ and edge $e$ such that $C=\cres{A}{B}{e}$. \\
		Circuit polynomials $p_A$ and $p_B$ and elimination variable $x_e$.\\
	\textbf{Output}: Circuit polynomial $p_C$ for $C$.
	\begin{algorithmic}[1]
		\State Compute the resultant $p = \res{p_A}{p_B}{x_e}$.
		\If {$p$ is irreducible}
		\State $p_C = p$
		\Else
		\State $p_C$ = {\bf CleanUpResultant}($p$)
		\EndIf
		\State \Return $p_C$
	\end{algorithmic}
\end{algorithm}

\subsection{Correctness of \cref{alg:resultant}.} 
We proceed by analyzing the steps.
\subparagraph{Steps 1-4.} Their correctness is established by \cref{lem:structureCircuitPoly} and \cref{cor:structureCircuitPoly} below.

\begin{theorem}\label{lem:structureCircuitPoly}Let $C$ be a sparsity circuit on $n+1$ vertices and $p_C$ its corresponding circuit polynomial. There exist sparsity circuits $A$ and $B$ on at most $n$ vertices with circuit polynomials $p_{A}$ and $p_{B}$ such that $p_{C}$ is an irreducible factor over $\mathbb Q$ of $\res{p_{A}}{p_{B}}{x_e}$, where $e\in A\cap B$.
\end{theorem}

\begin{proof}Given a sparsity circuit $C$ on $n+1$ vertices we can find two sparsity circuits $A$ and $B$ on at most $n$ vertices such that $C=\cres{A}{B}{e}$ for some $e\in A\cap B$ by the proof of \cref{prop:circuits3conn}. Let $p_{A}$ and $p_{B}$ be the corresponding circuit polynomials.
	
	The polynomials $p_{A}$ and $p_{B}$ are contained in $\cm{m}$ for some $m\geq n+1$ and the resultant $\res{p_{A}}{p_{B}}{x_e}$ is a non-constant polynomial in $R=\mathbb Q[(A\cup B)\setminus\{x_e\}]$ supported on $C$. Since $\ideal{p_{A},{p_{B}}}\subset\cm{m}$, we have that $\res{p_{A}}{p_{B}}{x_e}$ is contained in the elimination ideal $\cm{m}\cap\mathbb{Q}[C]=\ideal{p_C}$ (by \cref{lem:resultantSupport}).
\end{proof}

\begin{corollary}\label{cor:structureCircuitPoly}Under the assumptions of Theorem \ref{lem:structureCircuitPoly}, the resultant $\res{p_{A}}{p_{B}}{x_e}$ is a circuit polynomial if and only if it is irreducible (over $\mathbb Q$).\end{corollary}

The clean-up part would not be necessary if the resultant would always be irreducible. But in general $p_C$ will only be one of the irreducible factors over $\q$ of $\res{p_A}{p_B}{x_e}$.

\begin{lemma}\label{lem:resultantNotIrreducible}
	The resultant of two circuit polynomials is not always a circuit polynomial.
\end{lemma}
\begin{proof}
	We prove the Lemma with an example, which can be easily generalized. Recall from \cref{cor:nonunique} that in general a sparsity circuit $C$ can be represented as the combinatorial resultant of two circuits in more than one way. If $C=\cres{C_1}{C_2}{e}=\cres{C_3}{C_4}{f}$ and $p_{C_i}$ for $i\in\{1,\dots,4\}$ are the corresponding circuit polynomials, then $\res{p_{C_1}}{p_{C_2}}{x_e}$ and $\res{p_{C_3}}{p_{C_4}}{x_f}$ will in general be distinct elements of $\ideal{p_C}$. The 2-connected circuit in \cref{fig:exampleCircuit} has two distinct CCR trees, one in which the root is obtained as the combinatorial resultant of two $K_4$'s, and the other in which the root is obtained as the combinatorial resultant of two wheels on 4 vertices. The corresponding circuit polynomials in the former case are of homogeneous degree 3 and quadratic in any indeterminate, and in the latter case they are of homogeneous degree 8 and quartic in any indeterminate (see \cref{sec:experiments}). Using \cref{prop:resultantHomogeneous} to compute the homogeneous degrees of the resultants, we obtain homogeneous degrees 8 and 48, respectively. Both resultants have the same circuit as its supporting set, hence they are both in the elimination ideal $\ideal{p_C}$, but only the one of homogeneous degree 8 is the circuit polynomial (which was verified by checking for irreducibility). 
\end{proof}

We can generalize the example in the proof of  \cref{lem:resultantNotIrreducible} in the following way. Let $C$ be a sparsity circuit on $n\geq 5$ vertices. Consider the set of all possible decompositions of $C$ as a combinatorial resultant of two sparsity circuits $A$ and $B$ on at most $n$ vertices: 
\[ Decompositions(C)=\{(A,B,e)\mid C=\cres{A}{B}{e}, |V(A)|,|V(B)|\leq |V(C)|\}\]
and the set of all resultants of corresponding circuit polynomials:
\[ Resultants(C)=\{\res{p_{A}}{p_{B}}{x_e}\mid (A,B,e)\in Decompositions(C)\}.\] 

The circuit polynomial $p_C$ of the circuit $C$ in the proof of \cref{lem:resultantNotIrreducible} had the property of being the polynomial in $ Resultants(C)$ of minimal homogeneous degree. One might therefore conjecture that for any sparsity circuit $C$, the polynomial in $ Resultants(C)$ of minimal homogeneous degree is the circuit polynomial for $C$; in that case no irreducibility check would be required as we can compute the homogeneous degree of $\res{p_{A}}{p_{B}}{x_e}$ from the homogeneous degrees and the degrees in $x_e$ of $p_{A}$ and $p_{B}$ (\cref{prop:resultantHomogeneous}). However, we will show in \cref{prop:k33notObtainable} that in general the circuit polynomial of a circuit $C$ is not necessarily {\em by itself} in $ Resultants(C)$; only a multiple of it (by a non-trivial polynomial) is. This fact leads to the following natural question.

\begin{problem}
	Identify sufficient conditions under which $\res{p_{A}}{p_{B}}{x_e}$ is $p_C$.
\end{problem}

	\begin{algorithm}[ht]
		\caption{{\bf CleanUpResultant($C$, $p$)}\\ Extract the circuit polynomial from a reducible polynomial.\\ 
		{\bf Preconditions:}\\
		$p$ is a resultant of two other circuit polynomials.\\
		$p$ is supported on a circuit $C$.
		}
		\label{alg:cleanUpResultant}
		\textbf{Input}: A circuit $C=\cres{A}{B}{e}$ and the polynomial $p$ obtained as Res($p_A$,$p_B$,$x_e$).\\
		Assume that $p$ is reducible.\\
		\textbf{Output}: Circuit polynomial $p_C$ for $C$.
		\begin{algorithmic}[1]
			\State factors = factorize $p$ over $\q$
			\State factors = discard factors with support not equal to $C$
			\If {exactly one remaining factor (possibly with multiplicity)}
			\State$p_C$ = the unique factor supported on $C$ 
			\State \Return $p_C$
			\Else 
			\State apply a test of membership in the CM ideal on the remaining factors
			\State $p_C$ = unique factor for which ideal membership test succeeded
			\State \Return $p_C$		
			\EndIf
		\end{algorithmic}
	\end{algorithm}
	
If $\res{p_{A}}{p_{B}}{x_e}$ is not irreducible, \cref{alg:resultant} invokes \textbf{CleanUpResultant} (\cref{alg:resultant}, whose correctness we now analyze.

\subparagraph{Step 1.}  In step 1 we first factorize $p$ over $\q$, which can be achieved in polynomial time (see \cite{Kaltofen} for a historical overview). Up to multiplicity, exactly one of the irreducible factors of $p$ is in $\cm n$, and that factor is precisely the circuit polynomial $p_C$ (because $p_C$ generates the elimination ideal $\cm{n}\cap\q[C]$). The desired factor can be deduced in two steps: an analysis of the supports of all the factors and an ideal membership test. 

\subparagraph{Steps 2-5: analyzing the supports of the irreducible factors.} Recall that we identify a circuit $C$ with the variables $\supp p_C$ in the support of the corresponding circuit polynomial $p_C$ and that the elimination ideal $\ideal{p_C}$ is an ideal of $\mathbb \q[C]$. Let $C=\cres{A}{B}{e}$. Since $\res{p_{A}}{p_{B}}{x_e}\in\ideal{p_C}$, any irreducible factor (over $\mathbb Q$) of this resultant is supported on a subset of $\supp p_C$ that is not necessarily proper. At least one these factors must be supported on exactly $\supp p_C$, and if there is only one such factor, then that factor must be $p_C$.

\begin{problem}
	Identify sufficient conditions for which $\res{p_{A}}{p_{B}}{x_e}$ has exactly one factor (up to multiplicity) supported on $C$. 
\end{problem}

Lacking a definitive answer at this time, we proceed to Step 6.\\

\subparagraph{Steps 6-9: ideal membership test.} We take into consideration only those irreducible factors of $\res{p_{A}}{p_{B}}{x_e}$ that are supported on $\supp{p_C}$ (the others are automatically discarded as not belonging to the ideal). We then have to test each factor for membership in $\cm n$. This test can be done via a \grobner{} basis algorithm with respect to any monomial order, not necessarily an elimination order. The first factor determined to be in $\cm n$ is $p_C$.

It is not yet clear that this test is necessary: in practical experiments with our method, we have not yet encountered the need. 

\begin{problem}
	Produce an example where the resultant of two circuit polynomials in the Cayley-Menger ideal,  whose combinatorial resultant is a circuit $C$, has a factor different from $p_C$ but supported on $\supp p_C$, or prove that this never happens.
\end{problem}

\subsection{The impact of the ideal membership test.} 
The main complexity-theoretic bottleneck in our approach for computing circuit polynomials is that we {\em may} still have to compute a \grobner{} basis in order to apply an ideal membership test. If it turns out that this step cannot be avoided, there are results suggesting that this test will not reduce our method back to a costly version of a \grobner{} basis calculation.

An ideal membership test is indeed done by computing a \grobner{} basis, but it does not require an elimination order, which is by all accounts impractical.  Elimination orders are only necessary for computing elimination ideals (and this is what we are avoiding with our resultant-based algorithm): it is well documented that they behave badly (see \cite[section 4]{BayerMumford} and section \emph{Complexity Issues} in \cite[\S10 of Chapter 2]{CoxLittleOshea}). On the other hand, graded orders show better performance but cannot be used to compute elimination ideals.

In summary: our approach avoids the use of an elimination order, requires only one elimination step that is obtained with resultants, and is followed by a factorization with a potential ideal membership test that can be performed by a \grobner{} basis with respect to \emph{any} monomial order. Hence we are free to choose a monomial order for $\cm n$ that we expect to have the best performance. Of course, it is difficult to know a priori what that {\em good} order will be. A further investigation of this part of the algorithm remains to be pursued, in connection with the open problems described previously.

\section{Computing a circuit polynomial from a combinatorial circuit-resultant (CCR) tree.}
\label{sec:resTree}

We have now all the ingredients to describe an algorithmic solution to the Main Problem stated in the Introduction: given a rigidity circuit $C$, compute its circuit polynomial $p_C$. 

\begin{algorithm}[h]
	\caption{\textbf{CircuitPolynomial($T_C$):} \newline Compute a circuit polynomial from a CCR tree, inductively.}
	\label{alg:resTreeAlg1}
		\textbf{Input}: A CCR tree $T_C$ with root a circuit $C$.\\
		\textbf{Output}: Circuit polynomial $p_C$ for $C$.\\
		\textbf{Method}: Traverse the tree $T_C$ bottom-up, level by level.
	\begin{algorithmic}[1]
		\State $h = $ height of $T_C$
		\State level $= h-1$
		\While{$\textnormal{level}\geq 0$}
			\State At all the nodes $C_i$ of the current level, compute the circuit polynomial $p_{C_i}$ from the polynomials at its two children nodes $\{C_j,C_k\}$ using \textbf{CircuitPolynomialResultant} (\cref{alg:resultant})
			\State $\textnormal{level}=\textnormal{level}-1$
		\EndWhile
		\State \Return $p_C$
	\end{algorithmic}
\end{algorithm}

One way of doing this is captured by \cref{alg:resTreeAlg1}.  It uses a combinatorial circuit-resultant tree (CCR tree) $T_C$ that was precomputed with \cref{alg:comb}. It inductively computes polynomials supported by circuits at levels of the tree closer to the root from polynomials supported on circuits on a higher level. This algorithm stores all circuit polynomials on one level prior to going to the next level. The method becomes impractical when the CCR tree has a large number of vertices on some level, as would be the case, say,  when the binary CCR tree is balanced. The correctness of \cref{alg:resTreeAlg1} follows directly from \cref{alg:comb} and \cref{alg:resultant}. 

\cref{alg:resTreeAlg2} takes an alternative approach and traverses the CCR tree in postfix order. This is naturally described as a recursive procedure. The recursion stack retains left child circuit polynomials along a path to a node from the root in the CCR tree, and thus its space complexity depends on the depth of the tree. 

\begin{algorithm}[h]
	\caption{{\bf CircuitPolynomialRecursive}: \newline Circuit polynomial from CCR tree, postfix traversal processing}
		\label{alg:resTreeAlg2}
		\textbf{Input}: A CCR tree $T_C$ with root a circuit $C$.\\
		\textbf{Output}: Circuit polynomial $p_C$ for $C$.\\
		\textbf{Method}: Traverse the tree $T_C$ in postfix order.
	\begin{algorithmic}[1]
		\If {$C$ is isomorphic to $K_4$}
		\State $p_C = p_{K_4}$ with the appropriate relabeling of vertices
		\State \Return $p_C$
		\Else 
		\State Let $T_A, T_B$ be the left and right subtrees of $T_C$, with $C=\cres{A}{B}{e}$ and $x_e$ the elimination variable.
		\State $p_A$ = {\bf CircuitPolynomialRecursive}($T_A$)
		\State $p_B$ = {\bf CircuitPolynomialRecursive}($T_B$)
		\State $p_C$ = {\bf CircuitPolynomialResultant}($\{ A, B, e \}$, $\{ p_A, p_B, x_e \}$) (\cref{alg:resultant})
		\EndIf
		\State \Return $p_C$
	\end{algorithmic}
\end{algorithm}

\emph{Finding a performance-optimal CCR tree for the computation of a specific circuit polynomial is a problem that remains to be investigated. It is expected that a tree that balances depth, breadth and various algebraic parameters of the polynomials involved in the resultant steps would yield the best performance.}

\subsection{The ``delayed clean up'' heuristic.} 
Algorithms \ref{alg:resTreeAlg1} and \ref{alg:resTreeAlg2} described above invoke a \textbf{CleanUpResultant} within the \textbf{CircuitPolynomialResultant} call associated to each node of the CCR tree. This is not necessary: we could just compute the resultant instead of invoking the whole \textbf{CircuitPolynomialResultant} (\cref{alg:resultant}) and delay the cleaning up of the resultant polynomials until we reach the root or when absolutely necessary. \emph{Absolutely necessary} means that either (a) a resultant  vanishes or that (b) the \grobner{} Basis calculation for the ideal membership test in the clean up of the resultant is too expensive in terms of resources (time and memory), e.g.\ it takes too long, exhausts the available memory resources or crashes.  This simple ``delayed clean up'' heuristic may be useful in practice, in the sense that it may speed up the calculations in specific cases. We prove now that it is correct if we handle the vanishing resultant as follows.

Let $r_C = \res{r_A}{r_B}{x_e}$ be the resultant of two previously computed polynomials $r_A$ and $r_B$ that have not been cleaned up. They contain the circuit polynomials $p_A$, resp.\ $p_B$ among their (not common) factors. If $r_C$ vanishes, then $r_A$ and $r_B$ have some common factors. We proceed with a {\bf SimplifiedCleanUp} and factorize $r_A$ and $r_B$, remove their common factors to obtain $q_A$ and $q_B$ and recompute the new (non-vanishing) resultant $q_C = \res{q_A}{q_B}{x_e}$. This simplified cleaning up procedure does not require an ideal membership test. The resultant $q_C$ is well defined, because $q_A$ (resp.\ $q_B$) contains the circuit polynomial $p_A$ (resp.\ $p_B$) among its factors, hence $x_e$ is in the support of both. The multiplicativity of the resultant (\cref{prop:basicPropResultants} (ii)) implies that the resultant $q_C$  of the simplified polynomials $q_A$ and $q_B$ will be non-zero and contain a unique factor (up to multiplicity) equal to the circuit polynomial $p_C$ for $C=\cres{A}{B}{e}$. Therefore, the algorithm can proceed in a ``delayed clean up'' fashion until it encounters another vanishing resultant, performs another factorization and so on, until it reaches the root, at which point a full clean up must be performed.

We do not know whether vanishing resultants will ever occur because in our experiments we have encountered only irreducible polynomials. High performance computing may help answer these remaining questions:

\begin{problem}
Find an example where a reducible polynomial appears in an intermediate step of a delayed clean up circuit polynomial calculation.
\end{problem}

\begin{problem}
Find an example where a delayed clean up circuit polynomial calculation has an intermediate resultant equal to zero.
\end{problem}

\begin{problem}
Provide experimental evidence on whether the ``delayed clean up'' he\-u\-ri\-stic can speed up a circuit polynomial calculation.
\end{problem}

\subsection{Complexity measures for CCR trees.} 
Recall from \cref{cor:nonunique} that a circuit $C$ can have more than one CCR tree. The circuit polynomial itself is independent of this choice, but in its calculation it is useful to keep the size of the intermediate polynomials, with respect to the number of monomial terms and homogeneous degree, as small as possible. In other words, for a rigidity circuit $C$ we would like to be able to identify an \emph{optimal} CCR tree. The complexity of the algebraic Algorithms \ref{alg:resTreeAlg1} and \ref{alg:resTreeAlg2} is influenced by several factors encoded in the CCR tree: its size (total number of resultant operations),  its breadth (number of nodes on the largest level), depth (longest path from root to a leaf) as well as the specificity of the elimination edge at each internal node. This motivates the following:

\begin{problem}
	Define a meaningful measure of CCR-tree complexity that would lead to effective computations of larger\footnote{E.g.\ larger than  those reported in  \cref{sec:experiments}.} circuit polynomials. 
\end{problem}

One can aim for a CCR tree in which the homogeneous degrees at each level are minimized, according to the formula given in \cref{prop:resultantHomogeneous}, however it is not clear if this is the best approach. Indeed, in the first algorithm the degree of the circuit polynomial at a node may be smaller than predicted by \cref{prop:resultantHomogeneous}, since the circuit polynomial may be just a factor and not the whole resultant.

Identifying optimal trees would impact the practical calculations of circuit polynomials. The concrete results reported later on in \cref{sec:experiments} of this paper were possible because we could easily select, when $n<7$, an optimal resultant tree from a small set of possibilities, but this set grows fast with $n$. It is desirable to be able to directly compute an optimal CCR tree, rather than having to iterate through all the possibilities when searching for an optimal one.

\begin{problem}
Refine Algorithm \ref{alg:comb} (and its analysis) to produce an {\em optimal} CCR tree, according to a measure of CCR-tree complexity leading to efficient resultant-based calculations of circuit polynomials.
\end{problem}

With the methods developed so far we were able to compute all the circuit polynomials in $\cm{6}$ except for the $K_{3,3}$-plus-one circuit. The computation of the circuit polynomial for the $K_{3,3}$-plus-one circuit exhausted all memory at the resultant step, i.e.\ Step 1 of \cref{alg:resultant}. However, by modifying the algorithm so that it also allows polynomials supported on \emph{dependent sets} in $\cm{n}$ that are not necessarily circuits, we were able to compute the circuit polynomial for the $K_{3,3}$-plus-one circuit. We present now this extended algorithm.

\section{Combinatorial Resultant Trees.}
\label{sec:extendedCombRes}

We generalize the algorithms in \cref{sec:resTree} by allowing all dependent sets in the rigidity matroid at the nodes, with the aim of improving computational performance.

First we relax some of the constraints imposed on the resultant tree by the construction from \cref{sec:resultantTree}. The internal nodes correspond, as before, to combinatorial resultant operations, but: (a) they are no longer restricted to be applied only on circuits or to produce only circuits; (b) the leaves can be labeled by graphs other than $K_4$'s, and (c) the sequence of graphs on the nodes along a path from a leaf to the root is no longer restricted to be strictly monotonically increasing in terms of the graphs' vertex sets. 

\begin{definition}
	\label{def:generators}
	A finite collection $\Gen$ of \emph{dependent graphs} such that $K_4 \in \Gen$ will be called a \emph{set of generators}. 
\end{definition}

The generators in $\Gen$ will be the graphs allowed to label the leaves. For the purpose of generating (combinatorial) circuits and computing (algebraic) circuit polynomials, we choose a set of generators, discussed in \cref{sec:generatorsCM},  that are \emph{dependent} in the rigidity matroid.

\begin{definition}
	\label{def:resultantTree}
	A \emph{combinatorial resultant tree} {\em (shortly, CR tree)} with generators in $\Gen$ is a finite binary tree such that: (a) its leaves are labeled with graphs from $\Gen$, and (b) each internal node marked with a graph $G$ and an edge $e\not\in G$ corresponds to a combinatorial resultant operation applied on the two graphs $G_1, G_2$ labeling its children. Specifically,  $G=\cres{G_1}{G_2}{e}$, where the edge $e\in G_1\cap G_2$.
\end{definition}

Hence, CCR trees are special cases of CR trees. An example of a CR tree which is not a CCR tree is illustrated in \cref{fig:resTreeK33}.

\begin{figure}[ht]
	\centering
	\includegraphics[width=.8\textwidth,trim={0 0 0 0.5mm},clip]{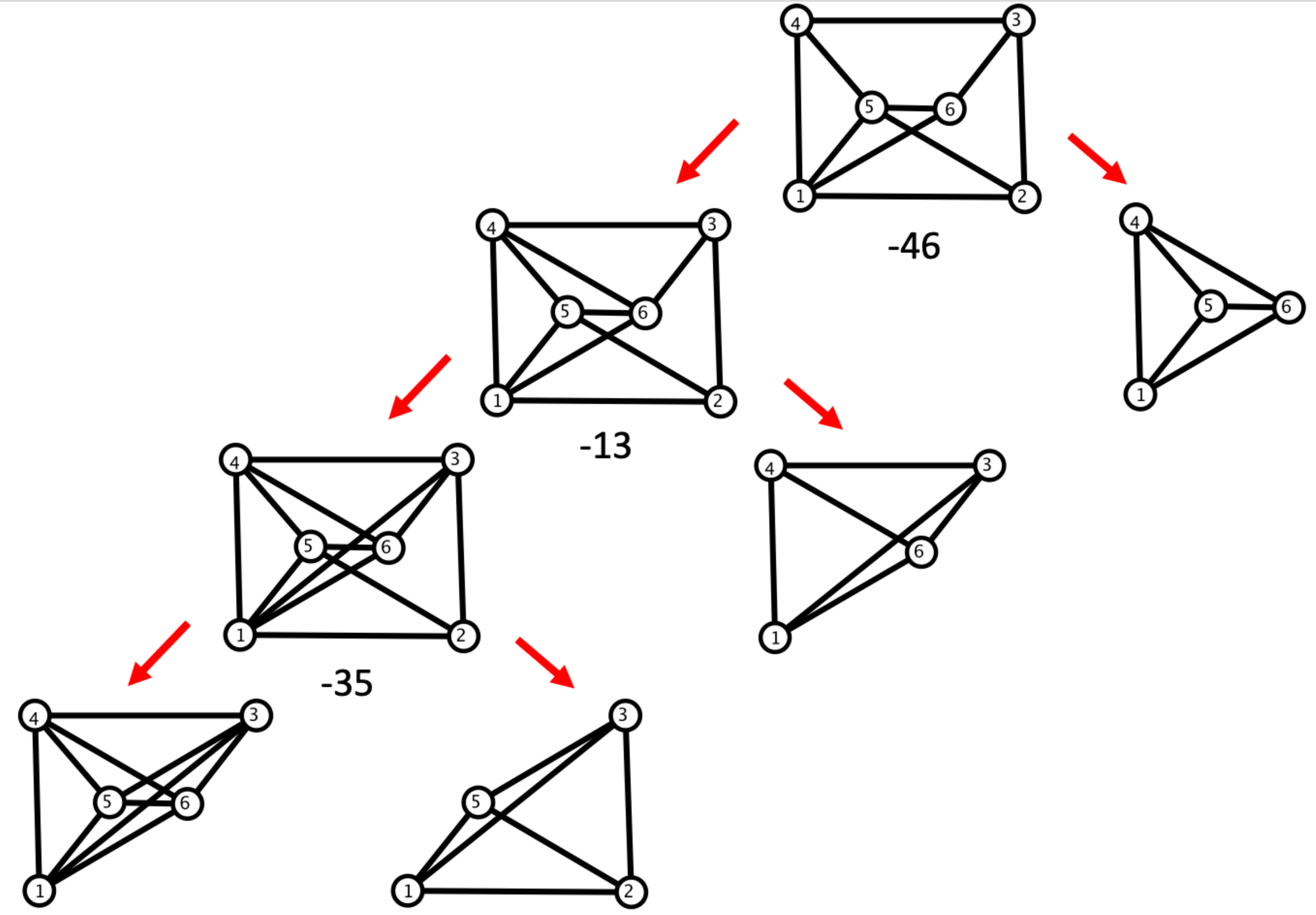}
	\caption{A \emph{combinatorial resultant tree} for the $K_{3,3}$-plus-one circuit: its leftmost leaf and the two internal nodes along the leftmost path to the root are labeled with rigid dependent graphs which are \emph{not} circuits.}
	\label{fig:resTreeK33}
\end{figure}

\begin{lemma}
	\label{lem:dependent}
	If the generators \emph{Gen} are dependent graphs (in the rigidity matroid), then all the graphs labeling the nodes (internal, not just the leaves) of a  combinatorial resultant tree are also \emph{dependent}.
\end{lemma}

\begin{proof}
The proof is an induction on the tree nodes, with the base cases at the leaves. We define an edge of $G$ to be redundant if after its deletion the graph remains rigid; otherwise the edge is said to be critical: its removal makes the graph flexible. For the inductive step, assume that $G_1$ and $G_2$ are the dependent graphs labeling the two children of a node labeled with $G=\cres{G_1}{G_2}{e}$, where $e\in E_{\cap}$ is an edge in the common intersection $G_{\cap}$. We consider two cases, depending on whether $e$ is \emph{redundant} in both or \emph{critical} in at least one of $G_1$ and $G_2$. In each case, we identify a subset of the combinatorial resultant graph $G$ which violates Laman's property, hence we'll conclude that the entire graph $G$ is dependent.
	
{\bf Case 1: $e$ is redundant in both $G_1$ and $G_2$.} This means that there exist subsets of edges $C_1\subset G_1$ and $C_2\subset G_2$, both containing the edge $e$, which are circuits (their individual spanned-vertex sets may possibly contain additional edges, but this only makes it easier to reach our desired conclusion). Their intersection $C_1\cap C_2$ cannot be dependent (by the minimality of circuits). Hence their union, with edge $e$ eliminated, has at least $2n_{\cup}-2$ edges (cf.\ the proof of \cref{lem:combRes2n2}), hence it is dependent. 

{\bf Case 2: $e$ is critical in $G_1$ or critical in $G_2$.} Let's assume it is critical in $G_1$. Since $G_1$ is dependent and $e\in G_1$ is critical, it means that the removal of $e$ from $G_1$ creates a flexible graph which is still dependent. As a flexible graph, it splits into edge-disjoint rigid components; in this case, at least one of these components $R$ is dependent. Then, since the removal of $e$ does not affect $R$, it follows that $R$ and thus the resultant graph $G=\cres{G_1}{G_2}{e}$, remain dependent.
\end{proof}

\begin{definition}
	\label{def:rdepResultantTree}
	Given a circuit $C$, a \emph{valid  combinatorial resultant tree} for $C$ is a combinatorial resultant tree with root $C$ and whose leaves (and hence nodes) are dependent graphs.
\end{definition}

The example in \cref{fig:resTreeK33} is a valid combinatorial resultant tree for the $K_{3,3}$-plus-one circuit. After reviewing the necessary algebraic notions in the next section, we will use it in \cref{subsec:computingK33} to demonstrate our generalized algebraic elimination algorithm described in \cref{sec:algebraicNew}.
\section{Generators of the 2D Cayley-Menger ideal.}
\label{sec:generatorsCM}

We work with the set $\cmgen n$ of generators for the 2D Cayley-Menger ideal $\cm n$ as given by the set of all $5\times 5$ minors of the $(n+1)\times (n+1)$ Cayley matrix. Each generator $g\in \cmgen n$ is identified with its support graph $G_g$, as defined in \cref{sec:circuitsCMideal}. To motivate the possible choices for the family of graphs $\Gen$ for the generalized combinatorial resultant trees defined in \cref{sec:extendedCombRes}, we now tabulate the support graphs of all generators, up to multiplication by a non-zero constant, relabeling and graph isomorphism. 

To find all these graphs, it is sufficient to consider the set $\cmgen{10}$ of all $5\times 5$ minors of $\cm{10}$. Using a computer algebra package we can verify that this set has 109 619 distinct minors, of which 106 637 have distinct support graphs. The IsomorphicGraphQ function of Mathematica was used to reduce them to the $14$ graph isomorphism classes, 11 of which are shown in \cref{fig:generatorGraphs}. The only two representatives with less than 6 vertices are $K_4$ and $K_5$. There are three isomorphism classes on 6, 7, 8 vertices (one is $K_6$), two on 9 and one on 10 vertices. The corresponding generator polynomials are, up to isomorphism (relabeling of variables induced by relabeling of the vertices), unique for the given support, with a few exceptions: for $K_5$, we found $3$ distinct (non-isomorphic) polynomials.

Note that there may be polynomials in $\cm{n}$ supported on the same set as a generator from $\cmgen n$, but which themselves do not arise from a single $5\times 5$ minor of a Cayley matrix. For example, if $p\in \cmgen n$ is supported on a $K_5$ and $q\in \cmgen n$ is supported on a $K_6$ such that $\supp p \subset \supp q$, then $p+q$ has the support of a generator on $K_6$ but itself is not in $\cmgen n$.

\begin{figure}[ht]
	\centering
	\includegraphics[width=0.33\textwidth]{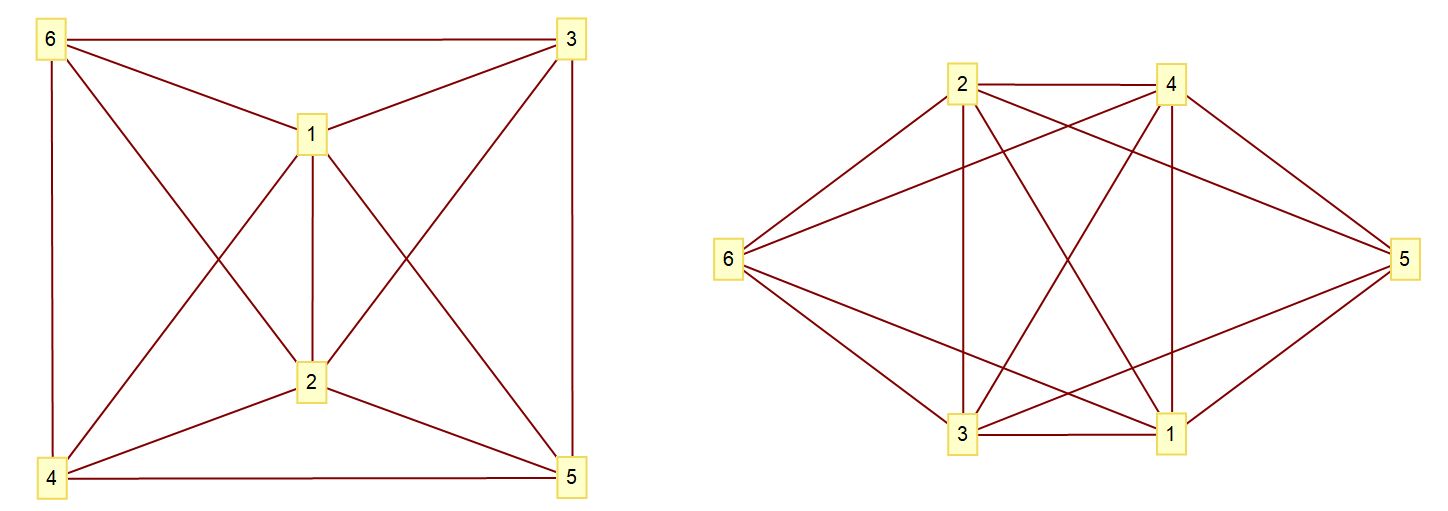}
	\includegraphics[width=0.45\textwidth]{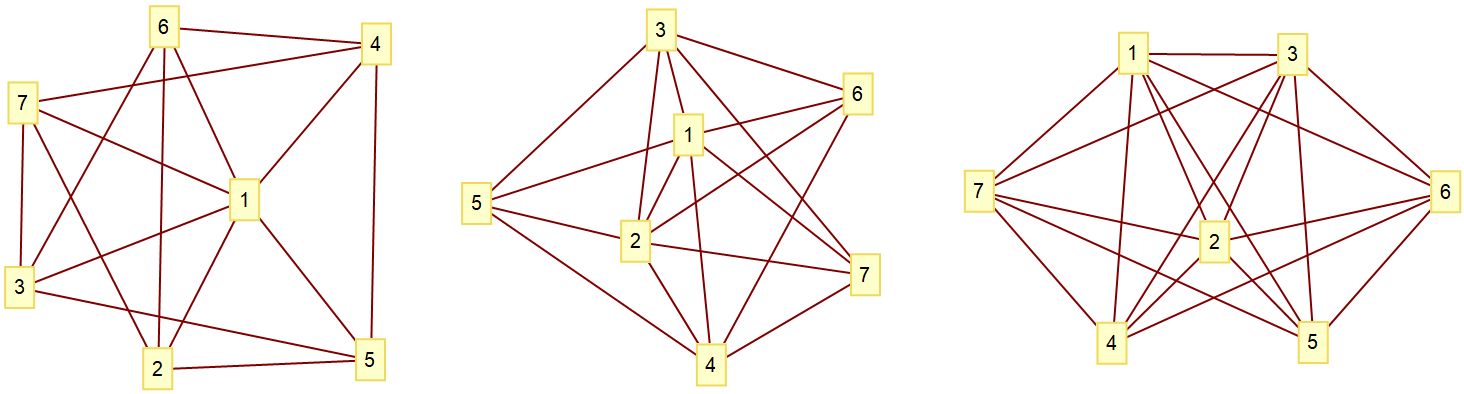}\\
	\includegraphics[width=0.45\textwidth]{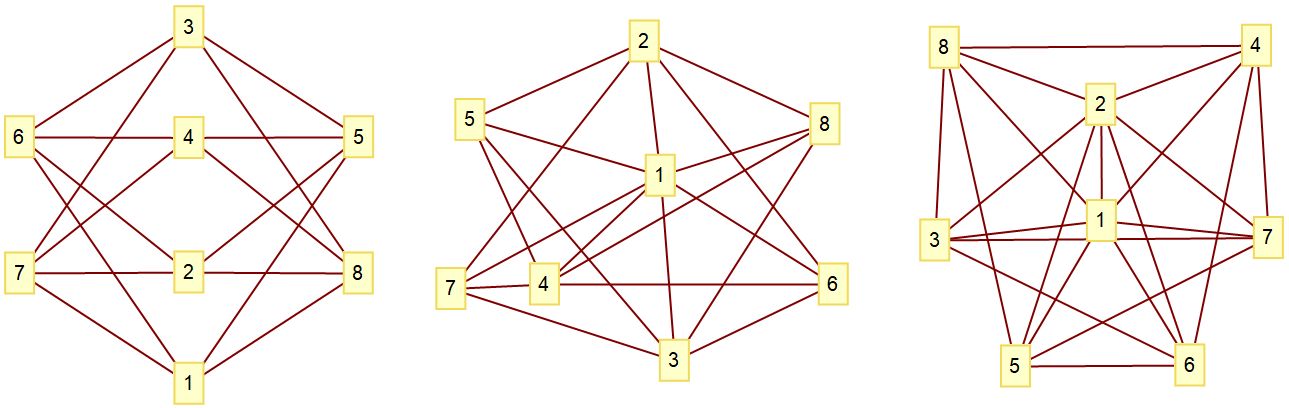}
	\includegraphics[width=0.3\textwidth]{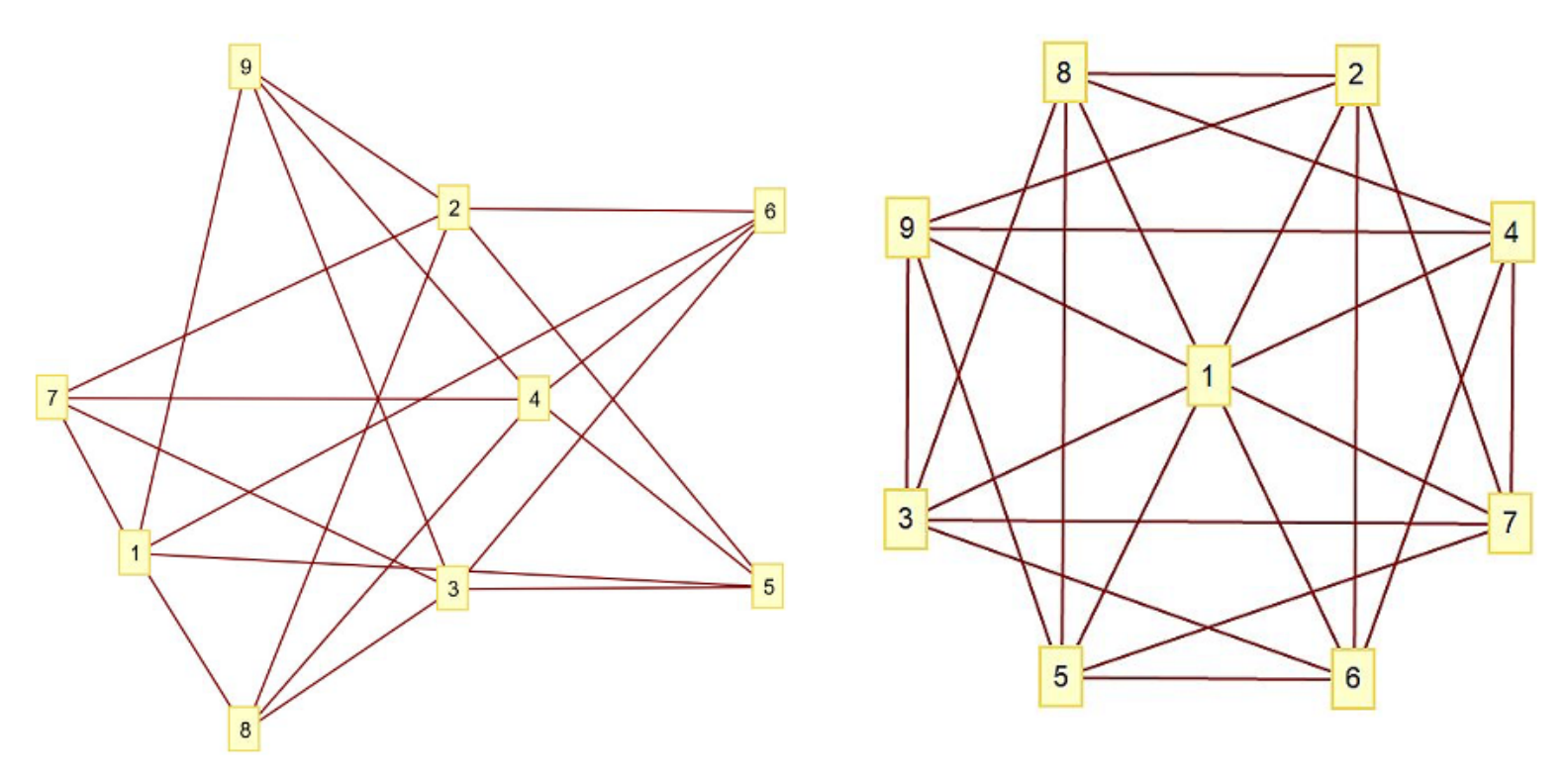}
	\includegraphics[width=.15\textwidth]{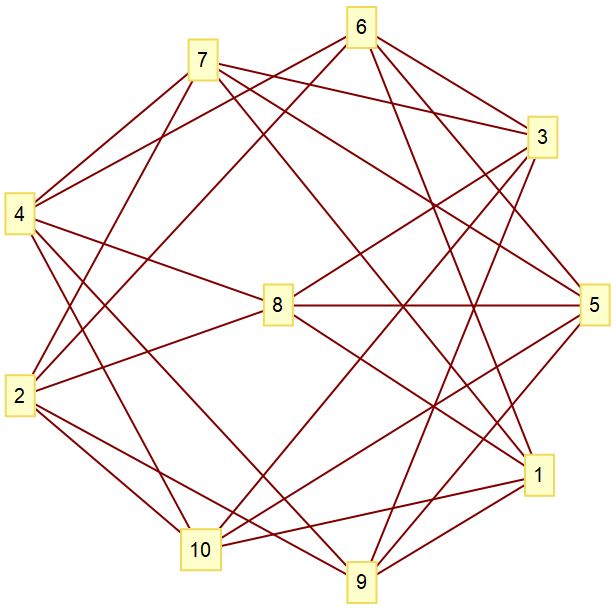}	
	\label{fig:generatorGraphs}
	\caption{The $14$ graph isomorphism classes of Cayley-Menger generators consist in the three complete graphs $K_4, K_5$, $K_6$ and the 11 graphs on $6$ to $10$ vertices shown here.}
\end{figure}
\section{Algorithm: circuit polynomial from combinatorial resultant tree.}
\label{sec:algebraicNew}

We now have all the ingredients for describing \cref{alg:resTreeExtAlg1} that computes the circuit polynomial $p_C$ for a circuit $C$ from a given combinatorial resultant tree $T_C$, or returns a message that $p_C$ can not be computed using $T_C$. Just like the algorithms of \cref{sec:resTree}, it computes resultants at each node of the tree, starting with the resultants of generators of $\cm{n}$ supported on leaf nodes. At the root node the circuit polynomial for $C$ is extracted from the irreducible factors of the resultant at the root. The main difference lies at the intermediate (non-root) nodes, as described in \cref{alg:resTreeExtAlg1} below. This is because the polynomials sought at non-leaf nodes, not being supported on circuits, are not necessarily irreducible polynomials \emph{supported on the desired dependent graph} as was the case in \cref{sec:resTree}. Hence, conceivably, they may have factors that are not in the Cayley-Menger ideal, and it might be the case that none of their factors that are in the Cayley-Menger ideal are supported on the desired graph, but their product with other factors is. Moreover, it might be the case that an intermediate resultant $\res{f}{g}{x}$ is zero, with $x$ being present only in the supports of common factors of $f$ and $g$, in which case the algorithm can not resume along the chosen tree $T_C$. It remains, however, as an open question (which may entail experimentation with gigantic polynomials) to explicitly find such examples (we did not find any so far) and to prove what may or may not happen. 

\begin{algorithm}[ht]
	\caption{Computing a polynomial in the Cayley-Menger ideal supported on a node of a  combinatorial resultant tree - simple version.}
	\label{alg:resTreeExtAlg1}
	\textbf{Input}: Non-leaf node $G$ of a  combinatorial resultant tree $T_C$. Polynomials $v,w\in\cm{n}$ supported on the child nodes of $G$ and $x_e$, the indeterminate to be eliminated.\\
	\textbf{Output}: Polynomial $p\in\cm{n}$ supported on $G$ or a string stating that $p$ could not be computed.
	\begin{algorithmic}[1]
		\State Compute the resultant $r=\res{v}{w}{x_e}$.
		\State If $r = 0$ \Return ``Not possible to compute $p$''.
		\State Factorize $r$ over $\mathbb Q$ and store all factors supported on dependent sets in the list \textit{candidates}.
		\If{$\textit{candidates}=\{p\}$}
		\If{$\supp p = G$} \Return $p$
		\Else{} \Return $p\cdot \Pi_{x\in G\setminus \supp p}x$
		\EndIf
		\Else{}	
		\ForAll{$p \in \textit{candidates}$}
		\State Test $p$ for membership in $\cm n$ with an ideal membership test
		\If{$p\in \cm n$}
		\If{$\supp p = G$} \Return $p$
		\Else{} \Return $p\cdot \Pi_{x\in G\setminus \supp p}x$
		\EndIf
		\EndIf				
		\EndFor
		\EndIf		
	\end{algorithmic}
\end{algorithm}

\subparagraph{Proof of correctness of \cref{alg:resTreeExtAlg1}.} Recall that $\q[G]$ denotes the ring of polynomials with indeterminates $x_{ij}$ with $i<j$ given by the edges $ij$ of $G$.

Steps 1--2: Compute the resultant. If the resultant is zero, the algorithm terminates with the message that it is not possible to continue along $T_C$. We can attempt to replace one or both of $v$ and $w$ with other polynomials in $\cm n$ with appropriate support that would lead to a non-zero resultant, however in our presentation we assume that all the choices made in previous calls of \cref{alg:resTreeExtAlg1} (e.g.\ the choice of a candidate in line 9) remain fixed. 
 
Step 3: The elimination ideal $\cm n \cap \q[G]$ is prime, hence at least one irreducible factor $p$ of $r$ is in $\cm n$.

Step 4: If there is exactly one factor $p$ supported on a dependent set, then that factor must necessarily be in $\cm n$. This follows from the primality of $\cm n \cap \q[G]$: assume for simplicity that $r$ factors as $q_1\cdot q_2\cdot p$ with only $p$ being supported on a dependent set. If $q=q_1\cdot q_2$ is supported on an independent set, then it is not in $\cm n$, hence $p$ must be in $\cm n \cap \q[G]$. If $q$ is supported on a dependent set, then $q\in \cm n$ would imply that one of $q_1$ or $q_2$ is in $\cm n$, but none of the two are. Therefore $p\in \cm n \cap \q[G]$ in any case.

Steps 5--6: There are now two possibilities for $p$: either it is supported on $G$, in which case we return it, or it is supported on a proper subset of $G$. If its support is a proper subset of $G$, we can in principle return any polynomial $q p$ such that $\supp{q p}=G$. Recall that the resultant is multiplicative (\cref{prop:basicPropResultants}), hence in a subsequent invocation of the algorithm, in the computation of $\res{qp}{f}{y}=\res{q}{f}{y}\res{p}{f}{y}$ for some $f$ and $y$ we can keep the factor $\res{q}{f}{y}$ unevaluated. An alternative would be to modify the resultant tree $T_C$ by replacing $G$ with the graph $G_p$ given by the support of $p$ (as defined in \cref{sec:circuitsCMideal}). However, in our presentation we keep the resultant tree fixed throughout and choose $q$ to simply be the product $\Pi_{x\in G\setminus \supp p}x$ of all indeterminates in $G\setminus \supp p$. 

In our experiments we are yet to encounter an example in which an irreducible factor supported on a dependent set that is a proper subset of $G$ appears. We leave as an open problem to find an example, or prove that it can not occur.

\begin{problem}
		Consider an intermediate node $G$ in a combinatorial resultant tree and let $r=\res{f}{g}{x_e}$ be the resultant supported on $G$ with respect to the polynomials supported on the child nodes of $G$, as in \Cref{alg:resTreeExtAlg1}. Find examples where $r$ has exactly one irreducible factor supported on a dependent set, and such that it is properly contained in $G$, or prove that this never happens. 	
\end{problem}

Steps 7--12: If there is more than one irreducible factor supported on a dependent set, we store them in the list \textit{candidates} in some order. Factors are then tested for membership in $\cm n$ with an ideal membership test, in the order in which they are stored in the list \textit{candidates}. The first irreducible factor that passes the test is returned if its support is $G$, or it is completed to a polynomial supported on $G$ in the same way as described above and then returned.

We have not encountered examples in which more than one irreducible factor supported on a dependent set appeared, however this is most likely because we were only able to perform computations on graphs with up to 8 vertices.

\begin{problem}\label{problem:moreThanOneFactor}
	Consider an intermediate node $G$ in a combinatorial resultant tree and let $r=\res{f}{g}{x_e}$ be the resultant supported on $G$ with respect to the polynomials supported on the child nodes of $G$, as in \Cref{alg:resTreeExtAlg1}. Find examples where $r$ has more than one irreducible factor supported on a dependent set, or prove that this never happens.
\end{problem}

Since $G$ is not necessarily a circuit, the elimination ideal $\cm n \cap \q[G]$ is no longer necessarily principal, and we can no longer guarantee the existence of a unique irreducible factor $p$ of $r$ that is both supported on $G$ and in $\cm n$. We have not encountered this possibility in our experiments, and we leave it as an open question.

\begin{problem}
	If \cref{problem:moreThanOneFactor} has a positive answer, find examples with two or more irreducible factors supported on $G$, or prove that this never happens.
\end{problem}

\subparagraph{Refinements of \cref{alg:resTreeExtAlg1}.}
If at a node of $T_C$ we have $\res{v}{w}{e}=0$, we can attempt to replace $v$ or $w$ with other appropriate polynomials in $\cm n$. In particular we can attempt to recompute $v$ or $w$ by choosing a different polynomial from the list of candidates in line 9. This approach however might require recomputing $v$ and $w$ many times, and we can still not guarantee that $\res{v}{w}{x_e}$ would be non-zero. We leave as an open problem to find the conditions on $v$ and $w$ so that $\res{v}{w}{x_e}$ is not zero.

\begin{problem}Consider the case in which at an intermediate node of $T_C$ we have $\res{v}{w}{e}=0$. Is it always possible to recompute $v$ and $w$ with \cref{alg:resTreeExtAlg1} by choosing a different polynomial from list of candidates (line 9 of the algorithm) so that $\res{v}{w}{e}\neq 0$?
\end{problem}

Alternatively we can replace one or both branches of the resultant tree for $G$ (taken as the subtree of $T_C$ rooted at $G$) with a tree that would lead to a non-zero resultant at $G$. For that purpose it would be useful to have an algorithm that enumerates the resultant trees of a dependent graph. Such enumeration appears to be much more challenging than for CCR trees (\cref{prob:enumCCR}) and it is unclear that an efficient solution to the following problem can be obtained:

\begin{problem}
	Develop an algorithm for enumerating resultant trees of a dependent graph.
\end{problem} 

If the answer to \cref{problem:moreThanOneFactor} is positive, we have to decide which polynomial to output. In \cref{alg:resTreeExtAlg1} the first irreducible factor with dependent support that passes the ideal membership test is chosen and returned (possibly padded by the indeterminates in $G\setminus\supp p$). However, it may be the case that the first irreducible factor that passes the ideal membership test is not the best choice if what we have in mind is the goal of simplifying the resultant computation when this algorithm is invoked on the parent of $G$. For example, relative to the remaining factors that pass the ideal membership test, the first factor that passed the test could have a very large degree in the indeterminate that is to be eliminated in the subsequent invocation of the algorithm, which, as a consequence, would lead to a very large dimension of the Sylvester determinant.

We propose the following decision criteria in the case when $r$ has multiple irreducible factors $\{p_1,\cdots,p_k\}$ in $\cm n$. From the set $\{p_1,\cdots,p_k\}$ choose the polynomial:
\begin{enumerate}
 \item[i)] with the least degree in the indeterminate to be eliminated when \cref{alg:resTreeExtAlg1} is invoked on the parent of $G$.
 \item[ii)]If there is more than one such choice, we choose the one with the least homogeneous degree.
 \item[iii)] If there still is more than one choice, we choose the first one with the least number of monomials.
\end{enumerate} 
Criterion (\emph{i}) ensures that when the algorithm is invoked on the parent of $G$, the dimension of the Sylvester determinant will be the least possible; criterion (\emph{ii}) ensures that the resultant will be of least possible homogeneous degree (\cref{prop:resultantHomogeneous}), while criterion (\emph{iii}) minimizes the total number of monomials that appear as entries in the Sylvester determinant. 

This choice of decision criteria may not be the best possible, and we leave as an open problem to formulate other decision criteria.

\begin{problem}
	If \cref{problem:moreThanOneFactor} has a positive answer, establish criteria for deciding which polynomial to return as output.
\end{problem}
\section{Experiments.}
\label{sec:experiments}
In this section we discuss our experimental work, carried out with the algorithms presented in this paper, that led to effective computations of all circuit polynomials in $\cm{6}$. \cref{tbl:circPoly} summarizes the results. To the best of our knowledge, except for the circuit polynomial of $K_4$, these polynomials have not been computed before. Each example of a circuit polynomial is presented up to relabelling of vertices. All the circuit polynomials computed in this section are available at the GitHub repository \cite{malic:streinu:GitHubRepo}. For comparison purposes, we also include some preliminary calculations done or attempted with \grobner{} basis methods. 

\begin{table}[ht]
\centering
\caption{Results: all circuit polynomials on $n\leq 6$ vertices, two circuit polynomials on $n=7$ vertices and two circuit polynomials on $n=8$ vertices. The method \grobner{} is the computation of a \grobner{} basis of ideals generated by two circuit polynomials, as explained in \cref{subsec:grobner}. 	The method Resultant A9.1 is \cref{alg:resTreeAlg1}, and the method Resultant A12.1 is \cref{alg:resTreeExtAlg1}.}
\label{tbl:circPoly}
\begin{tabularx}{\textheight}
	{|c|c|c|>{\centering}m{1.65cm}|c|>{\centering}m{0.95cm}|}
	\cline{1-6}
	
	$n$ & 
	Circuit & Method & 
	Comp.\ time (seconds) &
	No.\ terms & 
	Hom.\ degree \cr	
	
	\cline{1-6}	
	\multirow{2}*{4} & 
	\multirow{2}*{$K_4$} & 
	\multirow{2}*{Determinant} & 
	\multirow{2}*{0.0008} & 
	\multirow{2}*{22} & 
	\multirow{2}*{3} \cr	
	& & & & & 
	\cr	
		
	\cline{1-6}			
	\multirow{2}*{5} &  
	\multirow{2}*{Wheel on 4 vertices} & 
	\grobner{} & 
	0.02 &
	\multirow{2}*{843} & 
	\multirow{2}*{8} \cr	
	& & Resultant A9.1 & 0.013 & & 
	\cr	
			
	\cline{1-6}			
	\multirow{2}*{6} & 
	\multirow{2}*{2D double banana} & 
	\grobner{} & 
	0.164 & 
	\multirow{2}*{1 752} & 
	\multirow{2}*{8} \cr	
	& & Resultant A9.1 & 0.029 & & 
	\cr	
		
	\cline{1-6}			
	\multirow{2}*{6} & 
	\multirow{2}*{Wheel on 5 vertices} & 
	\grobner{} & 
	10 857 &
	\multirow{2}*{273 123} & 
	\multirow{2}*{20} \cr	
	& & Resultant A9.1 & 7.07 & & 
	\cr	
	
	\cline{1-6}		
	\multirow{2}*{6} & 
	\multirow{2}*{Desargues-plus-one} & 
	\grobner{} & 454 753 & 	
	\multirow{2}*{658 175} & 
	\multirow{2}*{20} \cr	
	& & Resultant A9.1 & 14.62 & & 
	\cr	
		
	\cline{1-6}					
	\multirow{2}*{6} & 
	\multirow{2}*{$K_{3,3}$-plus-one} & 
	\multirow{2}*{Resultant A12.1} & 
	\multirow{2}*{979.42} & 
	\multirow{2}*{1 018 050} & 
	\multirow{2}*{18} \cr	
	& & & & & 
	\cr	
		
	\cline{1-6}			
	\multirow{2}*{7} & 
	\multirow{2}*{2D double banana $\oplus_{16}$ $K_4^{1567}$} & 
	\multirow{2}*{Resultant A9.1} & 
	\multirow{2}*{38.14} & \multirow{2}*{1 053 933} & 
	\multirow{2}*{20} \cr	
	& & & & & 
	\cr	
			
	\cline{1-6}		
	\multirow{2}*{7} & 
	\multirow{2}*{2D double banana $\oplus_{56}$ $K_4^{4567}$} & 
	\multirow{2}*{Resultant A9.1} & 
	\multirow{2}*{89.86} & 
	\multirow{2}*{2 579 050} & 
	\multirow{2}*{20} \cr
	& & & & &
	\cr
	
	\cline{1-6}		
	\multirow{2}*{8} & 
	\multirow{2}*{2D double banana $\oplus_{45}$ $K_4^{4578}$} & 
	\multirow{2}*{Resultant A9.1} & 
	\multirow{2}*{109.8} & 
	\multirow{2}*{3 413 204} & 
	\multirow{2}*{20} \cr
	& & & & &
	\cr
	
	\cline{1-6}		
	\multirow{2}*{8} & 
	\multirow{2}*{2D double banana $\oplus_{56}$ $K_4^{5678}$} & 
	\multirow{2}*{Resultant A9.1} & 
	\multirow{2}*{302.47} & 
	\multirow{2}*{9 223 437} & 
	\multirow{2}*{20} \cr
	& & & & &
	\cr		
		
	\cline{1-6}
	
\end{tabularx}
\end{table}

\subparagraph{The $K_4$ circuit.} The only circuit polynomial that is directly obtainable as a generator of $\cm{n}$ for any $n\geq 4$, and does not require \grobner{} basis methods or resultant computations, is the circuit polynomial of a $K_4$ graph (possibly relabeled). This polynomial has 22 terms, homogeneous degree 3, and is of degree 2 in any of its variables.
\subsection{Computation of circuit polynomials via \grobner{} bases.}\label{subsec:grobner}
In principle a circuit polynomial $p\in\cm n$ can be computed by computing a \grobner{} basis $\mathcal G_{\cm n}$ for $\cm n$ with respect to an \emph{elimination order} on the set $\{x_{i,j}\mid 1\leq i<j\leq n\}$ in which all the indeterminates in the complement of $\supp p$ are greater than all the indeterminates in $\supp p$.

Given $\mathcal G_{\cm n}$ it is straightforward to determine a \grobner{} basis $\mathcal G_{\ideal p}$ for the ideal $\ideal p=\cm{n}\cap\q[\supp p]$: it is the intersection $\mathcal G_{\ideal p}=\mathcal G_{\cm n}\cap \q[\supp p]$.
Therefore, the only element in $\mathcal G_{\cm n}$ supported on $\supp p$ is precisely $p$, possibly multiplied by a non-zero scalar.

\subparagraph{\grobner{} basis for $\cm{n}$ with respect to an elimination order.}
We were able to compute a \grobner{} basis with respect to an elimination order only for $n=5$. Already for $n=6$ we did not succeed in carrying out such a computation, within a reasonable amount of time, neither in Mathematica nor in Macaulay2.

\subparagraph{\grobner{} basis of ideals generated by two circuit polynomials.} For comparison purposes, we describe a second method that we experimented with. This one takes into account the combinatorial structure presented in \cref{sec:combRes} but works with \grobner{} bases rather than resultants. Let  $A$, $B$ and $C$ be circuits such that $C=\cres{A}{B}{e}$, where $e$ is a common edge of $A$ and $B$. To compute the circuit polynomial $p_C$ of the circuit $C$, it is sufficient to calculate only a \grobner{} basis $\beta$ of the ideal $\ideal{p_A,p_B}$ generated by the circuit polynomials of $A$ and $B$, with respect to an elimination order in which the indeterminates in $(A\cup B)\setminus C$ are eliminated. This follows from 
$\ideal{p_A,p_B}\cap \q[C]\subseteq \cm{n}\cap\q[C]=\ideal{p_C}$, where
if $\ideal{p_A,p_B}$ is prime, then the \grobner{} basis $\beta$ will be exactly equal to $\beta=\{p_C\}$. Otherwise, a factorization and a subsequent ideal membership test for the factors supported on $C$ of each polynomial in $\beta$ will be required.

With this method we were able to compute all the circuit polynomials of circuits on 6 vertices except the $K_{3,3}$-plus-one circuit. It took us $0.164$ seconds to compute the 2D double banana, a bit over 3 hours to compute the wheel on 5 vertices, and 126 hours to compute the Desargues-plus-one circuit polynomial (see \cref{tbl:circPoly}).

\subsection{Computation of circuit polynomials with resultants.}\label{sec:compCircPoly}
We demonstrate now the effectiveness of our algorithm by computing all the circuit polynomials on up to $6$ vertices. They are supported on five types of graphs: a $4$-wheel $W_4$ (on $4$ cycle vertices with a $5$th vertex at the center), a $5$-wheel, a 2D ``double banana'' obtained as a $2$-sum of two $K_4$ graphs, the Desargues-plus-one graph, and the $K_{3,3}$-plus-one graph. They are shown in \cref{fig:graphW5K4} and \cref{fig:6circuits}. We are recording only the computation of the root of a particular resultant tree. We chose resultant trees that were most efficient for each computation. The relevant parameters of each circuit (size, homogeneous degree) and comparative timings for its computation are shown in \cref{tbl:circPoly}. Two more circuits on $7$ vertices, as well as two on $8$ vertices, were also computed using $2$-sum resultants, which give the best resultant trees. 

\subparagraph{Wheel on 4 vertices.} This circuit was very fast to compute. It has (up to relabeling) exactly one resultant tree with two $K_4$ leaves and a single application of a resultant, which produces an irreducible polynomial. Irreducibility was verified with Mathematica. This polynomial has 843 terms, its homogeneous degree is 8, and it is of degree 4 in each of its variables.

\subparagraph{The ``2D double banana''.} Recall from \cref{fig:exampleCircuit} that the 2D double banana can be obtained as the combinatorial resultant of two $K_4$'s or of two $4$-wheels. The first tree led to a very fast calculation, and the resultant produced an irreducible polynomial. This polynomial has 1752 terms, its homogeneous degree is 8, and it is of degree 4 in each of its variables.

However, on our computers we did not succeed in calculating the circuit polynomial using the second resultant tree, or as a \grobner{} basis of an ideal generated by the circuit polynomials of the two $4$-wheels, with respect to an elimination order. Here is a possible explanation. Recall that \cref{prop:resultantHomogeneous} allows us to predict the homogeneous degree of the resultant of two homogeneous polynomials. In particular, the homogeneous degree of the resultant for two $4$-wheels has homogeneous degree 48, whereas the resultant of the circuit polynomials of two $K_4$ graphs has homogeneous degree 8. Hence, we could see immediately that we should discard the former, as in the latter case we obtain a much simpler polynomial. This example inspires the following conjecture:

\begin{problem}
	Prove that a $2$-sum is more efficient than any other type of combinatorial resultant, in computing a circuit polynomial as a resultant of two circuits.
\end{problem}

\subparagraph{Wheel on 5 vertices.} We computed this circuit from a $4$-wheel and a $K_4$, and obtained directly an irreducible polynomial. Irreduciblity was verified in Mathematica. This polynomial has 273123 terms, its homogeneous degree is 20, and it is of degree 8 in each of its variables.

\subparagraph{The Desargues-plus-one circuit.}\label{example:desarguesPlus1} The rigidity theory literature refers to the graph $D$ with edges $\{12,14,15,23,26,34,36,45,56\}$ as the Desargues graph, due to its similarity to the incidence structure arising from the classical Desargues configuration of lines. The graph $D$ can be completed to a circuit (what we call Desargues-plus-one) by adjoining to it exactly one of the missing edges, with all choices of missing edge resulting in isomorphic graphs. The circuit can be obtained as a combinatorial resultant of a $4$-wheel (with cycle $1,2,3,4$ and $5$ at the center) and a $K_4$ on vertices $2,3,5,6$, by eliminating the edge $35$. Using the previously computed $4$-wheel circuit polynomial, the resultant calculation took under $15$ seconds - which is impressive when compared to the $5$ days and $6$ hours taken by the \grobner \ basis method. The resultant polynomial is irreducible, has homogeneous degree 20, it is of degree 12 in the variable $x_{2,5}$ and of degree $8$ in the remaining variables.

\subsection{The $K_{3,3}$-plus-one circuit.}
\label{example:K33plus1}
The complete bipartite graph $K_{3,3}$ on the vertex partition $\{1,4,5\}\cup\{2,3,6\}$ is minimally rigid. It can be completed to a circuit by adding to it exactly one of the missing edges. All these choices result in isomorphic graphs. 

We were not able to compute its circuit polynomial with \cref{alg:resTreeAlg1} or \cref{alg:resTreeAlg2}. All attempts completely exhausted all computational resources at the resultant step. However, we succeeded with the approach described in \cref{sec:algebraicNew}. This method allowed us to carry out the full computation, described step-by-step in \cref{subsec:computingK33}. The irreducible circuit polynomial has 1018050 terms, homogeneous degree 18, and is of degree 8 in each variable.

The properties of this polynomial imply an interesting fact, which is relevant for a better understanding of \cref{alg:resultant}: it provides, indirectly, the first example of a circuit polynomial on which the last resultant step {\em in any of the possible combinatorial resultant trees} would have to produce a polynomial which is {\em never} irreducible. Hence a factorization and an inspection of factors for membership in the Cayley-Menger ideal will be necessary at the root, either by inspecting the supports or by performing a test of membership in the Cayley-Menger ideal. The proof is instructive and we include it here.

\begin{proposition}
	\label{prop:k33notObtainable} 
	Let $A$ and $B$ be rigidity circuits on 6 or less vertices such that neither is the $K_{3,3}$-plus-one circuit and such that $\cres{A}{B}{e}$ is the $K_{3,3}$-plus-one circuit for some common edge $e$. If $p_A$ and $p_B$ are the circuit polynomials for $A$ and $B$, then $\res{p_A}{p_B}{x_e}$ is reducible.
\end{proposition}

\begin{proof}
	Let $h_A$ and $h_B$ be the homogeneous degrees, and let $d_A$ and $d_B$ be the degrees in $x_e$ of $p_A$ and $p_B$, respectively. By Proposition \ref{prop:resultantHomogeneous}, the homogeneous degree of $\res{p_A}{p_B}{x_e}$ is $h_Ad_B+h_Bd_A-d_Ad_B$, so if $\res{p_A}{p_B}{x_e}=c\cdot p_{K_{3,3}\textnormal{-plus-one}}$ for some $c\in\q$, then $h_Ad_B+h_Bd_A-d_Ad_B=18$. However, by \cref{sec:compCircPoly} the values of $(h_A,d_A)$ and $(h_B,d_B)$ can only be in the set $\{(3,2), (8,4), (20,8), (20,12)\}$ and no choice corresponds to $h_Ad_B+h_Bd_A-d_Ad_B=18$.
\end{proof}

As a final observation, we note that the $K_{3,3}$-plus-one graph can be obtained as the combinatorial resultant of two $4$-wheels: one wheel on $1, 2, 3, 4$ with $5$ in the center, and the other on $1, 3, 4, 6$ with $5$ in the center, on the elimination edge $15$. Since the circuit polynomial for a $4$-wheel has homogeneous degree $8$ and both have degree $4$ in $x_{1,5}$, it follows from Proposition \ref{prop:resultantHomogeneous} that their resultant has homogeneous degree $48$. Hence the circuit polynomial for $K_{3,3}$-plus-one appears as a factor in this resultant, with multiplicity not greater than 2. Unfortunately, we were not able to compute the resultant of these two $4$-wheels before our machines ran out of memory. We have attempted to brute-force the computation by first computing the resultant of two general degree 4 polynomials in the variable $x$, which has 219 monomials. We then substituted the coefficients (w.r.t.\ $x$) of the circuit polynomials for the two wheels into the 219 monomials. We then proceeded to expand them, and save each of the 219 expansions to disk. This took approx.\ 5 days of computing on a HPC and in total occupies approx.\ 1.7TB of data (stored in Mathematica's uncompressed .mx format). However, adding together the 219 expanded monomials failed and we did not pursue this direction further. We estimate that a powerful enough machine with at least 2TB of RAM could be forced to compute the resultant of two wheels on 4 vertices.

\subsection{Example: the $K_{3,3}$-plus-one circuit polynomial.}
\label{subsec:computingK33}
At the leaves of the tree we are using irreducible polynomials from among the generators of the Cayley-Menger ideal. The polynomials corresponding to the nodes on the leftmost path from a leaf to the root are referred to, below, as $D_1$ (leftmost leaf), $D_2$ and $D_3$ (for the next two internal nodes with dependent graphs on them) and $C$ for the circuit polynomial at the root, see \cref{fig:resTreeK33}. The leaves on the right are three $K_4$ circuit polynomials: $C_1$ supported on vertices $\{ 1,2,3,5 \}$, $C_2$ supported on $\{ 1,3,4,6 \}$ and $C_3$ supported on $\{ 1,4,5,6 \}$. For the polynomial $D_1$ at the bottom leftmost leaf, supported by a dependent $K_5$ graph, we have used the generator:

\begin{align*}
	&x_{15} x_{34}^2-x_{16} x_{34}^2-x_{56} x_{34}^2-x_{14} x_{35} x_{34}+x_{16} x_{35} x_{34}+x_{14} x_{36} x_{34}-2 x_{15} x_{36} x_{34}\\
	&+x_{16} x_{36} x_{34}-x_{13} x_{45} x_{34}+x_{16} x_{45} x_{34}+x_{36} x_{45} x_{34}+x_{13} x_{46} x_{34}-2 x_{15} x_{46} x_{34}\\
	&+x_{16} x_{46} x_{34}+x_{35} x_{46} x_{34}-2 x_{36} x_{46} x_{34}+x_{13} x_{56} x_{34}+x_{14} x_{56} x_{34}-2 x_{16} x_{56} x_{34}\\
	&+x_{36} x_{56} x_{34}+x_{46} x_{56} x_{34}-x_{14} x_{36}^2+x_{15} x_{36}^2-x_{13} x_{46}^2+x_{15} x_{46}^2-x_{35} x_{46}^2+x_{14} x_{35} x_{36}\\
	&-x_{16} x_{35} x_{36}-x_{36}^2 x_{45}+x_{13} x_{36} x_{45}-2 x_{14} x_{36} x_{45}+x_{16} x_{36} x_{45}-2 x_{13} x_{35} x_{46}+x_{14} x_{35} x_{46}\\
	&+x_{16} x_{35} x_{46}+x_{13} x_{36} x_{46}+x_{14} x_{36} x_{46}-2 x_{15} x_{36} x_{46}+x_{35} x_{36} x_{46}+x_{13} x_{45} x_{46}\\
	&-x_{16} x_{45} x_{46}+x_{36} x_{45} x_{46}-x_{13} x_{36} x_{56}+x_{14} x_{36} x_{56}+x_{13} x_{46} x_{56}-x_{14} x_{46} x_{56}
\end{align*}

The set of generators supported on $K_5$ contains more than this polynomial. There are two other available choices, of homogeneous degrees 4 or 5, which, in addition, can have quadratic degree in the elimination indeterminate $x_{35}$. The choice of this particular generator was done so as to minimize the complexity of (the computation of) the resultant: its homogeneous degree $3$ and degree $1$ in the elimination variable $x_{35}$ are both minimal among the three available options. 

At the internal nodes of the tree we compute, using resultants and factorization, irreducible polynomials in the ideal whose support matches the dependent graphs of the combinatorial tree, as follows. 

The resultant $p_{D_2}=\res{p_{D_1}}{p_{C_1}}{x_{35}}$ is an irreducible polynomial supported on the graph $D_2$ in \cref{fig:resTreeK33}. This graph contains the final result $K_{3,3}$-plus-one as a subgraph, as well as two additional edges, which will have to be eliminated to obtain the final result. {\em Thus the resultant tree is not strictly increasing with respect to the set of vertices along a path, as was the case in \cref{sec:resultantTree}.} However, when the set of vertices remains constant (as demonstrated with this example), the dependent graphs on the path towards the root are {\em strictly decreasing with respect to the edge set.}

The resultant $p_{D_3}=\res{p_{D_2}}{p_{C_2}}{x_{13}}$ is a reducible polynomial with 222108 terms and two non-constant irreducible factors. Only one of the factors is supported on $D_3$, with the other factor being supported on a minimally rigid (hence independent) graph. Thus this factor, the only one which can be in the CM ideal (and it must be, by primality considerations), is chosen as the new polynomial $p_{D_3}$  with which we continue the computation.

The final step to obtain $C$ is to eliminate the edge $46$ from $D_3$ by a combinatorial resultant with $C_3$. The corresponding resultant polynomial $p_{C}$ is a reducible polynomial with 15 197 960 terms and three irreducible factors. As in the previous step, the analysis of the supports of the irreducible factors shows that only one factor is supported on the $K_{3,3}$-plus-one circuit, while the other two factors are supported on minimally rigid graphs. This unique irreducible factor is the desired circuit polynomial for the $K_{3,3}$-plus-one circuit.

The computational time on an 2019 iMac with 6 CPU cores at 3.7 GHz in Mathematica v13, including factorizations to irreducible components was 979.42 seconds. The computation and factorization of the final resultant step took up most of the computational time (562.5, resp.\ 394.9 seconds).
\section{Concluding Remarks.}
\label{sec:concludingRemarks}

In this paper we introduced the combinatorial resultant operation, analogous to the  classical resultant of polynomials.  We offer here some final comments and suggestions for further research.

\subparagraph{Irreducibility test.} 
Our methods still have several computational drawbacks, in that they require irreducibility checks, with a possible further factorization and an ideal membership test for those factors that have the support of a circuit.

Ideally we would like to detect combinatorially when a resultant of two circuit polynomials that has the support of a circuit will be irreducible. The absolute irreducibility test of Gao \cite{Gao} which states that a polynomial is absolutely irreducible if and only if its Newton polytope is integrally indecomposable, in conjunction with the description of the Newton polytope of the resultant of two polynomials by Gelfand, Kapranov and Zelevinsky \cite{GelfandKapranovZelevinskyNewton, GelfandKapranovZelevinsky} gives a combinatorial criterion for absolute irreducibility, but not for irreducibility over $\mathbb Q$. However, not every circuit polynomial is absolutely irreducible, for example the circuit polynomial of a wheel on 4 vertices is irreducible over $\mathbb Q$ but not absolutely irreducible. 

\subparagraph{What we observed in practice.} It is worth noticing that whenever in our computations we had to decide which factor of a resultant belonged to $\cm n$, we never had to perform an ideal membership test. It was always sufficient to inspect only the supports of the irreducible factors of the resultant. In all cases where the calculation succeeded, all but one irreducible factor were supported on Laman graphs, and one factor was supported on a dependent set. It seems unlikely that this is the general case, and it would be of interest to determine under which conditions does the resultant have exactly one factor (up to multiplicity) supported on a dependent set in $\amat{\cm n}$.

\subparagraph{Open problems.} We conclude the paper with a few more open problems concerning the algebraic and geometric structure of the resultant of two circuit polynomials.

\begin{problem}
	Let  $A$, $B$ and $C$ be circuits such that $C=\cres{A}{B}{e}$. Let $p_C$, $p_{A}$ and $p_{B}$ be the corresponding circuit polynomials. Under which conditions is it the case that $\res{p_{A}}{p_{B}}{x_e}$ is of the form $\alpha\cdot p_C^m$ for $m\geq 1$ with $\alpha\in \q$?
\end{problem}

\begin{problem}
More generally, for two polynomials $p,q\in \cm n$ with $x_e\in\supp p\cap\supp q$, under which conditions has the resultant exactly one irreducible factor supported on a dependent set in $\amat{\cm n}$?
\end{problem}

\begin{problem}
Generalize \cref{prop:k33notObtainable} to the question of whether reducibility of $\res{p_A}{p_B}{x_e}$ can be inferred from graph-theoretic data (circuits $C$, $A$, $B$ and edge $e$ such that $C=\cres{A}{B}{e}$).
\end{problem}

This question appears to be very challenging. The answer depends heavily on the specific polynomials $p_A$, $p_B$ and the variable $x_e$ and pertains to the relationship between (affine) varieties related to $r=\res{p_A}{p_B}{x_e}$, $p_A$ and $p_B$. 
Let $R=\mathbb C[C]$ be a polynomial ring, $p_A,p_B\in R[x_e]$ and let $I_{x_e}$ denote the elimination ideal $\ideal{p_A,p_B}\cap R$. Let $l_A$ (resp.\ $l_B$) be the leading coefficient of $p_A$ (resp.\ $p_B$) with respect to $x_e$. Then by the Extension Theorem \cite[Theorem 8 in \S6 of Ch.\ 3]{CoxLittleOshea} and the Closure Theorem \cite[Theorem 4 in \S4 of Ch.\ 4]{CoxLittleOshea} we have the following equality of (affine) varieties: $V(r)=V(l_A,l_B)\cup V(I_{x_e})$. Furthermore, if $r$ factors as $q\cdot p_C^k$ for some positive integer $k$, then $V(r)=V(q)\cup V(p_C)$. Ideally we would want $V(r)=V(p_C)=V(I_{x_e})$ but in general $V(p_C)$ is only contained in $V(I_{x_e})$ and $V(r)$. Hence the structure of $V(r)$, in particular its irreducibility, depends on {\em algebraic} data $V(l_A,l_B)$ and $V(I_{x_e})$, whose relationship to the combinatorial, graph-theoretical data is yet to be found.

\medskip

Further interesting questions pertain to parameters of circuit polynomials such as the degree in a single variable or the number of monomials. The first one, the degree with respect to a single variable $x_e$ in the support of a circuit polynomial, is related to  the literature  on the number of embeddings of Laman graphs, where the best known upper bound is $3.77^n$ \cite{bartzos:etAl:bound} for $n$ vertices. Bounds on the degree of an individual indeterminate of a $3$-connected circuit polynomial can be infered from here, while for the $2$-connected ones their decomposition into $3$-connected components is needed. On the other hand, we are not aware of any such bounds on the number of monomial terms of circuit polynomials, but have observed that their number quickly becomes large, as shown in the Table \ref{tbl:circPoly}.

\begin{problem}
How big do circuit polynomials get, i.e.\ what are upper and lower bounds on the number of monomial terms relative to the number of vertices $n$?
\end{problem}

\begin{problem}
		When working with an extended collection of generators, not all of them circuits (such as those from \cref{sec:generatorsCM}), decide if a given circuit has a combinatorial resultant tree with at least one non-$K_4$ leaf from the given generators.
\end{problem}

\section*{Acknowledgments}
We would like to thank the anonymous reviewers for their comments, which have helped improve the presentation, and for suggesting references that have increased the scope of the paper.

\bibliographystyle{siamplain}
\bibliography{references}

\end{document}